\documentclass[11pt]{amsart}

\usepackage{amssymb}
\usepackage{mathrsfs}
\usepackage[top=1in,bottom=1in,left=1.25in,right=1.25in]{geometry}
\usepackage{enumerate}
\usepackage[numeric, abbrev]{amsrefs} % Reference
\usepackage{hyperref}

\usepackage[all]{xy}

\newtheorem{theorem}{Theorem}[section]
\newtheorem{definition}[theorem]{Definition}
\newtheorem{remark}[theorem]{Remark}
\newtheorem{example}[theorem]{Example}
\newtheorem{proposition}[theorem]{Proposition}

\newtheorem{corollary}[theorem]{Corollary}
\newtheorem{lemma}[theorem]{Lemma}

\DeclareMathOperator{\Lie}{Lie}
\DeclareMathOperator{\End}{End}
\DeclareMathOperator{\Res}{Res} % Weil restriction
\DeclareMathOperator{\GL}{GL}
\DeclareMathOperator{\Spec}{Spec}

\DeclareMathOperator{\Hom}{Hom}
\DeclareMathOperator{\Ker}{Ker}
\DeclareMathOperator{\Image}{Im}
\DeclareMathOperator{\Isom}{\underline{Isom}}
\DeclareMathOperator{\Aut}{Aut}
\DeclareMathOperator{\Mat}{M}
\DeclareMathOperator{\dR}{dR}
\DeclareMathOperator{\Ad}{Ad}
\DeclareMathOperator{\naive}{naive}
\DeclareMathOperator{\loc}{loc}
\DeclareMathOperator{\AV}{AV}
\DeclareMathOperator{\GSp}{GSp}
\DeclareMathOperator{\Adm}{Adm}
\DeclareMathOperator{\rdt}{rdt}
\DeclareMathOperator{\Gal}{Gal}
\DeclareMathOperator{\ve}{\varepsilon}

\DeclareMathOperator{\calG}{\mathcal{G}}
\DeclareMathOperator{\calH}{\mathcal{H}}

\DeclareMathOperator{\calM}{\mathcal{M}}
\DeclareMathOperator{\calO}{\mathcal{O}}
\DeclareMathOperator{\calV}{\mathcal{V}}

\DeclareMathOperator{\bbA}{\mathbb{A}}
\DeclareMathOperator{\bbF}{\mathbb{F}}
\DeclareMathOperator{\bbG}{\mathbb{G}}
\DeclareMathOperator{\bbI}{\mathbb{I}}
\DeclareMathOperator{\bbM}{\mathbb{M}}
\DeclareMathOperator{\bbQ}{\mathbb{Q}}
\DeclareMathOperator{\bbZ}{\mathbb{Z}}

\DeclareMathOperator{\scrA}{\mathscr{A}}
\DeclareMathOperator{\scrF}{\mathscr{F}}

\DeclareMathOperator{\scrL}{\mathscr{L}}

\DeclareMathOperator{\Sht}{Sht}
\DeclareMathOperator{\pf}{pf}
\DeclareMathOperator{\rank}{rank}
\DeclareMathOperator{\ur}{ur}
\DeclareMathOperator{\Hdg}{Hdg}
\DeclareMathOperator{\Fil}{Fil}
\DeclareMathOperator{\PR}{PR}
\DeclareMathOperator{\Newt}{Newt}
\DeclareMathOperator{\Frob}{Frob}
\DeclareMathOperator{\id}{id}
\DeclareMathOperator{\Ver}{Ver}

\def\ra{\rightarrow}
\def\lra{\longrightarrow}
\def\ov{\overline}
\def\ul{\underline}
\def\wh{\widehat}
\def\wt{\widetilde}
\def\st{\stackrel}
\def\tr{\textrm}
\def\ext{\mathrm{ext}}
\def\Sat{\mathrm{Sat}}

\def\loc{\mathrm{loc}}
\def\spl{\mathrm{spl}}
\def\tor{\mathrm{tor}}
\def\Fl{\mathscr{F}\ell}
\def\min{\mathrm{min}}
\def\Rep{\mathrm{Rep}}
\def\can{\mathrm{can}}
\def\sub{\mathrm{sub}}
\def\V{\mathcal{V}}
\def\H{\mathcal{H}}
\def\Q{\mathbb{Q}}
\def\R{\mathbb{R}}
\def\Z{\mathbb{Z}}
\def\F{\mathbb{F}}
\def\G{\mathcal{G}}
\def\M{\mathsf{M}}

\newcommand{\Pair}[2]{\left\langle #1, #2 \right\rangle}
\newcommand{\Zip}[2]{{#1}\text{-}\mathrm{Zip}^{#2}} % stack of G-zip
\newcommand{\Unram}[1]{#1^{\rm ur}}

\newcommand{\Spl}[1]{#1^\spl}

\begin{document}

\title[$F$-zips with additional structure on splitting models]{$F$-zips with additional structure on splitting models of Shimura varieties}
\author{Xu Shen}
\author{Yuqiang Zheng}
\date{}

\address{Morningside Center of Mathematics\\
		Academy of Mathematics and Systems Science\\
	Chinese Academy of Sciences\\
	No. 55, Zhongguancun East Road\\
	Beijing 100190, China}
\address{University of Chinese Academy of Sciences, Beijing 100149}
\email{shen@math.ac.cn}

\address{Academy of Mathematics and Systems Science, Chinese Academy of Sciences, No. 55, Zhongguancun East Road\\
	Beijing 100190, China}
\email{zhengyq@amss.ac.cn}

\renewcommand\thefootnote{}

\renewcommand{\thefootnote}{\arabic{footnote}}

\begin{abstract}
	We construct universal $G$-zips on good reductions of the Pappas-Rapoport splitting models for PEL-type Shimura varieties. We study the induced Ekedahl-Oort stratification, which sheds new light on the mod $p$ geometry of splitting models. Building on the work of Lan on arithmetic compactifications of splitting models, we further extend these constructions to smooth toroidal compactifications.  Combined with the work of Goldring-Koskivirta on group theoretical Hasse invariants, we get an application to Galois representations associated to torsion classes in coherent cohomology in the ramified setting.
\end{abstract}

%\begin{abstract}
%Nous
%\end{abstract}

\maketitle
\setcounter{tocdepth}{1}
\tableofcontents

\section{Introduction}
This paper deals with the mod $p$ geometry and arithmetic of some Shimura varieties at \emph{ramified} places. More precisely, we study the reduction modulo $p$ of the splitting models for PEL-type Shimura varieties constructed by Pappas-Rapoport in \cite{PappasRapoport2005}. For \emph{smooth} splitting models, we explore the mod $p$ Hodge structures on their special fibers by constructing universal $F$-zips with additional structure of fixed type (determined by the Hodge cocharacter). Then we derive some consequences on the geometry and coherent cohomology. In particular, we reprove the main results of Bijakowski-Hernandez \cite{BijakowskiHernandez2022} by different methods. We also clarify and generalize the work of Reduzzi-Xiao \cite{ReduzziXiao2017} in the ramified Hilbert case. \\
 
 PEL moduli spaces are central objects to study in arithmetic geometry and Langlands program.
 If the associated reductive group is unramified at $p$, Kottwitz constructed smooth integral PEL moduli spaces in \cite{Kottwitz1992}. In the general setting, Rapoport and Zink in \cite{RapoportZink1996} introduced similar and generalized integral PEL moduli spaces $\scrA^{\naive}$. These are called naive integral models, as shown by Pappas in \cite{Pappas1} that in some ramified unitary case, the moduli scheme $\scrA^{\naive}$ fails to be flat over $\calO_E$, the ring of integers of the local reflex field $E$.
To study integral models of PEL-type Shimura varieties in ramified case, in \cite{PappasRapoport2005} Pappas and Rapoport introduced the so-called splitting integral models $\scrA^\spl$. Roughly speaking, in the setting of \cite{PappasRapoport2005}, the associated $p$-adic reductive group $G$ has the form of a Weil restriction and the ramification mainly comes from restriction of scalars. After fixing the data, the model $\scrA^\spl$ is proposed as a relative moduli space over the naive integral model $\scrA^{\naive}$, the latter constructed in \cite{RapoportZink1996}. For any scheme $S$ over $\calO_E$, $\scrA^{\naive}(S)$ classifies abelian schemes with PEL structure $\ul{A}=(A,\lambda, \iota,\alpha)$ over $S$. Fix a sufficiently large field extension $F|E$. For a scheme $S$ over $\calO_F$,  $\scrA^\spl(S)$ classifies $(\ul{A}, \underline{\mathscr{F}_\bullet})$ where $\ul{A}\in\scrA^{\naive}(S)$ and $\underline{\mathscr{F}_\bullet}$ is a filtration on the cotangent bundle $\omega_{A/S}$, satisfying certain conditions related to the ramification data. Pappas and Rapoport proved that the model $\scrA^\spl$ over $\calO_F$ admits nice properties. In particular, there is a scheme $\mathbb{M}^\spl$ (called the splitting local model) over $\calO_F$, so that $\scrA^\spl$ and $\mathbb{M}^\spl$ sit in a \emph{local model diagram}. Moreover, by construction $\mathbb{M}^\spl$ can be realized as a \emph{twisted product} of unramified local models. There is a natural morphism $\scrA^\spl\ra \scrA^{\naive}$, which is the composition of the forgetful morphism $\scrA^\spl\ra \scrA^{\naive}_{\calO_F}$ over $\calO_F$ and the base change projection map $\scrA^{\naive}_{\calO_F}\ra \scrA^{\naive}$.  Pappas and Rapoport defined the canonical model $\scrA$ as the scheme theoretic image of this morphism. Then $\scrA$ is a  flat integral model over $\calO_E$. In case that the group $G$ is tamely ramified and $p\nmid |\pi_1(G_{der})|$, then up to Hasse principle the scheme $\scrA$ should coincide with the Kisin-Pappas integral models \cite{KisinPappas2018}. 
We refer to \cite{PappasRapoport2005} for more information and details on the general theory of splitting models (so far only available in the PEL type case), and to \cite{Pappas2,Pappas3,PappasRapoport2021} for some recent progress on the canonical models.\\

The most well known splitting models come from the example of Hilbert modular varieties. In this case, $\scrA=\scrA^{\naive}$ and the special fiber of $\scrA$ was studied previously by Deligne and Pappas in \cite{DP}.
Let $L$ be a totally real field of degree $g>1$ and $p$ a prime number. Let $\kappa|\F_p$ be a large enough finite field and $\mathcal{M}^{\tr{DP}}=\scrA\otimes\kappa$ the Deligne-Pappas moduli space over $\kappa$, which parametrizes abelian schemes with real multiplication given by $\calO_L$ together with polarization and level structure. If $p$ is unramified in $L$, this is a smooth scheme (a special case of the Kottwitz models \cite{Kottwitz1992}). Here we are mainly concerned with the case $p$ ramifies in $L$. Then $\mathcal{M}^{\tr{DP}}$  is only a normal scheme which is singular (cf. \cite{DP}). By contrast, the special fiber of $\scrA^\spl$ over $\kappa$, in this case denoted by $\mathcal{M}^{\tr{PR}}$,  is \emph{smooth} and the natural morphism \[\mathcal{M}^{\tr{PR}}\ra \mathcal{M}^{\tr{DP}}\] is a \emph{resolution of singularities}. In \cite{ReduzziXiao2017} Reduzzi and Xiao constructed $g$ partial Hasse invariants on $\mathcal{M}^{\tr{PR}}$ by carefully exploring the structure of Pappas-Rapoport filtrations. On the other hand, in the ramified case the number of partial Hasse invariants on $\mathcal{M}^{\tr{DP}}$ is strictly less than $g$ (see the introduction of \cite{ReduzziXiao2017} and the references therein). Reduzzi and Xiao applied these $g$ partial Hasse invariants on $\mathcal{M}^{\tr{PR}}$ to construct Galois pseudo-representations attached to torsion Hecke eigenclasses in the coherent cohomology. This shows a big advantage to work with splitting models. The space $\mathcal{M}^{\tr{PR}}$ and the partial Hasse invariants on it have been serving as a basic tool in the study of geometry and arithmetic of Hilbert modular varieties, for example, see the recent works of Sasaki \cite{Sasaki} and Diamond-Kassaei \cite{DK20}.\\

In the more general PEL setting, recently Bijakowski and Hernandez in \cite{BijakowskiHernandez2022} studied some aspects of the mod $p$ geometry of splitting models. More precisely, they proved that $\scrA^\spl$ with maximal level at $p$ is smooth under some conditions on the PEL datum. Roughly, these conditions are to ensure that at a $p$-adic place $v_i$, the group has the form $\Res_{F_i|\bbQ_p}H_i$ where $H_i$ is unramified over $F_i$, and the level at $p$ is hyperspecial for these $H_i$. In particular, this excludes the ramified unitary groups (labeled as type (AR) in \cite{BijakowskiHernandez2022}) as local factors. Bijakowski and Hernandez also proved that the $\mu$-ordinary locus (defined as the maximal Newton stratum) is open and dense in the special fiber $\scrA^\spl_0$ of $\scrA^\spl$. For this, they introduced a so-called Hodge stratification and proved that the maximal Hodge stratum (which contains the $\mu$-ordinary locus) is open and dense.\\

In this paper, we essentially work in the same setting as that in \cite{BijakowskiHernandez2022}. Our first key observation is that one can modify the local model diagram between $\scrA^\spl$ and $\mathbb{M}^\spl$ constructed in \cite{PappasRapoport2005} to a local model diagram between $\scrA^\spl$ and $\prod_{i,j.l}\mathbb{M}^\loc(\G_{i,j}^l,\mu_{i,j}^l)$, a product of unramified local models related to the PEL datum, cf. Proposition \ref{prop loc mod Gspl}.
This is not quite surprising, as we already mentioned above that the splitting local model $\mathbb{M}^\spl=\mathbb{M}^\spl(\G,\mu)$ was constructed in \cite{PappasRapoport2005} as a twisted product of unramified local models $\mathbb{M}^\loc(\G_{i,j}^l,\mu_{i,j}^l)$. Let $G$ be the reductive group defined by the PEL datum, which we assume to be connected (thus we exclude the type D case), and $\{\mu\}$ the attached geometric conjugacy class of Hodge cocharacters of $G$. In the following we often fix a suitable choice of $\mu$ in this conjugacy class.
Recall that $E|\Q_p$ is the field of definition of $\{\mu\}$ and $F|E$ is a large enough extension (so that $G$ splits over $F$).
The group scheme related to the modified local model diagram is $\G^\spl$, a reductive group over $\calO_F$, which is the reductive model of the split group $G_F$. On the other hand, let $\G$ be the parahoric model of $G$ over $\Z_p$ associated to the integral PEL datum.
The local model diagram of \cite{PappasRapoport2005} corresponds to a morphism of algebraic stacks \[\scrA^\spl\ra [\mathbb{M}^\spl(\G,\mu)/\G_{\calO_F}],\] which is more suited to study the canonical model $\scrA$ and its related canonical local model.
On the other hand, the local model diagram here corresponds to a morphism \[\scrA^\spl\ra [\mathbb{M}^\loc(\G^\spl,\mu)/\G^\spl],\] where $ \mathbb{M}^\loc(\G^\spl,\mu)=\prod_{i,j,l}\mathbb{M}^\loc(\G_{i,j}^l,\mu_{i,j}^l).$
If there is an $i$ such that the finite extension $F_i|\Q_p$ is ramified, then in general \[\G_{\calO_F}\neq \G^\spl.\] 
From the modified diagram here,
we can immediately deduce the smoothness of $\scrA^\spl$ if the local factors $\G_i$ satisfy the same condition as that in \cite{BijakowskiHernandez2022} (compare \cite{BijakowskiHernandez2022} Theorem 2.30 and Remark 2.31). Under this condition, the parahoric subgroup $K_p=\G(\Z_p)$ is in fact \emph{very special} in the sense of \cite{Zhu}. 

In fact, the above observation on the modified local model diagram leads us to go much further. Recall that the theory of $F$-zips was introduced and studied by Moonen-Wedhorn in \cite{MW} as a candidate of mod $p$ Hodge structure. This notion has been promoted and enlarged by Fontaine-Jannsen \cite{FJ}, Drinfeld \cite{Dri20}, and Bhatt-Lurie \cite{BL}. Here we work with the generalization of $F$-zips in another direction: the notion of $G$-zips for a reductive group $G$ over a finite field, cf. \cite{PinkWedhornZiegler2011, PinkWedhornZiegler2015}. Back to splitting models of PEL-type Shimura varieties, we will work with \emph{smooth} splitting models $\scrA^\spl$  from now on (in particular, this excludes the ramified unitary group case).
Let $\kappa$ be the residue field of $\calO_F$ and $\scrA^\spl_0$ the special fiber of $\scrA^\spl$ over $\kappa$. 
\begin{theorem}[Theorem \ref{thm EO}]\label{thm intro 1}
	\begin{enumerate}
		\item There is
a natural $\G^\spl_0$-zip of type $\mu$ over $\scrA^\spl_0$, where $\G^\spl_0$ is a reductive group over $\bbF_p$ constructed from the PEL datum such that $\G^\spl_{0,\kappa}=\G^\spl\otimes_{\calO_F}\kappa$. 
\item The induced map to the moduli stack of $\G^\spl_0$-zips of type $\mu$ 
\[\zeta: \scrA^\spl_0\ra \G^\spl_0\tr{-}\mathrm{Zip}^\mu_{\kappa}\] is smooth and surjective.
	\end{enumerate}
\end{theorem}
Let us first comment on the related reductive groups appearing here.
For $\scrA^\spl_0$, the group $\G^\spl_0$ replaces $\G_0^{\tr{rdt}}$, the maximal reductive quotient of the special fiber $\G_0$ of $\G$. The group $\G_0^{\tr{rdt}}$ is mainly related to the geometry of $\scrA_0$, the special fiber of the canonical model $\scrA$. Recall that there is a natural morphism $\scrA^\spl_0\ra \scrA_0$.
The difference between $\G^\spl_0$ and $\G_0^{\tr{rdt}}$ reflects the ramification data. Indeed,
at a $p$-adic place $v_i$, the local factors of the two groups admit the following description:
\[\G^\spl_{0,i}=\Res_{\kappa_i|\bbF_p}H_i^{e_i},\quad \G_{0,i}^{\tr{rdt}}= \Res_{\kappa_i|\bbF_p}H_i,\]
where $\kappa_i$ is the residue field of $F_i$, $e_i$ is the ramification index of the extension $F_i|\Q_p$, and $H_i$ is a reductive group over $\kappa_i$. Note that the Hodge cocharacter $\mu$ of $G$ over $F$ naturally admits a reduction to a cocharacter of $\G^\spl_0$ over $\kappa$.\\

Now we briefly explain the construction of the universal $\G^\spl_0$-zip of type $\mu$. In the unramified PEL case, we have $\scrA_0=\scrA^{\naive}_0=\scrA^{\spl}_0$ and the construction is direct, cf. \cite{MW, ViehmannWedhorn2013} and \cite{Zhang2018EO}. However, in the ramified case, the construction becomes rather indirect and complicated.
Since a general PEL datum is involved, we illustrate the ideas by working with the example of Hilbert-Siegel case. For the general case, see sections \ref{sec:splitting} and \ref{section Fzip}. Then the PEL datum $(B, \ast, V, \Pair{\cdot}{\cdot}, \calO_B,\Lambda)$ is of type C and $B=L$ is a totally real field which we assume $[L:\Q]>1$. We further assume that
there is only one finite place $v$ of $L$ over $p$ and the local field $F_1:=L_v$ is ramified over $\Q_p$, i.e. the ramification index $e=[F_1:F_1^{\tr{ur}}]>1$, where $F_1^{\tr{ur}}$ is the maximal unramified sub-extension of $F_1$ over $\Q_p$. Let $f=[F_1^{\tr{ur}}:\Q_p]=[\kappa_1:\mathbb{F}_p]$, where $\kappa_1$ is the residue field of $F_1$, so that $[F_1:\Q_p]=ef$. Let $K^p\subset G(\mathbb{A}_f^p)$ be a fixed sufficiently small open compact subgroup. Recall $F|F_1$ is a fixed large enough extension with residue field $\kappa$.
For any $\kappa$-scheme $S$, let $\underline{A}=(A, \lambda, \iota, \alpha)$ be an $S$-point of $\mathscr{A}^{\naive}=\mathscr{A}^{\naive}_{K^p}$. Then we have an exact sequence of locally free sheaves of $\calO_S$-modules \[0 \to \omega_{A/S} \to H_{\dR}^1(A/S) \to \Lie_{A^\vee/S} \to 0.\] In the unramified case, this exact sequence, together with the Frobenius and Verschiebung morphisms, defines the $F$-zip with additional structure of type $\mu$ over $\scrA^{\spl}_0$. However, in our ramified case here, it turns out a posteriori that the type of the Hodge filtrations constructed in this way will \emph{vary}, which reflects the singularities of $\scrA_0$ in some sense. 

Indeed, we have the Kottwitz-Rapoport stratification on $\scrA_0$ (see later), and on each stratum by \cite{ShenYuZhang2021}  we have a $\G_0^{\tr{rdt}}$-zip of certain type (dependent on this stratum). 
Recall the natural map $\scrA_0^\spl\ra \scrA_0$. We can pullback these $\G_0^{\tr{rdt}}$-zips to the corresponding preimages of KR strata in $\scrA_0^\spl$. But these do not give the correct object that we want.
Here, to construct the $\G^\spl_0$-zip of fixed type $\mu$ over $\scrA^\spl_0$, we have to take into account the splitting structures. It is here that our modified local model diagram (over $\kappa$) \[\scrA^\spl_0\ra [\mathbb{M}^\loc(\G^\spl,\mu)_0/\G^\spl_{\kappa}]\] plays the key role. Under our smoothness assumption, the underlying topological space of $[\mathbb{M}^\loc(\G^\spl,\mu)_0/\G^\spl_{\kappa}]$ has only one point. Then the construction is guided by the constructions in \cite{ShenYuZhang2021} subsections 3.3 and 3.4. On the other hand, we remark that it is very hard to work with the original local model diagram $\scrA^\spl_0\ra [\mathbb{M}^\spl(\G,\mu)_0/\G_{\kappa}]$ of \cite{PappasRapoport2005}. The quotient stack $[\mathbb{M}^\spl(\G,\mu)_0/\G_{\kappa}]$ is quite complicated, and in fact it is not known whether the set of $\calG_\kappa$-orbits on $\mathbb{M}^\spl(\G,\mu)_0$ is finite or not in general (cf. \cite{Berg} section 8 and \cite{Berg2}).\\

Let us discuss in more detail the construction of the $\G^\spl_0$-zip of fixed type $\mu$ over $\scrA^\spl_0$.
Let $(\underline{A}, \underline{\mathscr{F}_\bullet})$ be an $S$-valued point of ${\mathscr{A}_0}^\spl$. Then by definition
	$\underline{A}=(A, \lambda, \iota,\alpha)\in \mathscr{A}^{\naive}(S)$, $\underline{\mathscr{F}_\bullet}=\{\mathscr{F}_{j}^l\}$ is a Pappas-Rapoport filtration of $\calO_{F_1} \otimes_{\Z_p} \calO_S$-module $\omega_{A/S}$. To explain this term, 
	we write $\mathcal{H} =H_{\dR}^1(A/S)$, which is an $\calO_{F_1^{\tr{ur}}} \otimes_{\Z_p} \calO_S$-module,
	hence it has a decomposition \[\mathcal{H} =  \bigoplus_{j} \mathcal{H}_{j},\]
	where $\mathcal{H}_{j}$ is the locally free sub $\calO_S$-module of $\H$ where $\calO_{F_1^{\tr{ur}}}$ acts by the fixed embedding $\sigma_j: \calO_{F_1^{\tr{ur}}}\ra \calO_F$. For each $j$, there is a pairing on $\H_j$ induced from the pairing on $\H$ coming from the polarization $\lambda$.  Similarly, we have a decomposition $\omega_{A/S}=\bigoplus_j\omega_j$ with $\omega_j\subset \mathcal{H}_j$
	for each $1\leq j\leq f$. The Pappas-Rapoport filtration is by definition a filtration of locally direct $\calO_S$-factors of each $\omega_{j}$.
	\[0 = \mathscr{F}_{j}^0 \subset \mathscr{F}_{j}^1 \subset \cdots \subset \mathscr{F}_{j}^{e} = {\omega_{j}} \subset \mathcal{H}_{j}\]
	with $\calO_{F_1}/ (p) \simeq \kappa_1[T]/(T^{e}) = \kappa_1[\varepsilon_1]$ action, such that
		$\kappa_1$ acts on $\mathscr{F}_{j}^l$ by $\sigma_{j}: \mathcal{O}_{\Unram{F_1}} \to \kappa$;
		for all $1 \leq l \leq e$, $\mathscr{F}_{j}^l / \mathscr{F}_{j}^{l-1}$ is locally free of rank $d_{j}^l$ and $\varepsilon_1 \mathscr{F}_{j}^l \subset \mathscr{F}_{j}^{l-1}$.
		
The $\calO_{F_1^{\tr{ur}}}$-action on $\Lambda$ induces a decomposition $\Lambda=\bigoplus_{j}\Lambda_j$. For each embedding $\sigma^l_j: \calO_{F_1}\hookrightarrow \calO_F$ which extends $\sigma_j: \calO_{F_1^{\tr{ur}}}\hookrightarrow\calO_F$, we set $\Lambda_j^l=\Lambda_j\otimes_{\calO_{F_1},\sigma^l_j}\calO_F$ (note that the lattice $\Lambda_{j}^l$ is denoted by $\Xi_{j}^l$ in \cite{PappasRapoport2005}*{Proposition~5.2}) and $\Lambda^l=\bigoplus_j\Lambda_j^l$. Set $ \Lambda_{j,0}^l=\Lambda_j^l\otimes_{\calO_F}\kappa$ and $\Lambda^\spl_0=\bigoplus_l\bigoplus_j\Lambda_{j,0}^l$. One can show that there is a standard $F$-zip structure on $\Lambda^\spl_0$, whose Hodge filtration is given by the cocharacter $\mu$ over $\kappa$.

For $S={\mathscr{A}_0}^{\spl}$ and $(\underline{{A}}, \underline{\mathscr{F}_\bullet}) \in {\mathscr{A}_0}^{\spl}(S)$ the universal object,
one can associate a module
\[\mathcal{M}:= \bigoplus_{ l=1}^e \bigoplus_{ j=1}^f \mathcal{M}_{j}^l, \quad\text{ with }\quad\mathcal{M}_{j}^l := \varepsilon_1^{-1}\mathscr{F}_{j}^{l-1} / \mathscr{F}_{j}^{l-1}.\]
Then each $\mathcal{M}_{j}^l$ is a locally free $\calO_S$-module, and locally isomorphic to $\Lambda_{j}^l \otimes_{\calO_F} \calO_S$, cf. \cite{PappasRapoport2005}.
Now the idea is to transfer the Hodge and conjugate filtrations on $\mathcal{H}$ to filtrations on each $\mathcal{M}^l:=\bigoplus_j\mathcal{M}_{j}^l$, and to show that the $F$-zip structure on $\mathcal{H}$ induces an $F$-zip structure on $\mathcal{M}$. To this end, one can first complete the filtration of $\omega$ to a full filtration of $\mathcal{H}$, then apply the Frobenius and Verschiebung morphisms successively to the full filtration, and make use of the usual relation $\tr{Ker}\, F=\tr{Im}\, V, \tr{Ker}\, V=\tr{Im}\, F$ to succeed. In fact, we proceed by an alternative and equivalent approach. For any $1\leq j \leq f$, inspired by \cite{ReduzziXiao2017} we construct
 explicit linear maps $F_{j}^l: \mathcal{M}_{j}^{l-1} \to \mathcal{M}_{j}^{l},\; V_{j}^l: \mathcal{M}_{j}^l \to \mathcal{M}_{j}^{l-1}$ for each $2 \leq l \leq e$ (which are easy to construct),  and  $\sigma^{-1}$-linear ($\sigma$-linear) maps  
 $F_{j}^1: \mathcal{M}_{j-1}^{e} \to \mathcal{M}_{j}^1,\; V_{j}^1:\mathcal{M}_{j}^1 \to \mathcal{M}_{j-1}^{e}$ (which are induced by the Frobenius and Verschiebung, and the definition of $V_{j}^1$ is a little subtle).
%\begin{enumerate}
%	\item If $2\leq l \leq e$, define $V_{j}^l: \mathcal{M}_{j}^l \to \mathcal{M}_{j}^{l-1},\quad x \mapsto \varepsilon_i x,$. On the other side, the injection $\varepsilon_1^{-1}\mathscr{F}_{j}^{l-2} \hookrightarrow \varepsilon_1^{-1}\mathscr{F}_{j}^{l-1}$ induces a morphism
%	$F_{j}^l: \mathcal{M}_{j}^{l-1} \to \mathcal{M}_{j}^{l}.$
%	\item If $l=1$, let $V_{j}: \mathcal{H}_{j} \to \mathcal{H}_{j-1}$ be the Verschiebung morphism and $F_{j}: \mathcal{H}_{j-1} \to \mathcal{H}_{j}$ the Frobenius morphism.
%These induce a $\sigma^{-1}$-linear map 
%	\[V_{j}^1:\mathcal{M}_{j}^1 \to \mathcal{M}_{j-1}^{e},\quad x \mapsto V_{j}(\varepsilon_1^{1-e}(x)),\] 
%	and a $\sigma$-linear map \[F_{j}^1: \mathcal{M}_{j-1}^{e} \to \mathcal{M}_{j}^1,\quad x \mapsto F_{j}(x).\]
%\end{enumerate}
Now for each $1\leq l\leq e$, 
$\mathcal{M}^l = \bigoplus_{j=1}^f \mathcal{M}_{j}^l$
is a locally free $\calO_S[\varepsilon_1]$-module with $\kappa_1$ action, and locally
$\calM^l \simeq \Lambda^l \otimes_{\calO_F} \calO_S.$
%We then have
%\[\calM= \bigoplus_{1 \leq l \leq e} \calM^l \simeq \bigoplus_{1 \leq l \leq e} \Lambda_{0,k}^l =\Lambda_{0,k}^{\spl}.\]
%$\calM$ is a $k$ vector space of rank $e fd$.
For each $l$, set $F^l=\bigoplus_jF_{j}^l$ and $V^l=\bigoplus_jV_{j}^l$.
We have the following linear morphisms:
\[\xymatrix{
	{\mathcal{M}^{1}} \ar@<1ex>[r]^{F^2} & {\mathcal{M}^{2}} \ar@<1ex>[l]^{V^2} \ar@<1ex>[r]^{F^3} & \cdots \ar@<1ex>[l]^{V^3} \ar@<1ex>[r]^{F^{e}} & {\mathcal{M}^{e}} \ar@<1ex>[l]^{V^{e}}.
}\]
and semi-linear morphisms ($F^{1}$ is $\sigma$-linear and $V^{1}$ is $\sigma^{-1}$-linear)
\[\xymatrix{
	{\mathcal{M}^{e}} \ar@<1ex>[r]^{F^1} & {\mathcal{M}^1} \ar@<1ex>[l]^{V^1}.
}\]
Then (see Lemmas \ref{lem l big} and \ref{lem l=1} respectively)
\begin{enumerate}
	\item For $2\leq l\leq e$, we have
	$\Image\,F^l = \Ker\,V^l, \, \Ker\,F^l = \Image\,V^l.$
	%Moreover,
	%$\rank(\Ker\,V_{j}^l) = d_{j}^l$ and $\rank(\Ker\,F_{j}^l) = d - d_{j}^l$.
	\item 
	There is a canonical isomorphism $g: \calM^1 \to \calM^1$ sending $\Ker\,V^1$ to $\Image\,F^1$.
\end{enumerate}
Therefore
we get an $F$-zip
\[(\calM = \bigoplus_l \calM^l, C = \bigoplus_l \Ker\,F^l, D = \bigoplus_l \Ker\,V^l, \varphi_\bullet)\]
over $\scrA_0^{\spl}$.  Moreover, there are natural compatible symplectic form and $\calO_L$-action on $\mathcal{M}$.
Comparing with the standard $F$-zip structure on $\Lambda^\spl_0$, we get the $\G^\spl_0$-zip of type $\mu$ over $\scrA_0^{\spl}$. Roughly, the process of transforming $\mathcal{H}$ to $\calM$ is a way of ``semi-simplification''. During this transformation the different types of $\G^{\rdt}_0$-zips become uniform as $\G^\spl_0$-zips, which reflects some key features of the resolution of singularities $\scrA_0^\spl\ra \scrA_0$. Note that even in the Hilbert case, the above construction reveals more complete geometric information than that in \cite{ReduzziXiao2017}.
For more details, see Proposition \ref{prop splGzip} and subsection \ref{subsection zip}.\\

Let $k=\ov{\kappa}$ be an algebraic closure of $\kappa$. By Theorem \ref{thm intro 1}, we get an induced stratification 
\[ \scrA^\spl_k=\coprod_{w\in {}^JW}\scrA^{\spl,w}_k\]
which we call the Ekedahl-Oort (EO) stratification. Here the index set ${}^JW$ is the subset of the (absolute) Weyl group of $\G^\spl_0$ defined by the cocharacter $\mu$ in the usual way, which is equipped with the partially order $\preceq$ as that in \cite{PinkWedhornZiegler2011,PinkWedhornZiegler2015}.
By Theorem \ref{thm intro 1}, each stratum $\scrA^{\spl,w}_k$ is non-empty, smooth, and we have the usual closure relation for the EO strata given by $\preceq$. In fact, the non-emptiness is more subtle, requiring additional efforts. We prove the non-emptiness by comparing the minimal EO stratum of $\scrA^\spl_k$ with the minimal EKOR stratum of $\scrA_k$ (see below) in Proposition \ref{prop:EO-EKOR}, and show the minimal EKOR stratum is non-empty in the Appendix (Proposition \ref{prop non-empty EKOR}) by adapting the method of He-Zhou \cite{HeZhou2020}.

We can compare the EO strata with some other naturally raised strata. 
Recall the natural morphism \[\pi:\scrA^\spl_k\ra \scrA_k,\] which is a resolution of singularities: as in the Hilbert case (\cite{DP})
the scheme $\scrA_k$ is usually singular in the ramified case.
On $\scrA_k$, we have the following stratifications:
\begin{itemize}
	\item Newton stratification, which can be constructed by studying the associated $F$-isocrystals with additional structure.
	\item Kottwitz-Rapoport (KR) stratification $\scrA_k=\coprod_{w\in \tr{Adm}(\mu)_{K_p}}\scrA^{w}_k$, which is induced from the local model diagram for $\scrA_k$ and the geometry of $\mathbb{M}^\loc_k$, cf. \cite{PappasRapoport2005, Levin2016, HainesRicharz2020}.
	\item Ekedahl-Kottwitz-Oort-Rapoport (EKOR) stratification $\scrA_k=\coprod_{x\in {}^{K_p}\tr{Adm}(\mu)}\scrA^{x}_k$, which
	can be constructed as in \cite{ShenYuZhang2021}, see the Appendix of the current paper; in particular, see Example \ref{example EKOR} for more information on EKOR strata in the Hilbert case.
\end{itemize}
While on $\scrA^\spl_k$ we have the following stratifications:
\begin{itemize}
	\item Newton stratification, which can be constructed by forgetting the splitting structures and considering the associated $F$-isocrystals with additional structure,
	\item Hodge stratification, which is constructed in \cite{BijakowskiHernandez2022},
	\item Ekedahl-Oort stratification, which is constructed in this paper.
\end{itemize}
The Newton stratification on $\scrA^\spl_k$ naturally factors through $\scrA_k$. In subsection \ref{subsection Hodge and KR} we will show that the Hodge stratification also factors through $\scrA_k$. Moreover, the resulted  stratification on $\scrA_k$ is coarser than the Kottwitz-Rapoport stratification. On the other hand, the Ekedahl-Oort stratification on $\scrA^\spl_k$ does not factor through $\scrA_k$ in general. Nevertheless, one can try to compare the EO strata of $\scrA^\spl_k$ with pullbacks of KR and EKOR strata of $\scrA_k$ under $\pi$.
On each KR stratum $\scrA_k^w$, we have a $\G_0^{\tr{rdt}}$-zip of type $J_w$, see \cite{ShenYuZhang2021} and our Appendix. Pulling back to  $\scrA^\spl_k$, we get a family of $\G_0^{\tr{rdt}}$-zips of different types. It is curious to study the relationship with our $\G^\spl_0$-zip of fixed type $\mu$. In doing so, we prove the following results.
\begin{theorem}[Propositions \ref{maxKR} and \ref{prop EO in max KR}]
\begin{enumerate}
	\item Let $\scrA_k^{w_0}$ be the maximal KR stratum of $\scrA_k$. Then $\pi$ induces an isomorphism \[\pi^{-1}(\scrA_k^{w_0})\st{\sim}{\ra}\scrA_k^{w_0};\] furthermore $\pi^{-1}(\scrA_k^{w_0})$ is also the maximal Hodge stratum (called the generalized Rapoport locus in \cite{BijakowskiHernandez2022}).
	\item $\pi^{-1}(\scrA_k^{w_0})$ is a disjoint union of some EO strata of $\scrA^\spl_k$, and moreover
	 these EO strata are exactly pullbacks of the EKOR strata contained in $\scrA_k^{w_0}$. Thus the isomorphism $\pi^{-1}(\scrA_k^{w_0})\st{\sim}{\ra}\scrA_k^{w_0}$ refines into isomorphisms between the corresponding EO and EKOR strata.
\end{enumerate}
\end{theorem}
In particular, using a result of He-Nie \cite{HeNie2017}, we can deduce that the maximal EO stratum coincides with the $\mu$-ordinary locus, cf. Corollary \ref{cor mu-ordinary}. Thus by Theorem \ref{thm intro 1} we reprove the open density of $\mu$-ordinary locus of $\scrA^\spl_k$, cf. \cite{BijakowskiHernandez2022}.\\

Once we have the morphism \[\zeta: \scrA^\spl_0\ra \G^\spl_0\tr{-Zip}^\mu_{\kappa},\] we get Hasse invariants on $ \scrA^\spl_0$ by pulling back the group theoretic Hasse invariants on the stack $\G^\spl_0\tr{-Zip}^\mu_{\kappa}$ constructed by Goldring-Koskivirta in \cite{GoldringKoskivirta2019}. Note that in \cite{BijakowskiHernandez2017} Bijakowski and Hernandez have constructed Hasse invariants for the $\mu$-ordinary locus (which is also our maximal EO stratum) by an explicit method. It would be interesting to compare their construction with our construction here. On the other hand, in the Hilbert case, in subsection \ref{subsection exa Hilbert} we do compare the EO strata and Hasse invariants here with those constructed in \cite{ReduzziXiao2017}.\\

To proceed as \cite{ReduzziXiao2017} and \cite{GoldringKoskivirta2019} to get applications to Galois representations, we need arithmetic compactifications of splitting models and we have to extend the above construction to the compactifications. Fortunately, the arithmetic (toroidal and minimal) compactifications for splitting models with good properties have already been established by Lan in \cite{Lan18}, based on his previous constructions in \cite{Lan13, Lan16, Lan17}. At this point, we have to slightly modify the integral model $\scrA^\spl=\scrA^\spl_{K^p}$ by considering the open closed subspace \[\M_K^\spl\subset \scrA^\spl_{K^p}\] studied in \cite{Lan18}. Here $K=K^pK_p$ and $K_p=\G(\Z_p)$.
All previous constructions and results also hold for $\M_K^\spl$ by simple modifications.
For a projective and smooth compatible collections of cone decompositions $\Sigma$, by \cite{Lan18} we have a toroidal compactification $\M_{K, \Sigma}^{\spl,\tor}$, together with a proper surjective morphism to the minimal compactification $\oint: \M_{K, \Sigma}^{\spl,\tor}\ra \M_{K}^{\spl,\min}$.
As we are working with smooth $\M_K^\spl$, the compactification $\M_{K, \Sigma}^{\spl,\tor}$ is also smooth by \cite{Lan18} Proposition 3.4.14.
We write their special fibers as $\M^{\spl}_{K,0}$ and $\M^{\spl,\tor}_{K,\Sigma,0}$.
\begin{theorem}[Theorem \ref{thm zip map ext}]\label{thm intro 2}
	\begin{enumerate}
		\item 
		The $\calG^\spl_0$-zip of type $\mu$ on $\M^{\spl}_{K,0}$ extends to a $G(\bbA_f^p)$-equivariant $\calG^\spl_0$-zip of type $\mu$ on $\M^{\spl,\tor}_{K,\Sigma,0}$.
		\item The induced map $\zeta^\tor: \M^{\spl,\tor}_{K,\Sigma,0}\ra \calG^\spl_0\tr{-}\mathrm{Zip}^\mu_\kappa$ is smooth.
	\end{enumerate}
\end{theorem}
The extension of $\calG^\spl_0$-zip is easy to construct, as the canonical extension of the universal de Rham bundle together with its Hodge filtration has already been given by Lan in \cite{Lan18}, so we just repeat the above construction. To prove the smoothness of $\zeta^\tor$, we adapt some ideas of Andreatta \cite{And22} in the unramified Hodge type case. We also prove that EO strata here are well-positioned in the sense of Lan-Stroh \cite{LS18}. By the well-positionedness in fact we get another proof of the smoothness of $\zeta^\tor$ (this is the approach taken in \cite{LS18} for the unramified PEL case).\\

Before talking about arithmetic applications, we note another advantage of the smooth splitting models $\M_K^\spl$. As the modified local model diagram for $\M_K^\spl$ is similar to that in the unramified case, we get naturally automorphic vector bundles on $\M_K^\spl$ as in the classical procedure. In contrast, automorphic vector bundles do not extend to the singular integral model $\M_K$, the corresponding open and closed subspace of $\scrA=\scrA_{K^p}$.
We also get canonical and subcanonical extensions of automorphic vector bundles to $\M_{K, \Sigma}^{\spl,\tor}$, as the local model diagram extends. Moreover, pull backs of canonical extensions of automorphic vector bundles satisfy similar properties as in the unramified case (cf. \cite{Lan16a} Proposition 5.6). In particular, we have the relative vanishing of higher direct images for the projection $\oint: \M_{K, \Sigma}^{\spl,\tor}\ra \M_{K}^{\spl,\min}$, cf. Proposition \ref{prop rel vanish}  and \cite{Lan18} Theorem 4.4.9.\\

Theorems \ref{thm intro 1} and \ref{thm intro 2} allow us to deduce some further consequences on the mod $p$ geometry and arithmetic related to PEL-type Shimura varieties with ramifications arising from Weil restrictions. As an illustration, we discuss congruences of mod $p$ automophic forms and Galois representations associated to torsion classes in coherent cohomology. For some other possible applications, see \cite{DK20} and \cite{WZ}.
From Theorem \ref{thm intro 2}, we get extensions of Hasse invariants to toroidal compactifications. 
Then we can apply Goldring-Koskivirta's machinery of Hasse-regular sequences introduced in \cite{GoldringKoskivirta2019}, which formalizes some key properties of the subschemes defined by Hasse invariants. Indeed, the key technical conditions 6.4.2 and 7.1.2 of \cite{GoldringKoskivirta2019} have been verified in our setting as above.
Let $S$ be the finite set of primes $\ell$ where $K_\ell$ is not hyperspecial. Let $P_\mu\subset G_F$ be the parabolic subgroup associated to the cocharacter $\mu: \mathbb{G}_{m,F}\ra G_F$ in the usual way (see the paragraph above Corollary \ref{coro ext vb}), $\mathbf{L}\subset P_\mu$ the Levi subgroup and $V_\eta\in \Rep_F\mathbf{L}$ an irreducible representation of highest weight $\eta$. Using our modified local model diagram, we get the associated integral automorphic vector bundle $\V_\eta$ on $\M_K^\spl$ and the subcanonical extension $\V_\eta^\sub$ of $\V_\eta$ to $\M_{K, \Sigma}^{\spl,\tor}$. For each integers $i\geq 0, n\geq 1$, 
the Hecke algebra $\H^S$ acts on the cohomology groups \[H^i(\M_{K, \Sigma,\calO_F/\varpi^n}^{\spl,\tor}, \V_\eta^\sub),\] where $\varpi$ is a uniformizer of $\calO_F$. Let $\H^{i,n}_\eta$ be its image in the endomorphism algebra.
Fix a representation $r: {}^LG\ra \mathrm{GL}_m$ of the Langlands dual group over $\mathbb{C}$ (and fix an isomorphism $\mathbb{C}\simeq \ov{\Q}_p$). With all the geometric ingredients at hand, one can prove the following theorem by the same arguments as \cite{GoldringKoskivirta2019}.
\begin{theorem}[Theorem \ref{thm Galois}]\label{thm 3}
	Suppose that for any regular $C$-algebraic cuspidal automorphic representation $\pi'$ of $G$ with $\pi_{\infty}'$ discrete series, the pair $(\pi',r)$ satisfies the condition $LC_p$ of subsection \ref{subsec Galois}, which roughly says the existence of $\mathrm{GL}_m$-valued $p$-adic Galois representation attached to $\pi'$ satisfying the local-global compatibility outside $S$.
	\begin{enumerate}
		\item 
		For any $i\geq 0, n\geq 1, \eta\in X^\ast(T)_{\mathbf{L}}^+$,
		there exists a continuous Galois pseudo-representation
		\[\rho: \mathrm{Gal}(\ov{\Q}/\Q)\lra \H^{i,n}_\eta, \]such that $\rho(\Frob_v^j)=T_v^{(j)}$ for all $v\notin S$, where $\Frob_v$ is the geometric Frobenius at $v$ and $T_v^{(j)}\in \H^{i,n}_\eta$ is the element defined in subsection \ref{subsec Galois}.
		
		\item Let $\pi$ be a $C$-algebraic cuspidal automorphic representation of $G$ such that $\pi_\infty$ is a ($C$-algebraic) non-degenerate  limit of discrete series and $\pi^{K_p}_p\neq 0$. Then $(\pi,r)$ also satisfies $LC_p$.
	\end{enumerate}
\end{theorem}
As \cite{GoldringKoskivirta2019}, this theorem is a consequence of a result that the Hecke action on $H^i$ factorizes through $H^0$ with increased weight, which is in turn deduced by the machinery of Hasse invariants on toroidal compactifications, see Theorem \ref{Thm Hecke alg}.

For the unramified PEL type case, Theorem \ref{thm 3} is due to \cite{GoldringKoskivirta2019} (the key technical conditions 6.4.2 and 7.1.2 there hold for the unramified PEL type case), which recovers the main result of \cite{ERX} (for the unramified Hilbert case), and completes the work of Boxer \cite{Box} by a different approach. We refer to \cite{GoldringKoskivirta2019} subsection 10.1 for further discussions on the appearance of non-degenerate  limit of discrete series in higher coherent cohomology of Shimura varieties.
Recall that since $K_p\subset G(\Q_p)$ is a very special parahoric subgroup, irreducible smooth representations $\pi_p$ of $G(\Q_p)$ such that $\pi^{K_p}_p\neq 0$ can be classified by their spherical parameters, see \cite{Zhu} section 6. The significance of Theorem \ref{thm 3} lie in the flexibility it provides for the construction of automorphic Galois representations, in the sense that we allow the $p$-component of the automorphic representation $\pi$ to be ramified as $\pi^{K_p}_p\neq 0$. It reproves \cite{ReduzziXiao2017} Theorem 1.1 in the Hilbert case by a different method. We also mention that a similar result has been given in \cite{PS16} in the ramified setting, but there the integral structure is given by the strange integral models introduced by Scholze using perfectoid geometry.
As all these works demonstrate, Theorem \ref{thm 3} should be useful when studying modularity lifting problems, cf. \cite{CG}.\\

In this paper we restrict to the PEL case, since so far splitting models for Shimura varieties have only been constructed in this setting (\cite{PappasRapoport2005}). It would be interesting and useful to construct splitting models in more general Hodge and abelian type cases. This would certainly require quite different ideas and independent treatments. Once more general splitting models are available, we do expect that our constructions in this paper would naturally extend.\\

This paper is organized as follows. In section 2, we fix the PEL datum and discuss the corresponding Pappas-Rapoport splitting models. In particular, under a basic assumption on the datum we prove the smoothness of splitting models by a modified local model diagram. In section 3, we construct the universal $\G^\spl_0$-zip of type $\mu$ over the smooth special fiber $\scrA^\spl_0$ and study the induced Ekedahl-Oort stratification. As an application, we prove the open density of the $\mu$-ordinary locus. Along the way we also show that the Hodge stratification constructed in \cite{BijakowskiHernandez2022} descends to $\scrA_{0}$. In section 4, we discuss Hasse invariants on $\scrA^\spl_0$, by pulling back those on the zip stack constructed by  Goldring-Koskivirta. We also discuss some concrete examples, in particular we compare the Hasse invariants and EO strata here with those constructed by Reduzzi-Xiao in \cite{ReduzziXiao2017} in the Hilbert case. In section 5, we first review Lan's constructions of arithmetic compactifications for splitting models, then we extend the universal $\G^\spl_0$-zip  to the smooth toroidal compactifications $\M_{K, \Sigma,0}^{\spl,\tor}$  and prove the smoothness of the induced morphism $\zeta^\tor$. In section 6, we apply the method of Goldring-Koskivirta \cite{GoldringKoskivirta2019} to deduce similar consequences on Hecke algebras and Galois representations associated to torsion classes in coherent cohomology of smooth splitting models. Finally, in the Appendix we discuss related local model diagrams for general parahoric levels, and briefly review the construction of EKOR stratification on $\scrA_{0}$.\\
 \\
\textbf{Acknowledgments.} We thank Sian Nie for some helpful discussions. We would like to thank the referee for careful reading and detailed comments, which help to improve the paper.

\section{Splitting models of PEL-type Shimura varieties with good reduction}\label{sec:splitting}

In this section, we review the definition of Pappas-Rapoport splitting models for PEL-type Shimura varieties, following \cite{PappasRapoport2005} and \cite{Lan18}, for a maximal level at $p$. 
In the Appendix \ref{section loc mod} we will work with more general parahoric levels, and discuss the related local model diagrams and EKOR stratifications, which are needed in the main text for the maximal level case.
Our setting here is the same as that in \cite{BijakowskiHernandez2022}.
In fact, we will reprove the main results of loc. cit. by a different approach.

\subsection{Integral PEL datum}\label{subsec PEL data}
\subsubsection{PEL datum}
Fix a prime $p>2$.
Let $(B, *, V, \Pair{\cdot}{\cdot}, \calO_B, \Lambda, h)$ be an integral PEL datum at $p$.
This means (see \cites{Kottwitz1992, RapoportZink1996, Lan13})
\begin{itemize}
	\item $B$ is a finite dimensional semisimple algebra over $\mathbb{Q}$ with a positive involution $*$. 
		We further assume that \emph{$B_{\mathbb{Q}_p}$ is isomorphic to a product of matrix algebras over finite extensions of 
		$\mathbb{Q}_p$} (note that we allow the extensions to be \emph{ramified} over $\mathbb{Q}_p$).
	\item $V$ is a finitely generated faithful $B$-module.
	\item $\Pair{\cdot}{\cdot}: V \times V \to \mathbb{Q}$ is a non-degenerate symplectic form on $V$ such that 
		$\Pair{bv}{w}=\Pair{v}{b^*w}$ for all $v,w\in V$ and $b\in B$.
	\item $\calO_B$ is a $*$-invariant $\mathbb{Z}_{(p)}$-order of $B$ such that $\calO_B \otimes \mathbb{Z}_p$
		is a maximal $\mathbb{Z}_p$-order of $B_{\mathbb{Q}_p}$.
	\item $\Lambda$ is a $\calO_B\otimes\Z_p$-lattice in $V_{\mathbb{Q}_p}$, such that $\Pair{\cdot}{\cdot}$ induces a (not necessary perfect)
		pairing $\Lambda \times \Lambda \to \mathbb{Z}_p$.
	\item $G$ is the algebraic group over $\mathbb{Q}$ of (similitude) automorphisms of $(V, \Pair{\cdot}{\cdot})$, 
		i.e. for any $\mathbb{Q}$-algebra $R$, we have
		\[G(R) = \{g \in \GL_B(V_R) \mid \Pair{g z}{g z'} = c(g) \Pair{z}{z'} \text{ for some } c(g) \in R^\times, \, \forall\, z, z' \in V_R\}.\]
		We further require the group $G$ to be \emph{connected}.
	\item $h:\mathbb{S}=\Res_{\mathbb{C}|\mathbb{R}}(\mathbb{G}_m) \to G_\mathbb{R}$ is a group homomorphism which defines a
		Hodge structure of type $ \{(-1,0),(0,-1)\}$ on $V_\mathbb{R}$ and the form $\Pair{\cdot}{h(\sqrt{-1})\cdot}$ is
		symmetric and positive definite. (Note that $(\Lambda,\Pair{\cdot}{\cdot}, h)$ is a PEL-type $\calO_B$-lattice in the sense of \cite{Lan13} Definition 1.2.1.3.)
	\item $\mu:\mathbb{G}_{m,\mathbb{C}} \to G_\mathbb{C}$ is the Hodge cocharacter associated to $h$.
\end{itemize}

We make some explicit description of this datum, which will be fixed in the rest of this paper.
By assumption, there is a decomposition (as product of $*$-invariant simple factors)
\[B_{\mathbb{Q}_p} \simeq  \prod_{i=1}^r B_i = \prod_{i=1}^r \Mat_{m_i}(R_i), \]
where for each $1\leq i\leq r$, $R_i$ is a product of finite extensions (maybe \emph{ramified}) of $\mathbb{Q}_p$.
Let $I := \{1,\dots, r\}$ be the index set.
For each $i \in I$, 
let $F_i$ be the set of $*$-invariant elements in $R_i$, which is a field over $\bbQ_p$. Since we require the group $G$ to be connected,  
we can decompose the index set $I$ into four types: (C), (AL), (AU), (AR) (see also \cite{BijakowskiHernandez2022} subsection 2.2, but note that our $I$ is the quotient of their $\{1,\dots, r\}$ by the induced action of $\ast$). For each $i\in I$,
its type is defined by
\begin{enumerate}
	\item [(C):] if $*$ induces the identity on $R_i$, so we have $R_i = {F_i}$.
	\item [(AL):] if $R_i \simeq {F_i} \times {F_{i}}$ and $*$ exchanges the two factors.
	\item [(AU):] if $*$ is an automorphism of order 2 on $R_i$, and $R_i|F_i$ an unramified quadratic extension.
	\item [(AR):] if $*$ is an automorphism of order 2 on $R_i$, and $R_i|F_i$ a \emph{ramified} quadratic extension.
\end{enumerate}
As in this paper we are mainly interested in good reduction of splitting models, we will make the following assumption:
\[\text{For each index } i \in I, \text{ its type is \emph{not} (AR)}.\]
Thus our setting is the same as that in \cite{BijakowskiHernandez2022}. In the unramified case, see also \cite{ViehmannWedhorn2013} section 2.

For each $i\in I$,
let \[n_i = e_i f_i\] be the degree of $F_i$ over $\mathbb{Q}_p$,
where $e_i$ is the ramification index and $f_i$ the residue degree of $F_i |\mathbb{Q}_p$. Denote $F_i^{\ur}$ the maximal unramfied extension of $\Q_p$ inside $F_i$. Then $e_i=[F_i: F_i^{\ur}]$ and $f_i=[F_i^{\ur}: \Q_p]$.
The decomposition of $B_{\bbQ_p}$ induces a decomposition
\[V_{\bbQ_p} = \bigoplus_{i=1}^r V_i^{m_i}.\]
For each $i \in I$, 
the existence of $*$-symplectic form implies that $V_i$ is a vector space over $F_i$ of even dimension, so we write
\[d_i = \frac{1}{2} \dim_{F_i} V_i,\quad\tr{thus}\quad
\dim_{\Q_p}V_i=2d_in_i.\]
%\[\dim_{\bbQ_p} V_i = 
%\begin{cases}
%	4 n_i d_i & \text{if } i \text{ has type (AL)}, \\
%	2 n_i d_i & \text{otherwise}.
%\end{cases}\]
We have an isomorphism
\[\calO_B\otimes\Z_p\simeq \prod_{i=1}^r\Mat_{m_i}(\calO_{R_i})\] with $\calO_{R_i}$ the maximal order of $R_i$, and the induced decomposition \[\Lambda = \bigoplus_{ i=1}^r \Lambda_i^{m_i},\]
where each $\Lambda_i$ is an $\calO_{R_i}$-lattice in $V_i$, and there is an induced pairing $\langle\cdot,\cdot\rangle_i: \Lambda_i\times\Lambda_i\ra \calO_{R_i}$ if $i$ is of type (C) or (AL),  and $\langle\cdot,\cdot\rangle_i: \Lambda_i\times\Lambda_i\ra \calO_{F_i}$ if $i$ is of type (AU).
\subsubsection{Rational group structure}\label{subsub rational}
The restriction of the pairing $\langle\cdot,\cdot\rangle$ on $V_i$ defines an algebraic group $G_i=\Res_{F_i|\Q_p}H_i$ over $\bbQ_p$:
 for any $F_i$-algebra $R$, we have
\[H_i(R) = \{g_i \in \GL_{B_i}(V_{i,R}) \mid \Pair{g_i x}{g_i y} = c_i(g_i) \Pair{x}{y} \text{ for some } c_i(g_i) \in R^\times,\,\forall\, x,y \in V_{i,R}\}.\]
Then \[G_{\bbQ_p} \subset \prod_{i \in  I}G_i\] is the subgroup such that the local similitude factors $c_i$ are the same and defined over $\Q_p$.
From now on, we fix an isomorphism of fields $\mathbb{C} \simeq \overline{\mathbb{Q}}_p$. Then we view $\mu$ as a cocharacter of $G$ over $\ov{\Q}_p$.
Let $F$ be a fixed sufficiently large finite Galois extension of $\mathbb{Q}_p$ containing
all embeddings of $F_i$ to $\overline{\mathbb{Q}}_p$ such that $G_F$ is split.
Then the cocharacter $\mu$ is defined over $F$ and we can write
\[\mu: \bbG_{m,F} \to G_F \subset \prod_{i\in I} G_{i,F}.\]
For each $i\in I$,
the projection of $G_F$ to $G_{i,F}$ induces a cocharacter 
\[\mu_i: \bbG_{m,F} \to G_{i,F}.\]

For each $i \in I$, we order the embeddings $ F_i^{\rm ur} \hookrightarrow F$ as $\sigma_{i,1},\dots\,\sigma_{i,f_i}$. For each $1\leq j\leq f_i$,
there are $e_i$ extensions of $\sigma_{i,j}$ to embeddings $F_i\hookrightarrow F$, ordered and denoted by $\sigma_{i,j}^l: F_i \to F$, $1\leq l \leq e_i$.
We have the decompositions 
\[V_{i,F} = \bigoplus_{j=1}^{f_i} V_{i,j},\quad  V_{i,j} = \bigoplus_{l=1}^{e_i} V_{i,j}^l.\]
For each $1\leq j\leq f_i$,
the space $V_{i,j}$ is the subspace of $V_{i,F}$ such that $\Unram{F_i}$ acts through $\sigma_{i,j}$,
and for each $1\leq l\leq e_i$ the space $V_{i,j}^l$ is the subspace of $V_{i,j}$ such that $F_{i}$ acts through $\sigma_{i,j}^l$.
Note that we also have \[2 d_i = \rank_{F} V_{i,j}^l,\] 
and the pairing $\langle\cdot,\cdot\rangle$ on $V_i$ induces pairings
$\langle\cdot,\cdot\rangle: V_{i,j}^l\times V_{i,j}^l\ra F$ (which in fact are defined over $F_i$).
For each $i \in I$, as
$G_i = \Res_{F_i|\bbQ_p} H_i$,
we have 
\[G_{i,F} = \prod_{\sigma_{i,j}^l: F_i \to F} G_{i,j}^l,\]
where in the index $\sigma_{i,j}^l$ runs through all $1\leq j\leq f_i$ and $1\leq l\leq e_i$, and each
$G_{i,j}^l$ is (isomorphic to) a copy of $H_{i,F}$. 
So the cocharacter
\[\mu_i: \bbG_{m,F} \to G_{i,F} = \prod_{j,l} G_{i,j}^l,\]
induces a cocharacter
\[\mu_{i,j}^l: \bbG_{m,F} \to G_{i,j}^l\]
for each $1\leq j\leq f_i$ and $1\leq l \leq e_i$.

Since the only weights of $\mu: \mathbb{G}_{m,F} \to G_{F}$ are $0$ and $1$,
we get an induced decomposition
\[V_{F} = W \oplus W',\]
where $z \in F^\times$ acts on $W$ (resp. $W'$) by 1 (resp. $z$).
This implies that the pairing $\langle\cdot,\cdot\rangle$ induces an isomorphism
\[W \simeq (W')^\vee := \Hom_{F}(W', F).\]
We can also decompose the $F$-vector space $W$ as
\[W = \bigoplus_{i=1}^r W_i^{m_i}, \quad W_i = \bigoplus_{j=1}^{f_i} W_{i,j}\quad \text{ and } \quad W_{i,j} = \bigoplus_{l=1}^{e_i} W_{i,j}^l.\] We have
similar decompositions for $W'$. For each $i,j,l$, we have $V_{i,j}^l=W_{i,j}^l\oplus (W')_{i,j}^l$.
Let $d_{i,j}^l = \dim_F W_{i,j}^l$.
If we write $d_{i,j}^l = d_{\sigma_{i,j}^l}$ to make explicit the embedding $\sigma_{i,j}^l$, 
then there is an induced $*$-action on $\{\sigma_{i,j}^l\}_{i,j,l}$ and we set $(d_{i,j}^l)^* := d_{(\sigma_{i,j}^l)^*}$.
The natural isomorphism $W \simeq (W')^\vee$ gives the identity
\[(d_{i,j}^l)^* = \dim_F (W')_{i,j}^l,\quad \text{ and }\quad d_{i,j}^l + (d_{i,j}^l)^* = 2 d_i.\]
We can make these data more explicit according to the type of $i$: recall $G_i=\Res_{F_i|\Q_p}H_i$
\begin{enumerate}
	\item [(C):] in this case $R_i=F_i$, $\dim_{F_i}V_i=2d_i$, $H_i\simeq \GSp_{2d_i}$. We have
	 \[d_{i,j}^l=\dim_FW_{i,j}^l=\dim_F (W')_{i,j}^l=d_i\] for each $1\leq j\leq f_i$ and $1\leq l\leq e_i$.
	\item [(AL):] in this case $R_i=F_i\times F_i$, so we have a decomposition \[V_i=U_i\oplus U_i^{\vee}\] with $\dim_{F_i}U_i=d_i$, $H_i\simeq \GL_{d_i}\times \mathbb{G}_m$. 
	%We have a composition of $F_i$-vector spaces \[U_i=B_i\oplus B_i'\] induced from the $R_i$-module decomposition $V_i=W_i\oplus W_i'$ given by $\mu_i$. 
	Then there are induced decompositions
	$V_{i,j}^l=U_{i,j}^l\oplus (U^\vee)_{i,j}^{l}$,
	$W_{i,j}^l=A_{i,j}^l\oplus (A^\vee)_{i,j}^l$ and $(W')_{i,j}^l=(A')_{i,j}^l\oplus ((A')^\vee)_{i,j}^{ l}$. Let $a_{i,j}^l=\dim_F A_{i,j}^l, b_{i,j}^l=\dim_F(A')_{i,j}^l$, then by similar notations as above  \[(a_{i,j}^l)^\ast=b_{i,j}^l \quad \tr{and}\quad a_{i,j}^l+b_{i,j}^l=d_i\] for each $1\leq j\leq f_i$ and $1\leq l\leq e_i$.
	\item [(AU):] in this case $R_i|F_i$ is an unramified quadratic extension, and $V_i$ is a Hermitian space over $R_i|F_i$, $H_i=\tr{GU}(V_i,\langle\cdot,\cdot\rangle_i)$.
	We have $W_{i,j}^l=A_{i,j}^l\oplus B_{i,j}^l$ and $(W')_{i,j}^l=(A')_{i,j}^l\oplus (B')_{i,j}^l$. Let $a_{i,j}^l=\dim_F A_{i,j}^l, b_{i,j}^l=\dim_F(A')_{i,j}^l$, then by similar notations as above  \[(a_{i,j}^l)^\ast=b_{i,j}^l \quad \tr{and}\quad a_{i,j}^l+b_{i,j}^l=d_i\] for each $1\leq j\leq f_i$ and $1\leq l\leq e_i$.
\end{enumerate}

In the rest of the paper, 
the index $(i,j,l)$ will always be a tuple of integers running through
$i \in I, 1\leq j \leq f_i$ and $0 \leq l \leq e_i$.
\subsubsection{Integral group structure}\label{subsubsec int group}
Now we discuss the integral group theoretic data which will be needed later. Consider the parahoric group scheme $\G$ over $\Z_p$  defined by the integral PEL datum $(\calO_B, \ast,\Lambda, \langle\cdot,\cdot\rangle)$.
Recall the decomposition
$\Lambda = \bigoplus_{ i=1}^r \Lambda_i^{m_i}$,
where each $\Lambda_i$ is an $\calO_{R_i}$-lattice in $V_i$ together with an induced form $\langle\cdot,\cdot\rangle_i: \Lambda_i\times\Lambda_i\ra \calO_{R_i}$ ($i$ of type (C) or (AL)) or $\langle\cdot,\cdot\rangle_i: \Lambda_i\times\Lambda_i\ra \calO_{F_i}$ ($i$ of type (AU)).
Let $\calG_i / \bbZ_p$ be the parahoric subgroup of $G_i$ associated to $(\Lambda_i, \langle\cdot,\cdot\rangle_i)$,
then \[\G\subset \prod_{i\in I} \G_i\] is the subgroup of elements with the same local similitude factors defined over $\Z_p^\times$.
For each $i$, the  $\calO_{F_i}$-lattice $\Lambda_i$ is self dual with respect to the form $\langle\cdot,\cdot\rangle_i$, so that the above group $H_i$ is unramified over $F_i$. By abuse of notation, we still denote by $H_i$ its reductive integral model over $\calO_{F_i}$. Then
we have \[\calG_{i} = \Res_{\calO_{F_i}|\bbZ_p} H_i.\]

Recall that we have fixed a large enough field extension $F|\Q_p$.
Let $\kappa$ (resp. $\kappa_i$ for each $i\in I$) be the residue field of $F$ (resp. $F_i$ for each $i\in I$). For each $i\in I$, we denote also by $H_i$ the associated reductive group over $\kappa_i$.
Consider the artinian $\kappa_i$-algebra defined by the quotient of polynomial ring
$\kappa_i[\varepsilon_i] := \kappa_i[x]/(x^{e_i})$.
Then we have
\[ \calG_{i,0}:=\calG_i \otimes_{\bbZ_p} \bbF_p \simeq \Res_{\kappa_i[\varepsilon_i]|\bbF_p} H_i.\]
By \cite{Oesterle1984}*{Appendice 3}, there is an exact sequence
\[1 \to U \to \calG_{i,0} \to \Res_{\kappa_i |\bbF_p} H_i \to 1,\]
where $U$ is a connected unipotent subgroup of $\calG_{i,0}$. Let $\calG_{i,0}^{\rdt}$ be the maximal reductive quotient of $\calG_{i,0}$.
So we have
\[\calG_{i,0}^{\rdt} \simeq \Res_{\kappa_i|\bbF_p} H_i.\]
This group is the similitude automorphism group associated to \[\Lambda_{i,0}^{\rdt} := \Lambda_{i,0} \otimes_{\kappa_i[\varepsilon_i]}\kappa_i,\quad\tr{where}\quad \Lambda_{i,0}=\Lambda_i\otimes_{\Z_p}\bbF_p,\] which admits an induced $\kappa_i[\varepsilon_i]$-action and a pairing $\langle\cdot,\cdot\rangle_i: \Lambda_{i,0}\times \Lambda_{i,0}\ra \kappa_i$ (for $i$ of type (C) or (AU)) or $\langle\cdot,\cdot\rangle_i: \Lambda_{i,0}\times \Lambda_{i,0}\ra \kappa_i\times \kappa_i$ (for $i$ of type (AL)). Let $\calG_{0}^{\rdt}$ be the maximal reductive quotient of $\calG_{0}$. We have similarly an inclusion
\[ \calG_{0}^{\rdt}\subset \prod_{i\in I}\calG_{i,0}^{\rdt}.\]

\subsubsection{The group $\calG_0^\spl$}\label{subsubsec calG}
For each $i\in I$,
we define the group $\calG_{i,0}^{\spl} / \bbF_p$ as the similitude automorphism group associated to $(\Lambda_{i,0}^{\rdt})^{e_i}$ with its natural pairing. Then we have an isomorphism
\[\calG_{i,0}^{\spl} \simeq\Res_{\kappa_i|\bbF_p} H_i^{e_i}.\]Putting together, we define \[\G^{\spl}_0\subset\prod_{i\in I}\calG_{i,0}^{\spl}\] as the subgroup of element with the same $\F_p$-local similitude factors.
For each $i\in I, 1\leq j\leq f_i,1\leq l\leq e_i$,  by our choice of $F$, $G_{i,j}^l$ is a split reductive group over $F$. We denote by the same notations the reductive group scheme over $\calO_F$ and $\kappa$. Then each $\mu_{i,j}^l$ extends to a cocharacter of $G_{i,j}^l$ over $\calO_F$ 
and thus a cocharacter over $\kappa$:
\[\mu_{i,j}^l: \bbG_{m,\kappa} \to G_{i,j}^l.\]
The cocharacters $\mu_{i,j}^l$ then define a cocharacter of $\prod_{i\in I}\G^\spl_{i,0}$ over $\kappa$. One can check carefully that it factors through $\G^{\spl}_{0,\kappa}$:
\[\mu= \prod_{i,j,l} \mu_{i,j}^l : \bbG_{m,\kappa} \to \G^\spl_0\otimes_{\F_p}\kappa \subset \prod_{i,j,l} G_{i,j}^l.\]

Fix an $i\in I$.
For each $1\leq j \leq f_i, 1 \leq l \leq e_i$, we also define an $\calO_F$-lattice as
\[ \Lambda_{i,j}^l := \Lambda_i \otimes_{\calO_{F_i},\sigma_{i,j}^l} \calO_F.\]
The pairing on $\Lambda_i$ induces a natural pairing on the $\calO_F$-lattice $\Lambda_{i,j}^l$,
and the similitude automorphism group is isomorphic to $G_{i,j}^l$ over $\calO_F$. In contrast to the rational case,
in general $\calG_{i,\calO_F}$ does not split as a product of $G_{i,j}^l$ due to ramification.
So we define the split $\calO_F$-lattice as
\[\Lambda^{\spl} := \bigoplus_{i \in I} \Lambda_i^{\spl,m_i},\quad \Lambda_i^{\spl} := \bigoplus_{j,l} \Lambda_{i,j}^l,\] here and in the following the indexes $j$ and $l$ (in the second direct sum) run over $1\leq j\leq f_i$ and $1\leq l\leq e_i$. There are naturally induced pairings on $\Lambda^{\spl} $ and $\Lambda_i^{\spl} $.
Let $\G^\spl$ and $\G^\spl_i$ be the corresponding group schemes over $\calO_F$. Then $\G^\spl_i=\prod_{j,l}G_{i,j}^l$, and we have the inclusion as before \[\G^\spl\subset \prod_{i\in I}\G^\spl_i\] by requiring having the same $\Z_p^\times$-local similitude factors.
Over $F$ we have the following isomorphism
\[\Lambda_i^{\spl} \otimes_{\calO_F} F \simeq V_i \otimes_{\bbQ_p} F,\]
compatible with additional structures on both side. Thus $\G^\spl$ is the split reductive model of $G$ over $F$. Note that there is a natural homomorphism \[\G_{i,\calO_F}\ra \G^\spl_i\] which is not an isomorphism in general as remarked above. Therefore, in general we have
\[\G_{\calO_F}\ncong \G^\spl.\]

For each $i\in I, 1\leq j\leq f_i,1\leq l\leq e_i$, set \[\Lambda_{i,j,0}^l=\Lambda_{i,j}^l\otimes_{\calO_F}\kappa \quad \tr{and} \quad \Lambda_{i,0}^l = \bigoplus_{j=1}^{f_i} \Lambda_{i,j,0}^l.\] 
For each $1 \leq l \leq e_i$, we have the natural isomorphism
\[\Lambda_{i,0}^{\rdt} \otimes_{\F_p} \kappa \simeq \Lambda_{i,0}^l = \bigoplus_{j=1}^{f_i} \Lambda_{i,j,0}^l.\] Therefore, \[
\Lambda_{i,0}^{\spl} := \Lambda_{i}^{\spl} \otimes_{\calO_F} \kappa=\bigoplus_{j,l}\Lambda_{i,j,0}^l=\Lambda_{i,0}^{\rdt, e_i}\otimes_{\F_p}\kappa.\]
So $\calG_{i,0}^{\spl} \otimes \kappa$ is the similitude automorphism group associated to $\Lambda_{i,0}^{\spl}$,
and $\mu_i$ is a cocharacter of $\calG_{i,0}^{\spl}$ defined over $\kappa$.
By construction, we have an isomorphism of reductive groups over $\kappa$: \[\G^\spl_0\otimes_{\F_p}\kappa\simeq \G^\spl\otimes_{\calO_F}\kappa.\]
In other words, $\G^\spl_0$ is an $\F_p$ model of the special fiber of $\G^\spl$. Moreover, one can analyze the decomposition of $\Lambda_{i,0}^{\spl}$ induced by $\mu_i$ exactly as in the characteristic zero case (see the end of in \ref{subsub rational}).

\subsection{Smooth Pappas-Rapoport splitting models}
Recall that we have $\mu: \mathbb{G}_{m,\ov{\Q}_p}\ra G_{\ov{\Q}_p}$ the Hodge cocharacter of $G$ over $\ov{\Q}_p$.
Let $E$ be the field of definition of the conjugacy class attached to $\mu$. Then by the assumption on $F$, we have $E\subset F$.
 
Let $K^p$ be a sufficiently small open compact subgroup of $G(\mathbb{A}_f^p)$,
which will be fixed in the rest of this section. By works of Kottwitz \cite{Kottwitz1992} and Rapoport-Zink \cite{RapoportZink1996},
there is a scheme $\mathscr{A}^{\naive}=\mathscr{A}^{\naive}_{K^p}$ over $\calO_E$ representing the following moduli problem (see \cite{RapoportZink1996} chapter 6 for more details):
for any $\calO_E$-scheme $S$, $\mathscr{A}^{\naive}(S)$ classifies the isogeny classes of tuples
$(A,\lambda,\iota,\alpha)$, where
\begin{itemize}
	\item $A /S$ is an abelian scheme,
	\item $\lambda: A \to A^\vee$ is a $\mathbb{Z}_{(p)}^\times$-polarization,
	\item $\iota: \calO_B \to \End_S(A)\otimes_\mathbb{Z} \mathbb{Z}_{(p)}$ is an $\calO_B$-structure of $(A,\lambda)$,
		which is compatible with the Rosatti involution.
	\item $\alpha$ is a $K^p$-level structure.
\end{itemize}
Its generic fiber $\mathscr{A}^{\naive}_E$ is called the rational moduli space with respect to the PEL datum.

For any $\calO_E$-scheme $S$ and a point $\underline{A} \in \mathscr{A}^{\naive}(S)$,
there is an exact sequence of $\calO_S$-modules
\[0 \to \omega_{A/S} \to H_{\dR}^1(A/S) \to \Lie_{A^\vee/S} \to 0,\]
where $\omega_{A/S}$ is the sheaf of invariant differentials.
For simplicity, write $\omega = \omega_{A/S}$, and $\calH = H^1_{\dR}(A/S)$. 
The action of $\calO_B\otimes\Z_p$ on the $\calO_S$-modules $\calH$ and $\omega$ induces decompositions
\[\calH=\bigoplus_{i\in I}\calH_i^{m_i},\quad \calH_i=\bigoplus_{j=1}^{f_i}\calH_{i,j},\quad
\omega = \bigoplus_{i\in I} \omega_i^{m_i},\quad \omega_i = \bigoplus_{j=1}^{f_i} \omega_{i,j},\]
and similarly for $\Lie_{A^\vee/S}$. For each $i, j$ we have an exact sequence of $\calO_S$-modules
\[0 \to \omega_{i,j} \to \H_{i,j} \to \Lie_{A^\vee/S, i, j} \to 0.\]
As $\iota$ is compatible with $\lambda$, there is an isomorphism
\[\mathcal{H}_{i,j} \simeq \mathcal{H}_{i,j}^{\vee}.\]
 $\calH_{i,j}$ is a self-dual $\calO_S$-module of rank $2d_i e_i$. If $i$ is an index of type (AL),  there is a natural decomposition $\calH_{i,j} = \calH'_{i,j} \times \calH_{i,j}^{'\vee}$ such that $\calH_{i,j}'$ is an $\calO_S$-module  and $\calH_{i,j}^{'\vee}:=\Hom(\calH'_{i,j}, \calO_S)$. Moreover, the pairing of $\calH_{i,j}$ is given by the natural pairing of product of dual objects.

For each $i\in I$, let $\pi_i$ be a uniformizer of $\calO_{F_i}$.
Moreover, for each triple $(i,j,l)$, 
set \[\pi_{i,j}^l = \sigma_{i,j}^l(\pi_i).\] For each $i, j, l$,
we define two polynomials as
\[Q_{i,j}^{\leq l}(T) = \prod_{k=1}^l (T-\pi_{i,j}^k)\quad \text{ and }\quad Q_{i,j}^{>l}(T) = \prod_{k=l+1}^{e_i} (T- \pi_{i,j}^k).\]

\begin{definition}[See also Definition \ref{def gen splitting str}]\label{spl-str-def}
Let $S$ be an $\calO_F$-scheme and $\underline{A} = (A,\iota,\lambda,\alpha) \in \mathscr{A}^{\naive}(S)$ with the associated $\H$ and $\omega$ together with the induced additional structure.
	A \emph{splitting structure} with respect to $\underline{A}$ is given by the datum $\ul{\mathscr{F}_\bullet}=(\mathscr{F}_{i,j}^l)$:
	for each $i\in I, 1\leq j\leq f_i$,
	there is a filtration of $\mathcal{O}_S$-modules with $\calO_{F_i}$-action
	\[0 = \mathscr{F}_{i,j}^0 \subset \mathscr{F}_{i,j}^1 \subset \cdots \subset \mathscr{F}_{i,j}^{e_i} = \omega_{i,j},\]
	such that
	\begin{enumerate}
		\item For each $1\leq l\leq e_i$, $\mathscr{F}_{i,j}^l$ and $\mathcal{H}_{i,j} / \mathscr{F}_{i,j}^l$ are  finite locally free $\mathcal{O}_S$-modules.
		\item Write $[\pi_i]$ for the action of $\pi_i \in \calO_{F_i}$. Then for all $1\leq l \leq e_i$, we have
			\[([\pi_i] - \pi_{i,j}^l)\cdot \mathscr{F}_{i,j}^l \subset \mathscr{F}_{i,j}^{l-1},\]
			i.e. $(\pi_i \otimes 1 - 1 \otimes \sigma_{i,j}^l (\pi_i)) \cdot \mathscr{F}_{i,j}^l \subset \mathscr{F}_{i,j}^{l-1}$.
		\item The $\mathcal{O}_S$-module $\mathscr{F}_{i,j}^l / \mathscr{F}_{i,j}^{l-1}$ is locally free of rank $d_{i,j}^l$ for all $1 \leq l \leq e_i$.
		\item For each $1 \leq l \leq e_i$, 
			let $\mathscr{F}_{i,j}^{e_i+l} = {Q_{i,j}^{>e_i-l}(\pi_i)}^{-1}\mathscr{F}_{i,j}^{e_i-l}$. Then we require that
			\[\mathscr{F}_{i,j}^{e_i+l} =\mathscr{F}_{i,j}^{e_i-l, \perp}\quad \tr{and}\quad ([\pi_i] - \pi_{i, j}^{e_i}) \cdots ([\pi_i] - \pi_{i, j}^{e_i+l-1}) \mathscr{F}_{i, j}^{e_i+l} \subset \mathscr{F}_{i, j}^{e_i-l}.\]
			Here $\mathscr{F}_{i,j}^{e_i-l, \perp}$ is the orthogonal complement of $\mathscr{F}_{i,j}^{e_i-l}$ in $\calH_{i,j}$ under the pairing on $\calH_{i,j}$ induced by the polarization $\lambda$.
	\end{enumerate}
\end{definition}

More explicitly, we have the following description of splitting structures according to the type of $i$ (see the notations in \ref{subsub rational}; see also \cite{BijakowskiHernandez2022} subsections 2.3 and 2.4):
\begin{itemize}
  \item (C): We have $R_i = F_i$ and for each index $i, j, l$, $d_{i, j}^l = d_i$. One has a filtration
    \[0 = \mathscr{F}_{i, j}^0 \subset \mathscr{F}_{i, j}^1 \subset \cdots \subset \mathscr{F}_{i, j}^{e_i} = \omega_{i, j} \subset \calH_{i, j}\]
    where for each $1\leq l\leq e_i$, $\scrF_{i, j}^l$ is a locally a direct factor of rank $d_i l$, and the filtration satisfies the relations as in the above conditions (2) and (4).
    
  \item (AL): We have $R_i = F_i \times F_i$, and $\calH_{i} = \calH_i' \oplus \calH_i^{'\vee}$. For each $j$, the splitting structure is reduced to a filtration
  \[0 = \mathscr{F}_{i, j}^0 \subset \mathscr{F}_{i, j}^1 \subset \cdots \subset \mathscr{F}_{i, j}^{e_i} = \omega_{i, j}' \subset \calH_{i, j}',\]
  where each $\scrF_{i, j}^l$ is locally a direct factor of rank $a_{i, j}^1 + \cdots + a_{i, j}^l$, and the filtration satisfies the above condition (2).
  
  \item (AU): We have a quadratic unramified field extension $R_i|F_i$. For each $j$, we have a further decomposition
  \[\omega_{i, j} = \omega_{i, j}' \oplus (\omega_{i, j}')^\vee\]
  given by the unramified $\calO_{R_i}$-action, and the splitting structure is reduced to a filtration
  \[ 0 = \mathscr{F}_{i, j}^0 \subset \mathscr{F}_{i, j}^1 \subset \cdots \subset \mathscr{F}_{i, j}^{e_i} = \omega_{i, j}' \subset \calH_{i, j},\]
  where each $\scrF_{i, j}^l$ is locally a direct factor of rank $a_{i, j}^1 + \cdots + a_{i, j}^l$, and the filtration satisfies the above condition (2).
\end{itemize}

Finally, we can give the definition of the splitting model over $\mathscr{A}^{\naive}$, cf. \cite{PappasRapoport2005}.
\begin{definition}
	The splitting model $\mathscr{A}^{\spl}$ over $\calO_F$ is the scheme which represents the following moduli problem:
	for any scheme $S$ over $\calO_F$, $\mathscr{A}^{\spl}(S)$ is the isomorphism classes of tuples 
	$(\underline{A},\underline{\mathscr{F}_\bullet})$, where
	\begin{itemize}
		\item $\underline{A} = (A, \lambda, \iota, \alpha)$ is an $S$-point of $\mathscr{A}^{\naive}$.
		\item $\underline{\mathscr{F}_\bullet} = (\mathscr{F}_{i,j}^l) $ is a splitting structure of $\underline{A}$.
		\item The isomorphism is given by $\mathbb{Z}_{(p)}^\times$-isogenies between abelian schemes with $\calO_B$-structures and splitting structure.
	\end{itemize}
\end{definition}
Let $A^{\tr{univ}}$ be the universal abelian scheme over $\mathscr{A}^{\naive}$ with the associated $\H$ and $\omega$.
By the notation of \cite{Lan18}, we have
\[ \mathscr{A}^{\spl}=\mathrm{Spl}^+_{(\H,\omega,\iota)/\mathscr{A}^{\naive}\otimes \calO_F}. \]
By \cite{PappasRapoport2005} section 15 and \cite{Lan18} Lemma 2.3.9, 
there is a canonical isomorphism of schemes over $F$
\[\scrA^{\spl} \otimes_{\calO_F} F \simeq \mathscr{A}^{\naive} \otimes_{\calO_E} F.\]

\begin{proposition}\label{prop smooth spl}
	Assume that each $i\in I$ has type either (C), or (AL), or (AU), then the splitting model $\scrA^\spl$ is smooth over $\calO_F$.
\end{proposition}
\begin{proof}
	Let $\widetilde{\mathscr{A}^{\spl}}$ be the scheme over $\calO_F$ such that for any $\calO_F$-scheme $S$
	\[\widetilde{\mathscr{A}^{\spl}}(S) = \{(\underline{A},\underline{\mathscr{F}_\bullet}, \underline{\tau} = \{\tau_{i,j}^l\})\},\]
	where $(\underline{A},\underline{\scrF_\bullet})$ is an $S$-point of $\Spl{\mathscr{A}}$ and for each $i\in I, 1\leq j\leq  f_i$, $\underline{\tau}$ is a collection of isomorphisms
	\[\tau_{i,j}^l: \Ker({(\pi_i-\sigma_{i,j}^{l}(\pi_i)})|_{\mathcal{H}_{i,j} / \mathscr{F}_{i,j}^{l-1}}) \simeq \Lambda_{i,j}^l \otimes_{\calO_F} \mathcal{O}_S.\]
	By Propositions 5.2 and 9.2 of \cite{PappasRapoport2005}, when $i$ is of type (AL) or type (C), such an isomorphism exists locally.
	Moreover, when $i$ is of type (AU), the same proof  equally applies.

	Let $\mathbb{M}_{i,j}^l = \mathbb{M}^{\rm loc}(G_{i,j}^l, \mu_{i,j}^l)$ be the unramified local models over $\calO_F$.
	Then we have the following local model diagram 
	(which is a special case of the local model diagram in Proposition \ref{prop loc mod Gspl}):
	\[\xymatrix{
		&{\widetilde{\mathscr{A}^{\spl}}} \ar[ld]_{\pi} \ar[rd]^{q} \\
		{\mathscr{A}^{\spl}} & & {\prod_{i, j, l} \mathbb{M}^{l}_{i,j}},
	}\]
	where the index $(i,j,l)$ runs through $1 \leq i \leq r, 1 \leq j \leq f_i, 1 \leq l \leq e_i $, the morphism $q$ is smooth and given by
	\[(\underline{A}, \underline{\mathscr{F}_\bullet}, \underline{\tau}) \mapsto (\tau_{i,j}^l(\mathscr{F}_{i,j}^l / \mathscr{F}_{i,j}^{l-1}))_{i\in  I,1\leq j\leq  f_i}^{1\leq l \leq e_i},\]
	which is $\prod_{i, j, l} G_{i,j}^l$-equivariant.
	The morphism $\pi$ is the natural forgetful morphism, which is a $\prod_{i, j, l} G_{i,j}^l$-torsor.

	We have excluded the index of type (AR), so the right hand side of the local model diagram is a product of unramified local models with hyperspecial level, 
	which is known to be smooth. This shows that $\Spl{\mathscr{A}}$ is smooth over $\calO_F$.              
\end{proof}

Let $\scrA$ be the scheme theoretic image of the natural morphism
	\[\mathscr{A}^{\spl} \to \scrA^{\naive} \otimes_{\calO_E} \calO_F \to \mathscr{A}^{\naive}.\]Then $\scrA$ is flat over $\calO_E$. Moreover, $\scrA$ admits a local model diagram, cf. Proposition \ref{Prop loc mod diag}. Thus under the condition of Proposition \ref{prop smooth spl}, the morphism
	\[\scrA^\spl\lra\scrA\] is a resolution of singularities. Note that in this case, the associated parahoric subgroup $K_p=\calG(\Z_p)$ is very special in the sense of \cite{Zhu} Definition 6.1. 
	The purpose of this paper is to show that in our ramified setting, the smooth splitting model $\scrA^\spl$ admits many nice properties as in the unramified setting.
	
	Recall that the Hodge cocharacter $\mu: \mathbb{G}_{m, F}\ra G_F$ defines a parabolic subgroup $P_\mu=\{g\in G_F\,|\,\lim_{t\ra 0}\mu(t)g\mu(t)^{-1}\,\tr{exists}\}$ over $F$ in the usual way.  For any $(V,\eta)\in \Rep_FP_\mu$, we have the associated automorphic vector bundle $\V$ over \[\scrA^{\spl} \otimes_{\calO_F} F \simeq \mathscr{A}^{\naive} \otimes_{\calO_E} F\simeq \scrA\otimes_{\calO_E}F,\] cf. \cite{Milne90}. Since $G$ is split over $F$, we have a canonical reductive model $\calO_F$, which by abuse of notation we still denote by $G$ (the precise meaning of the notation $G$ will be clear from the context; sometimes as in the last subsection we denote it by $\calG^\spl$). Moreover,  the parabolic $P_\mu$ extends to a parabolic subgroup over $\calO_F$, which by abuse of notation we still denote by $P_\mu$, and $(V,\eta)$ naturally extends to a representation of $P_\mu$ over $\calO_F$. In general, it is hard to extend $\V$ to a vector bundle on $\scrA_{\calO_F}$. Nevertheless, we have
	\begin{corollary}\label{coro ext vb}
	$\V$ extends canonically to a vector bundle on $\scrA^\spl$.
	\end{corollary}
\begin{proof}
Note that the integral flag variety over $\calO_F$ is $\Fl(G,\mu)=G/P_\mu	=\prod_{i,j,l}\mathbb{M}_{i,j}^l $. The diagram in the proof of the above proposition is the diagram of schemes over $\calO_F$
\[\xymatrix{
& \wt{\scrA^{\spl}}\ar[rd]^q \ar[ld]_\pi&\\
\scrA^{\spl}& &\Fl(G,\mu).
} \]
Then we can construct an integral canonical model of $\V$ over $\scrA^\spl$ by using this diagram as in \cite{Milne90}.
\end{proof}

By abuse of notation, we still denote by $\V$ the integral automorphic vector bundle over $\scrA^\spl$. We discuss the example of the standard representation $\Lambda^\spl$. Let $(A^\tr{univ},\lambda,\iota)$ be the universal abelian scheme over $\scrA^\spl$ and $\H=H^1_{\dR}(A^{\tr{univ}}/\scrA^\spl)$. Then as before we have decompositions $\H=\bigoplus_i\H_i^{m_i}, \H_i=\bigoplus_j\H_{i,j}$. We have also the universal splitting structure $\ul{\scrF_\bullet}$ on $\H$.
For each $i,j,l$, we define
\[\mathcal{M}_{i,j}^l=\Ker((\pi_i\otimes 1-1\otimes\sigma_{i,j}^l(\pi_i))|\H_{i,j}/\scrF_{i,j}^{l-1}).\]
Then the automorphic vector bundle over $\scrA^\spl$ associated to $\Lambda^\spl$ is \[\mathcal{M}=\bigoplus_i\mathcal{M}_i^{m_i},\quad \mathcal{M}_i=\bigoplus_{j,l}\mathcal{M}_{i,j}^l.\]

\section{$F$-zips with additional structure and Ekedahl-Oort stratification}\label{section Fzip}

In this section, we keep our assumption that there is no index of type (AR), 
so we have a smooth splitting integral model $\Spl{\mathscr{A}} / \calO_F$.
We will construct a universal $F$-zip with additional structure over the special fiber $\Spl{\mathscr{A}}_0$ of $\Spl{\mathscr{A}}$. Then we study some basic properties of the induced Ekedahl-Oort stratification on $\Spl{\mathscr{A}}_0$.

\subsection{$F$-zips and $G$-zips}\label{subsec G-zip}
We first recall the notion of $F$-zips (\cite{MW}). For any $\F_p$-scheme $S$ and any object $M$ over $S$, we write $M^{(p)}$ for the pullback of $M$ under the absolute Frobenius of $S$. 
\begin{definition}
	Let $S$ be a scheme over $\F_p$ and $\sigma: S \to S$ is the absolute Frobenius of $S$. An $F$-zip over $S$ is a tuple $\underline{M} = (M, C^\bullet, D_\bullet, \varphi_\bullet)$ where
	\begin{itemize}
		\item $M$ is a locally free sheaf of finite rank on $S$;
		\item $C^\bullet = (C^i)_{i \in \Z}$ is a descending filtration on $M$ and each $C^i$ is locally free $\calO_S$-module;
		\item $D_\bullet = (D_i)_{i \in \Z}$ is an ascending filtration on $M$ and each $D_i$ is locally free $\calO_S$-module;
		\item $\varphi_\bullet = (\varphi_i)_{i \in \Z}$ and for each $i$, $\varphi_i: C^i /C^{i+1} \to D_i / D_{i-1}$ is a $\sigma$-linear map whose linearization
		\[\varphi_i^{\rm lin}: (C^i / C^{i+1})^{(p)} \to D_i / D_{i-1}\]
		is an isomorphism.
	\end{itemize}
\end{definition}

Next we briefly review the notion of $G$-zips and the theory of $G$-zip stacks (\cite{PinkWedhornZiegler2011,PinkWedhornZiegler2015}) for later use in this section.
Let $G$ be a connected reductive group over $\mathbb{F}_p$, and $\chi$ a cocharacter of $G$ defined over a finite field $\kappa|\mathbb{F}_p$.
Let $P_+$ (resp. $P_-$) be the parabolic subgroup of $G_\kappa$ such that its Lie algebra is the sum of spaces with non-negative weights (resp. non-negative weights) in $\Lie(G_\kappa)$ under $\Ad \circ \chi$.
We will also write $U_+$ (resp. $U_-$) for the unipotent radical of $P_+$ (resp. $P_-$).
Let $L$ be the common Levi subgroup of $P_+$ and $P_-$.
\begin{definition}
	Let $S$ be a scheme over $\kappa$.
	\begin{enumerate}
		\item A $G$-zip of type $\chi$ over $S$ is a tuple $\underline{I} = (I,I_+,I_-,\iota)$ consisting of
		\begin{itemize}
			\item a right $G_\kappa$-torsor $I$ over $S$,
			\item a right $P_+$-torsor $I_+ \subset I$,
			\item a right $P_{-}^{(p)}$-torsor $I_{-} \subset I$,
			\item an isomorphism of $L^{(p)}$-torsors $\iota: I_{+}^{(p)} / U_+^{(p)} \to I_{-} / U_-^{(p)} $.
		\end{itemize}
		\item A morphism $(I,I_+,I_-,\iota) \to (I',I'_+,I'_-,\iota')$ of $G$-zips of type $\chi$ over $S$ consists of 
			equivariant morphisms $I \to I'$ and $I_\pm \to I'_\pm$ that are compatible with inclusions and the isomorphism $\iota$ and $\iota'$.
	\end{enumerate}
\end{definition}

One can prove that the category of $F$-zips of rank $n$ is equivalent to the category of $\GL_n$-zips, see \cite{PinkWedhornZiegler2015} 8A. More generally, for a classical group $G$, a $G$-zip is equivalent to the $F$-zip associated to its natural faithful representation, together with the additional structure corresponding to the linear algebraic data defining the group, cf. \cite{PinkWedhornZiegler2015} section 8. %Moreover we have a standard example of $G$-zips as following:
%\begin{example}\label{zip-example}
 %Let $(G,\mu)$ be a cocharacter datum over $\F_p$, i.e. $G$ is a connected reductive group over $\F_p$ and $\mu$ is a cocharacter of $G$ defined over a finite field $\kappa$. Fix a faithful representation $V$ of $G$ over $\F_p$, then the cocharacter $\mu$ induce a natural descending and ascending filtration $C^\bullet, D_\bullet$ of $V_\kappa$ (given by the weights). Since there is a canonical isomorphism $V_\kappa^{(p)} \simeq V_\kappa$, we have natural isomorphisms  $\varphi_\bullet: (\gr _C^\bullet V)^{(p)} \simeq \gr _C^\bullet V \simeq \gr_\bullet^D V$, and a $F$-zip $(V_\kappa, C^\bullet, D_\bullet, \varphi_\bullet)$
%\end{example}

The category of $G$-zips of type $\chi$ over $S$ will be denoted by $\Zip{G}{\chi}_\kappa(S)$.
This defines a category fibered in groupoids $\Zip{G}{\chi}_\kappa$ over $\kappa$.
\begin{theorem}[\cite{PinkWedhornZiegler2015}]
	The fibered category $\Zip{G}{\chi}_\kappa$ is a smooth algebraic stack of dimension 0 over $\kappa$.
\end{theorem}

Denote by $\Frob_p:L \to L^{(p)}, \, l\mapsto l^{(p)}$ the relative Frobenius of $L$, and define $E_{G,\chi}$ by the fiber product
\[\xymatrix{
E_{G,\chi} \ar[rr] \ar[d] & & P_-^{(p)} \ar[d] \\
P_+ \ar[r] & L \ar[r]^{\Frob_p} & L^{(p)}
}\]
Then we have
\[E_{G,\chi} = \{(p_+ := l u_+, p_- := l^{(p)} u_-) \,|\, l \in L, u_+ \in U_+, u_- \in U_-^{(p)}\}.\]
It acts on $G_\kappa$ from the left-hand side as follows: for $(p_+, p_-) \in E_{G,\chi}(S)$ and $g \in G_\kappa(S)$,
set $(p_+, p_-) \cdot g := p_+ g p_-^{-1}$.
\begin{theorem}[\cite{PinkWedhornZiegler2015}]\label{thm quotient zip stack}
	We have an isomorphism of algebraic stacks $[E_{G,\chi} \backslash G_\kappa] \simeq \Zip{G}{\chi}_\kappa$.
\end{theorem}

Let $B \subset G$ be a Borel subgroup and $T \subset B$ a maximal torus.
Let $W := W(B,T)$ be the absolute Weyl group, and $I := I(B,T)$ the set of simple reflections defined by $B$.
Let $J \subset I$ be the simple roots whose inverse are roots of $P_+$.
Let $W_J$ be the subgroup of $W$ generated by $J$, and $^J W$ the set of elements $w$ such that $w$ is the element of minimal length in some coset $W_J w'$. By \cite{PinkWedhornZiegler2011} section 6, there is a partial order $\preceq$ on $^J W$. Let $k=\ov{\kappa}$ be an algebraic closure of $\kappa$, and $\varphi: W\ra W$ the isomorphism induced from the Frobenius of $G$. There is a distinguished element $x\in W$ satisfying certain technical conditions, see \cite{PinkWedhornZiegler2011} 12.2.
\begin{theorem}[\cite{PinkWedhornZiegler2011}]\label{thm PWZ11}
	For $w \in {}^J W$, and $(B', T')$ a Borel pair of $G_k$ such that $T' \subset L_k$ and $B' \subset P_{-,k}^{(p)}$,
	let $g, \dot{w} \in N_{G_k}(T')$ be the representative of $\varphi^{-1}(x)$ and $w$ respectively,
	and $G_w \subset G_k$ the $E_{G,\chi}$-orbit of $g B' \dot{w} B'$. Then
	\begin{enumerate}
		\item The orbit $G_w$ does not depend on the choices of $\dot{w}, T', B'$ or $g$.
		\item The orbit $G_w$ is a locally closed smooth subvariety of $G_k$. 
			Its dimension is $\dim(P)+l(w)$. Moreover, $G_w$ consists of only one $E_{G,\chi}$-orbit.
		\item Denote by $|[E_{G,\chi} \backslash G_\kappa] \otimes k|$ the topological space of $[E_{G,\chi} \backslash G_\kappa] \otimes k$,
		  and $^J W$ the topological space induced by the partial order $\preceq$.
			Then the association $w \mapsto G_w$ induces a homeomorphism $^J W \simeq |[E_{G,\chi} \backslash G_\kappa]|$.
	\end{enumerate}
\end{theorem}

In the following, we will apply the above results to the pair $(G,\chi)=(\G_0^\spl,\mu)$ constructed in subsection \ref{subsec PEL data}.

\subsection{Construction of the universal $\calG^\spl_0$-zip on $\scrA^\spl_0$}\label{subsection zip}

Let $\kappa$ be the residue field of $\calO_F$ and $\sigma$ the absolute Frobenius of $\kappa$.
We will construct a universal $\calG^\spl_0$-zip of type $\mu$ over the special fiber \[\scrA^\spl_0= \Spl{\mathscr{A}} \otimes_{\calO_F} \kappa\] of $\Spl{\mathscr{A}}$.
As in the unramified case, we will first construct a vector bundle $\mathcal{M}$ over $\scrA^\spl_0$, 
which can be viewed as a semi-simplification of the de Rham cohomology of the universal abelian scheme for the additional structure.
The construction of such a vector bundle was already mentioned below \cite{ReduzziXiao2017}*{Corollary~2.10} for Hilbert modular varieties.

For any $\kappa$-scheme $S$,
$\Spl{\mathscr{A}_0}(S)$ classifies the isogeny classes of $(\underline{A}, \underline{\mathscr{F}_\bullet})$, where
\begin{enumerate}
    \item $\underline{A}=(A, \lambda, \iota,\alpha)$ is an $S$-point of $\mathscr{A}^{\naive}$.
    \item $\underline{\mathscr{F}_\bullet}=(\mathscr{F}_{i,j}^l)$ is a splitting structure of $\calO_F \otimes_{\Z_p} \calO_S$-module $\omega_{A/S}$.
          This means: if we write $\mathcal{H} = \mathcal{H}_{\dR}^1(A/S)$, it is an $\calO_F \otimes \calO_S$-module,
          hence has a decomposition \[\mathcal{H} = \bigoplus_{i\in I} \bigoplus_{j=1}^{f_i} \mathcal{H}_{i,j}^{m_i},\]
          with $\mathcal{H}_{i,j}$ a locally free $\calO_S$-module equipped with a pairing $\H_{i,j}\times\H_{i,j}\ra\calO_S$. Similarly we have a decomposition $\omega=\bigoplus_{i\in I}\bigoplus_{j=1}^{f_i}\omega_{i,j}^{m_i}$.
          For each $i\in I,1\leq  j\leq f_i$, the splitting structure is a filtration of locally direct $\calO_S$-factors of $\omega_{i,j}$:
          \[0 = \mathscr{F}_{i,j}^0 \subset \mathscr{F}_{i,j}^1 \subset \cdots \subset \mathscr{F}_{i,j}^{e_i} = {\omega_{i,j}} \subset \mathcal{H}_{i,j}\]
          with $\calO_{F_i}/ (p) \simeq \kappa_i[T]/(T^{e_i}) = \kappa_i[\varepsilon_i]$-action, such that
          \begin{enumerate}
              \item $\kappa_i$ acts on $\mathscr{F}_{i,j}^l$ by $\sigma_{i,j}: \mathcal{O}_{\Unram{F_i}} \to \kappa$.
              \item For each $1 \leq l \leq e_i$, $\varepsilon_i \mathscr{F}_{i,j}^l \subset \mathscr{F}_{i,j}^{l-1}$.
              \item For each $1 \leq l \leq e_i$, $\mathscr{F}_{i,j}^l / \mathscr{F}_{i,j}^{l-1}$ is locally free of rank $d_{i,j}^l$.
              \item For each $1 \leq l \leq e_i$, $\mathscr{F}_{i,j}^{l,\perp}=(\varepsilon_i^{e_i-l})^{-1}\mathscr{F}_{i,j}^l$. 
          \end{enumerate}
\end{enumerate}

Let $S=\Spl{\mathscr{A}_0}$ and $\ul{A}$ the universal abelian scheme over $S$. For each triple $i,j,l$,
we define
\[\mathcal{M}_{i,j}^l := \varepsilon_i^{-1}\mathscr{F}_{i,j}^{l-1} / \mathscr{F}_{i,j}^{l-1}.\]
Note that each $\mathcal{M}_{i,j}^l$ is a locally free $\calO_S$-module of rank $2d_i$. 
As in the description below Definition~\ref{spl-str-def}, in fact we have some reduced description for $\mathcal{M}_{i,j}^l$ according to the type of $i$, which we leave to the reader.
 Now we write
\[\mathcal{M} = \bigoplus_{i\in I} \calM_i^{m_i},\quad \calM_i= \bigoplus_{l=1}^{e_i} \bigoplus_{j=1}^{f_i} \mathcal{M}_{i,j}^l,\]
then each $\mathcal{M}_{i,j}^l$ is locally isomorphic to $\Lambda_{i,j}^l \otimes_{\calO_F} \calO_S$.
First fix $i \in I$ and $1\leq j \leq f_i$. For each $1 \leq l \leq e_i$,
we have two natural maps:
\begin{enumerate}
	\item If $l \neq 1$, we have a natural map: \[V_{i,j}^l: \mathcal{M}_{i,j}^l \to \mathcal{M}_{i,j}^{l-1},\quad x \mapsto \varepsilon_i x,\] 
		which is an $\calO_S$-linear morphism. On the other side, the injection $\varepsilon_i^{-1}\mathscr{F}_{i,j}^{l-2} \hookrightarrow \varepsilon_i^{-1}\mathscr{F}_{i,j}^{l-1}$ induces an $\calO_S$-linear morphism
		\[F_{i,j}^l: \mathcal{M}_{i,j}^{l-1} \to \mathcal{M}_{i,j}^{l}.\]
	\item If $l=1$, let $V_{i,j}: \mathcal{H}_{i,j} \to \mathcal{H}_{i,j-1}$ be the Verschiebung morphism and $F_{i,j}: \mathcal{H}_{i,j-1} \to \mathcal{H}_{i,j}$ the Frobenius morphism.
    	  The map $V_{i,j}$ induces a natural map 
		\[V_{i,j}^1:\mathcal{M}_{i,j}^1 \to \mathcal{M}_{i,j-1}^{e_i},\quad x \mapsto V_{i,j}(\varepsilon_i^{1-e_i}(x)),\] which is $\sigma^{-1}$-linear. 
		Similarly, $F_{i,j}$ induces a $\sigma$-linear map \[F_{i,j}^1: \mathcal{M}_{i,j-1}^{e_i} \to \mathcal{M}_{i,j}^1,\quad x \mapsto F_{i,j}(x).\]
\end{enumerate}
Now for each $1\leq l\leq e_i$, define
\[\mathcal{M}_i^l := \bigoplus_{j=1}^{f_i}\mathcal{M}_{i,j}^l, \quad \Lambda_i^l := \bigoplus_{j=1}^{f_i} \Lambda_{i,j}^l.\]
Each $\calM_{i}^l$ is a locally free $\calO_S[\varepsilon_i]$-module with a $\kappa_i$-action, and we have locally an isomorphism
\[\calM_{i}^l \simeq \Lambda_i^l \otimes_{\calO_F} \calO_S.\]
(Note that the lattice $\Lambda_{i,j}^l$ is denoted by $\Xi_{i,j}^l$ in \cite{PappasRapoport2005}*{Proposition~5.2}.)
We also have locally an isomorphism
\[\calM_{i} = \bigoplus_{l=1}^{e_i} \calM_{i}^l \simeq \bigoplus_{l=1 }^{e_i} \Lambda_{i,0}^l\otimes\calO_S= \Lambda_{i,0}^{\spl}\otimes\calO_S.\]
Each $\calM_i$ is a locally free $\calO_S$-module of rank $2e_i f_i d_i$.
For each $i, l$, let $F_i^l=\bigoplus_j F_{i,j}^l$ and $V_i^l=\bigoplus_j V_{i,j}^l$.
We have constructed the following linear morphisms:
\[\xymatrix{
	{\mathcal{M}_i^{1}} \ar@<1ex>[r]^{F_i^2} & {\mathcal{M}_i^{2}} \ar@<1ex>[l]^{V_i^2} \ar@<1ex>[r]^{F_i^3} & \cdots \ar@<1ex>[l]^{V_i^3} \ar@<1ex>[r]^{F_i^{e_i}} & {\mathcal{M}_i^{e_i}} \ar@<1ex>[l]^{V_i^{e_i}}.
}\]
and semi-linear morphisms ($F_i^{1}$ is $\sigma$-linear and $V_i^{1}$ is $\sigma^{-1}$-linear)
\[\xymatrix{
	{\mathcal{M}_i^{e_i}} \ar@<1ex>[r]^{F_i^1} & {\mathcal{M}_i^1} \ar@<1ex>[l]^{V_i^1}.
}\]

\begin{lemma}\label{lem l=1}
	There is a canonical isomorphism $g_i: \calM^1_{i} \to \calM^1_i$ sending $\Ker(V_i^1)$ to $\Image(F_i^1)$.
\end{lemma}
\begin{proof} We can reduce to the case $S=\Spec\,R$ for a ring $R$ over $\kappa$.
	For $e_i=1$, we have $F_i^1 = \Frob, V_i^1 = \Ver$, so we can assume that $e_i \geq 2$.
	Note that
	\[\Ker(V_{i,j}^1) = \varepsilon_i^{e_i-1} \Ver^{-1}(\scrF_{i,j}^{e_i-1}), \quad \Image(F_{i,j}^1) = \Frob(\scrF_{i,j}^{e_i-1}).\]
	For each $j$, fix an isomorphism $\calH_{i,j} \simeq R[\varepsilon_i]^{2d_i}$ and lift it to $\widetilde{\calH}_{i,j} = R[[t]]^{2d_i}$ such that $\widetilde{\calH}_{i,j} / (t^{e_i}) = \calH_{i,j}$.
	For each $\scrF_{i,j}^l \subset \calH_{i,j}$, let $\widetilde{\scrF}_{i,j}^l$ be its lifting in $\widetilde{\calH}_{i,j}$.
	One can lift the Frobenius and Verschiebung morphisms of $\calH_i$ to injections of $\widetilde{\calH}_{i}$, such that
	\[\Ver(\widetilde{\calH}_{i,j}) = \widetilde{\scrF}_{i,j-1}^{e_i}, \quad \Frob(\widetilde{\scrF}_{i,j-1}^{e_i}) = t^{e_i}\widetilde{\calH}_{i,j}.\]
	This gives the following isomorphism:
	\[\xymatrix{
		t^{e_i-1} \widetilde{\calH}_{i,j} \ar[r]^{\Ver}_{\simeq} & t^{e_i-1} \widetilde{\scrF}_{i,j-1}^{e_i} \ar[r]^{\Frob}_{\simeq} & t^{e_i-1} t^{e_i} \widetilde{\calH}_{i,j} \ar[r]^{\cdot t^{-e_i}}_{\simeq} & t^{e_i-1} \widetilde{\calH}_{i,j}.
		}\]
	Such an isomorphism sends the lifting of $\Ker(V_{i,j}^1)$ to the lifting of $\Image(F_{i,j}^1)$.
	After modding out by $(t^{e_i})$ of the above isomorphism, we get an isomorphism
	\[g_{i}: \calM_i^1 \to \calM_i^1\]
	sending $\Ker(V_{i,j}^1)$ to $\Image(F_{i,j}^1)$.
\end{proof}

\begin{lemma}\label{lem l big}
For each integer $2\leq l\leq e_i$, we have the following identities:
\[\Image(F_i^l) = \Ker(V_i^l), \quad \Ker(F_i^l) = \Image(V_i^l).\]
Moreover, we have
\[\rank(\Ker(V_{i,j}^l)) = d_{i,j}^l\quad \tr{and }\quad \rank(\Ker(F_{i,j}^l)) = d_i - d_{i,j}^l.\]
\end{lemma}
\begin{proof}
By our definition, for each $l \geq 2$, we have
\[
    \Ker(F_{i,j}^l) = \Image(V_{i,j}^l) = \scrF_{i,j}^l, \quad \Image(F_{i,j}^l) = \Ker(V_{i,j}^l) = \varepsilon_i^{-1} \scrF_{i,j}^{l-2}.
\]
The rank of $\scrF_{i,j}^l / \scrF_{i,j}^{l-1}$ is $d_{i,j}^l$ by the definition of splitting structures.
\end{proof}
Recall that $\calM_{i,j}^l = \varepsilon_i^{-1} \scrF_{i,j}^{l-1} / \scrF_{i,j}^{l-1}\simeq \scrF_{i,j}^{2e_i-l+1}/\varepsilon_i\scrF_{i,j}^{2e_i-l+1}$. The symplectic pairing on $\scrF_{i,j}^{2e_i-l+1}\subset \calH_{i,j}$ 
induces a pairing on $\calM_{i,j}^l$.
%\subset \calH_{i,j} / \scrF_{i,j}^{l-1}$.
%Note that we have
%\[\varepsilon_i^{-1} \scrF_{i,j}^{l-1} \subset \varepsilon_i^{-(e_i-l)} \scrF_{i,j}^{l-1} = \scrF_{i,j}^{l-1, \perp}.\]
%As the symplectic pairing on $\calH_{i,j}$ induces a pairing on $\scrF_{i,j}^{l-1, \perp} / \scrF_{i,j}^{l-1}$,
%so we get a well-defined pairing on $\calM_{i,j}^l$ by restriction.

%Although we have construct a tuple $(\calM_i, \Ker(F_i), \Ker(V_i))$,
%our morphism $F_i$ and $V_i$ are not $\sigma$-linear in general,
%so this tuple is not an $F$-zip in general.
%We will compare this datum with a standard $F$-zip as follows.

Recall the group $\calG_0^{\spl}$ constructed in \ref{subsubsec calG} as a similitude group  with respect to the vector space $\Lambda_{0}^{\spl}$ with its symplectic form $\psi$.
Recall that we have introduced a cocharacter
\[\mu = \prod_{i, j, l} \mu_{i, j}^{l}: \bbG_{m,\kappa} \to \mathcal{G}^{\spl}_{0,\kappa}.\]
%Fix an index $i$, and let $\mu_i: \bbG_{m,\kappa} \to \calG_{i,\kappa}^{\spl}$ be the cocharacter associated to $\calG_{i,0}^{\spl}$.
Such a datum gives the parabolic subgroups $P_+, P_-$ of $\calG_{\kappa}^{\spl}$ and corresponding unipotent groups $U_+, U_-$,
and common Levi subgroup $L = P_+ \cap P_-$ as in previous subsection.
The cocharacter $\mu$ induces a standard $F$-zip
\[(\Lambda_{0}^{\spl}, C_0, D_0, \varphi_{0,\bullet}).\]
Moreover,  the above construction gives an $F$-zip
\[(\calM = \bigoplus_i \calM_i^{m_i}, C = \bigoplus_i \Ker(F_i)^{m_i}, D = \bigoplus_i \Ker(V_i)^{m_i}, \varphi_\bullet)\]
over $\scrA_0^{\spl}$, equipped with a natural symplectic form $\psi$ and an $\calO_B$-action with the natural isomorphisms
\[\varphi_\bullet:  \calM / C \simeq D^\phi, \quad C \simeq (\calM / D)^\phi,\]
where $\phi = \prod_i \phi_i: \calG_0^{\spl} \to \calG_0^{\spl}$ is the group isomorphism given by 
	\[\phi_i: \calG_{i,0}^{\spl} \to \calG_{i,0}^{\spl}, \quad (x_1, \dots, x_{e_i}) \mapsto (x_2, \dots, x_{e_i}, x_1^p),\]
and $D^\phi$ is the same module $D$ but with $\calG_0^{\spl}$-action twisted by $\phi$. The sub-vector bundles $C$ and $D$ are totally isotropic with respect to the symplectic form $\psi$ on $\calM$.
\begin{proposition}\label{prop splGzip}
We have
	\begin{enumerate}
		\item $I := \Isom((\calM, \psi), (\Lambda_{0}^{\spl}, \psi))$ is a $\calG_{0}^{\spl}$-torsor over $\scrA_0^{\spl}$.
		\item $I_{+} := \Isom(\calM \supset C \supset 0, \Lambda_{0}^{\spl} \supset C_0 \supset 0)$ is a $P_+$-torsor over $\scrA_0^{\spl}$.
		\item $I_{-} := \Isom(\calM\supset D \supset 0, \Lambda_{0}^{\spl} \supset D_0 \supset 0)$ is a $P_-^{\phi}$-torsor over $\scrA_0^{\spl}$.
		\item We have an $L^{\phi}$-equivariant isomorphism $\iota: (I_+ / U_+)^\phi \simeq I_- / U_-^\phi$. 
	\end{enumerate}
	In the above, all the isomorphisms preserve the natural additional structure.
\end{proposition}
\begin{proof}
  By the proofs of Propositions \ref{prop smooth spl} and \ref{prop loc mod Gspl}, there is a local model diagram
  \[\xymatrix{
  	& \widetilde{\scrA}^{\spl}_0 \ar[ld]_{\pi} \ar[rd]^{q} \\
  	\scrA^{\spl}_0 && M^{\loc},
  }\]
  where $M^{\loc}$ is the special fiber of $\mathbb{M}^\loc(\G^\spl,\mu)$,
  $q$ sends $(x = (\underline{A},\underline{\scrF}_\bullet), \tau:\calM_x \simeq \Lambda^{\spl}_0)$ to $\tau(\calM_x \supset C_x) \in M^{\loc} = \calG_0^{\spl} / P_+$.
  By the same arguments as in \cite[\S 3.4]{ShenYuZhang2021}, we have
  \begin{itemize}
  	\item $I = \widetilde{\scrA}^{\spl}_0$ is a $\calG_{0}^{\spl}$-torsor;
  	\item $I_+ = q^{-1}(\Lambda_0^{\spl} \supset C_0)$ is a $P_+$-torsor;
  %	\item $I_- = q^{c,-1}(D_0 \subset \Lambda_0^{\spl})$ is a $P_-^{\varphi}$-torsor;
  %	\item There is an $L^{\varphi}$-equivariant isomorphism $\iota: (I_+ / U_+)^{\varphi} \simeq I_- / U_-^{\varphi}$.
  \end{itemize}
  Recall the forgetful map $\scrA^{\spl}_0\ra \scrA_0$. There is a conjugate local model diagram for $\scrA_0$:
  \[\xymatrix{
    & \widetilde{\scrA_0} \ar[ld]_{\pi} \ar[rd]^{q^c} \\
    \scrA_0 && M^{\loc,c}(\calG_0, \mu)
  }\]
  where $M^{\loc,c}(\calG_0, \mu)$ classifies the conjugate filtration of $\scrA_0$ defined in \cite{ShenYuZhang2021}
  and $q^c$ is a smooth morphism (same as \cite[Theorem 3.4.2]{ShenYuZhang2021}) sending points of $\widetilde{\scrA}_0$ to corresponding conjugate filtration.
  Let 
  $$M^{\spl, c}(\calG_0, \mu) \to M^{\loc, c}(\calG_0, \mu)$$
  be the scheme over $k$ classifying the splitting structures of $M^{\loc, c}(\calG_0, \mu)$.
  The pull-back of the morphism $q^c$ gives the conjugate splitting local model diagram
  \[\xymatrix{
  	& \widetilde{\scrA}^{\spl}_0 \ar[ld]_{\pi} \ar[rd]^{q^c} \\
  	\scrA^{\spl}_0 && M^{\spl, c}(\calG_0, \mu).
  }\]
  By the previous method of constructing local model diagram of splitting models,
  there is a local model diagram:
  \[\xymatrix{
  	& \widetilde{\scrA}^{\spl}_0 \ar[ld]_{\pi} \ar[rd]^{q^c} \\
  	\scrA^{\spl}_0 && M^{\loc, c}
  }\]
  where $M^{\loc, c}:= M^{\loc, c}(\calG_0^\spl, \mu)$ and $q^c$ is a $\calG_0^\spl$-equivariant smooth morphism sending $(x = (\underline{A},\underline{\scrF_\bullet}), \tau:\calM_x \simeq \Lambda^{\spl}_0)$ to $\tau(D_x \subset \calM_x) \in M^{\loc,c} := \calG_0^{\spl} / P_-^{\phi}$. Combining these diagrams we have
  \begin{itemize}
    \item $I_- := q^{c,-1}(D_0 \subset \Lambda_0^{\spl})$ is a $P_-^{\phi}$-torsor. This follows from the same argument of $I_+$ using the conjugate local model diagram;
    \item There is an $L^{\phi}$-equivariant isomorphism $\iota: (I_+ / U_+)^{\phi} \simeq I_- / U_-^{\phi}$. This follows from Lemma~\ref{lem l=1} and Lemma~\ref{lem l big}.
  \end{itemize}
\end{proof}

Let $\phi' = \prod_i \phi'_i$ be the group isomorphism given by 
\[\phi'_i: \calG_{i,0}^{\spl} \to \calG_{i,0}^{\spl}, \quad (x_1, \dots, x_{e_i}) \mapsto (x_{e_i}, x_1^p, \dots, x_{e_i-1}^p).\]
Then we have $\phi \cdot \phi' = \phi' \cdot \phi = \sigma$,
so $I_-^{\phi'}$ is a $P_-^{(p)}$-torsor and 
\[\iota^{\phi'}: (I_+ / U_+)^{(p)} \simeq I_-^{\phi'} / U_-^{(p)}\]
is an $L^{(p)}$-equivariant isomorphism.
This is the universal $\mathcal{G}_0^\spl$-zip of type $\mu$ over $\Spl{\mathscr{A}_0}$,
i.e. there is a morphism of algebraic stacks over $\kappa$:
\[\zeta: \Spl{\mathscr{A}_0} \to \Zip{\Spl{\mathcal{G}_0}}{\mu}_\kappa.\]

\begin{definition}
	The fibers of $\zeta$ are called the Ekedahl-Oort strata of $\Spl{\mathscr{A}}_0$.
\end{definition}
Note that for each $1\leq i\leq r$, by construction we have in fact a universal $\calG_{i,0}^\spl$-zip of type $\mu_i$ over $\Spl{\mathscr{A}}_0$, thus a morphism \[\zeta_i: \Spl{\mathscr{A}_0} \to \Zip{\Spl{\mathcal{G}_{i,0}}}{\mu_i}_\kappa.\] 
Via the inclusion $\mathcal{G}_0^\spl\subset \prod_i\calG_{i,0}^\spl$,
the universal $\mathcal{G}_0^\spl$-zip of type $\mu$ over $\Spl{\mathscr{A}_0}$ induces a $\prod_i\calG_{i,0}^\spl$-zip of type $\prod_i\mu_i$, which is  isomorphic to the product of all the $\calG_{i,0}^\spl$-zip of type $\mu_i$.

\subsection{Smoothness of $\zeta$}

The dimension formula, closure relation, and smoothness of EO (Ekedahl-Oort) strata follow from the following proposition.
\begin{proposition}\label{prop sm zeta}
	The morphism $\zeta: \Spl{\mathscr{A}_0} \to \Zip{\Spl{\mathcal{G}_0}}{\mu}_\kappa$ is smooth.
\end{proposition}
\begin{proof}
	Let $k$ be the algebraic closure of $\kappa$. 
	Consider the following cartesian diagram
	\[\xymatrix{
	{\mathscr{A}_{k}^{{\spl}\#}} \ar[r]^{\zeta^{\#}} \ar[d] & {\Spl{\mathcal{G}}_k} \ar[d] \\
	{\Spl{\mathscr{A}}_{k}} \ar[r] & {\Zip{\mathcal{G}_0^{\spl}}{\mu}_\kappa \otimes k}.
	}\]
	The smoothness of $\zeta$ is equivalent to the smoothness of  $\zeta^\#$, which is also equivalent to the surjectivity of the induced map on tangent spaces at points of $\mathscr{A}^{{\spl}\#}(k)$.

	Recall $\Zip{\mathcal{G}_0^{\spl}}{\mu}_\kappa \simeq [E \backslash \mathcal{G}_0^{\spl}]$.
	Let $x^\#$ be a closed point of $\mathscr{A}_k^{\spl}$ with image $x$ in $\mathscr{A}_k^{\spl}$.
	Let $A$ be the complete local ring of $\mathscr{A}_k^{\spl}$ at $x$. Consider the cartesian diagram
\[\xymatrix{
	X \ar[r]^{\alpha} \ar[d] & {\Spl{\mathcal{G}}_k} \ar[d] \\
	{\Spec\,A} \ar[r] & {\Zip{\mathcal{G}_0^{\spl}}{\mu}_\kappa \otimes k}
}\]
The morphism $X \to \Spec\,A$ is a trivial $E$-torsor isomorphic to $\underline{I}_{u g}$ for an $A$-point $u g$ of $\Spl{\mathcal{G}}$.
Let $U_- \subset \Spl{\mathcal{G}}_k$ be the opposite unipotent subgroup determined by $\mu$, 
then $A$ is isomorphic to the complete local ring of $U_-$ at identity.
The trivialization induces an isomorphism and translates $\alpha$ into the morphism 
$\beta: \Spec\,A \times_k E_k \to \Spl{\mathcal{G}}_k$
given by for any $k$-scheme $T$, on $T$-points $\beta: (u,l,u_+,u_-) \mapsto l u_+ u g (l^{(p)}u_-)^{-1}$. 
Now the same method of the last paragraph of \cite{Zhang2018EO}*{Theorem 4.1.2} shows the map on the tangent space at $x$ is surjective.
\end{proof}

\subsection{Non-emptiness of EO strata}
Recall that we have natural morphisms over $\kappa$:
\[\scrA_{0}^\spl\ra \scrA_{0}\ra\scrA_{0}^{\naive}.\]
In the following we work over $k=\ov{\kappa}$. Although we are primarily interested in the geometry of EO strata of $\scrA_k^\spl$, sometimes it would be helpful to study the geometry of $\scrA_k$ together. 
On $\scrA_{k}^{\naive}$, there is
 a KR (Kottwitz-Rapoport) stratification given by the isomorphism class of $\omega_{A/S}$, cf. \cite{Hartw}.
By \cite{PappasRapoport2005, Levin2016, HainesRicharz2020}, there is also a KR stratification \[\scrA_k=\coprod_{w\in\Adm(\mu)_{K}}\scrA_k^w\] indexed by the $\mu$-admissible set: let $K=K_p=\G(\Z_p)$ and $\Adm(\mu)_{K}$ the $\mu$-admissible set of level $K$ as introduced in \cite{HeRapoport2017}, which is the partially ordered set of all KR types of level $K$. By the discussions in \ref{subsubsec int group}, we have a decomposition $K=\prod_{i\in I}K_i$ and thus $\Adm(\mu)_{K}=\prod_{i\in I}\Adm(\mu_i)_{K_i}$.
We also have an EKOR (Ekedahl-Kottwitz-Oort-Rapoport) stratification of $\mathscr{A}_k$, which is a refinement of the KR stratification, such that for each $w\in \Adm(\mu)_{K}$, we have a morphism of algebraic stacks \[\zeta_w: \mathscr{A}_k^w\ra \calG_0^{\rdt}\tr{-Zip}^{J_w},\]
see subsections \ref{subsec EKOR loc} and \ref{subsec EKOR gl} in the Appendix for more details. If $w=(w_i)\in \prod_{i\in I}\Adm(\mu_i)_{K_i}$, for each $i$, we have in fact a morphism $\zeta_{w_i}: \mathscr{A}_k^w\ra \calG_{0,i}^{\rdt}\tr{-Zip}^{J_{w_i}}$, the $i$-th component of $\zeta_w$.

\begin{proposition}\label{prop:EO-EKOR}
    Let $\pi: \mathscr{A}^{\spl}_k \to \mathscr{A}_k$ be the natural forgetful morphism.
    Let $x\in \mathscr{A}^{\spl}(k)$ and $y = \pi(x) \in \mathscr{A}(k)$.
    \begin{itemize}
        \item If $x$ is a minimal EO point of $\mathscr{A}^{\spl}(k)$ (i.e. $x$ lies in the minimal EO stratum), then $y$ is a minimal EKOR point of $\mathscr{A}(k)$ (i.e. $y$ lies in the minimal EKOR stratum, cf. \cite{ShenYuZhang2021} 1.2.5).
        \item If $y$ is a minimal EKOR point, then there exists $x' \in\pi^{-1}(y)\subset \mathscr{A}^{\spl}(k)$ such that $x'$ is a minimal EO point.
    \end{itemize}
\end{proposition}
\begin{proof}
	Both the involved $\calG_0^{\spl}$-zip and $\calG_0^{\rdt}$-zips admit decompositions over the indices $i$.
	We proceed according to the type of $i$.
	
Case (AL/AU):
   For each $j$, write $\sum_{l} d_{i,j}^l = t_j d_i + s_j$, where $0 \leq s_j < d_i$, and 
   \[\omega_{\min} := (\ve_i^{e_i - t_j  - 1})^{\oplus s_j} \bigoplus (\ve_i^{e_i - t_j})^{\oplus d_i - s_j}.\]
   Then a point $y \in \scrA_0$ is a minimal KR point if and only if \[\omega_{y,i,j} \simeq \omega_{\min} \bigoplus \omega_{\min}'\] (cf. \cite{PR03} section 3; here $\omega'_{\min}$
   and the splitting structure is determined by such structure on $\omega_{\min}$).
   Assume that $x = (\underline{A}, \underline{\scrF_\bullet})\in\mathscr{A}^{\spl}(k)$ is a minimal EO point and $y = \pi(x)$.
   The minimal condition is equivalent to
   \begin{enumerate}
   	\item $\scrF_{i,j}^1 \subset \Ker(V_{i,j}^1)\quad \tr{or}\quad \Ker(V_{i,j}^1) \subset \scrF_{i,j}^1 \subset \ve_i^{-1}\scrF_{i,j}^{0}$.
   	\item For each $l \geq 2$, $\scrF_{i,j}^l \subset \ve_i^{-1}\scrF_{i,j}^{l-2}\quad \tr{or}\quad 
   	\ve_i^{-1}\scrF_{i,j}^{l-2} \subset \scrF_{i,j}^l \subset \ve_i^{-1}\scrF_{i,j}^{l-1}$.
   \end{enumerate}
   This forces \[\omega_{x,i,j} \simeq \omega_{\min} \bigoplus \omega_{\min}',\] so $y = \pi(x)$ is a minimal KR point, which is then a minimal EKOR point by the condition (1).
   Conversely, if $y$ is a minimal EKOR point, then by the above condition (1) and (2), there is a splitting structure $\ul{\scrF_\bullet}$ of $y$, such that $x' = (y, \ul{\scrF_\bullet})$ minimal.
   
  Case (C):
   A point $y \in \scrA_0$ is a minimal KR point if and only if $\omega_{y,i,j}$ isomorphic to
   \[\omega_{\min} := ((\ve_i^{\lceil{e_i/2}\rceil}) \bigoplus (\ve_i^{\lfloor{e_i/2}\rfloor}))^{\oplus d_i} \subset k[\ve_i]^{\oplus 2d_i} .\]
  A similar construction as above gives the proof of type (C).
   
\end{proof}

\begin{corollary}
	The morphism $\zeta$ is surjective.
\end{corollary}
\begin{proof}
	By Proposition \ref{prop non-empty EKOR}, the EKOR strata of $\mathscr{A}_0$ are non-empty.
	Hence, minimal EO points of $\scrA_k^{\spl}$ exist by Proposition \ref{prop:EO-EKOR}.
	The smoothness of $\zeta: \scrA_0^\spl \to \Zip{\calG_0^{\spl}}{\mu}_\kappa$ plus the fact that its image contains a minimal point imply the surjectivity of $\zeta$.
\end{proof}

We summarize the above results in a theorem, which describes the expected properties of Ekedahl-Oort stratification.

\begin{theorem}\label{thm EO}
	We have the following basic properties of EO strata on smooth splitting models.
	\begin{enumerate}
	\item There is a smooth surjective morphism \[\zeta: \scrA_0^{\spl} \to \Zip{\calG_0^{\spl}}{\mu}_\kappa\] of algebraic stacks.
	\item Let ${}^J W$ be the topological space of $\Zip{\calG_0^{\spl}}{\mu}_k$,
		and define the EO stratum as ${\scrA_k^{\spl,w}} := \zeta^{-1}(w)$ for each $w \in {}^J W$.
		Then each EO stratum ${\mathscr{A}^{\spl,w}_k}$ is a smooth and locally closed subscheme of dimension $l(w)$,
		with the closure relation by the partial order $\preceq$, i.e. 
		\[\overline{\scrA_k^{\spl,w}} = \coprod_{w' \preceq w} {\mathscr{A}^{\spl,w'}_k}, \quad \tr{ for all } w \in {}^J W.\]
	\end{enumerate}
\end{theorem}

\subsection{The $\mu$-ordinary locus}\label{subsec mu-ord}

Let $U \subset \Spl{\mathscr{A}}_k$ be the maximal EO stratum, then by Theorem \ref{thm EO}, $U$ is open dense in $\Spl{\mathscr{A}}_k$.

Recall the Kottwitz set $B(G,\mu)$ (see subsection \ref{subsec nonempty}). 
The universal abelian scheme with additional structure $(A,\lambda,\iota)$ over  $\scrA_0^\spl$ defines as usual a map
\[\Newt: \scrA_0^\spl(k)\lra B(G,\mu),\]which by construction factors through $\scrA_{0}(k)$ and $\scrA_0^{\naive}(k)$. 
\begin{definition}
	The fibers of $\Newt$ define a decomposition of $\mathscr{A}^{\spl}_k$, which
	we call the Newton stratification of $\mathscr{A}^{\spl}_k$.
	Moreover, as $G$ is quasi-split at $p$, there is a unique maximal Newton stratum in $\mathscr{A}_k^\spl$, called the $\mu$-ordinary locus.
\end{definition}

When the group $G$ is unramified over $\mathbb{Q}_p$, Moonen \cite{Moonen2004} proved that
the $\mu$-ordinary locus coincides with the maximal EO stratum of $\scrA_k = \scrA_k^{\spl}$.
This fact was generalized to general Hodge-type Shimura varieties with good reduction at $p$ in \cite{Wortmann2013}.
We will prove that for the ramified PEL-type case, the same result holds for smooth splitting models. As in the last subsection, we will apply the geometry of KR and EKOR stratifications of $\scrA_k$.
Recall that $\Adm(\mu)_{K}$ is the $\mu$-admissible set of level $K$ , which is the partially ordered set of all KR types of level $K$. As $K\subset G(\bbQ_p)$ is a \emph{special parahoric subgroup}, by \cite{Richarz2016}*{Theorem~4.2, Corollary~4.6} 
the set $\Adm(\mu)_{K}$ has a unique maximal element.
\begin{proposition}\label{maxKR}
	Let $w_0$ be the maximal element of $\Adm(\mu)_{K}$ and $\mathscr{A}_k^{w_0}$ the maximal KR stratum of $\mathscr{A}_k$.
	Then the morphism $\pi: \mathscr{A}^{\spl}_k \to \mathscr{A}_k$ induces an isomorphism over $\mathscr{A}_k^{w_0}$,
	i.e. $\pi: \pi^{-1}(\mathscr{A}_k^{w_0}) \st{\sim}{\ra} \mathscr{A}_k^{w_0}$.
\end{proposition}
\begin{proof}
For every $i,j$,
let $\{a_{i,j}^l \mid 1 \leq l \leq e_i \}$ be a permutation of $\{d_{i,j}^l \mid 1 \leq l \leq e_i\}$ such that
\[a_{i,j}^1 \geq a_{i,j}^2 \geq \cdots \geq a_{i,j}^{e_i}.\]
We define a module \(\omega_{\max,i} = \bigoplus_{j} \omega_{\max, i, j}\), where
\[\omega_{\max,i,j} :=  k[\varepsilon_i]^{a_{i,j}^{e_i}} \bigoplus (\varepsilon_i)^{a_{i,j}^{e_{i-1}} - a_{i,j}^{e_i}} \bigoplus \cdots \bigoplus (\varepsilon_i^{e_i-1})^{a_{i,j}^1 - a_{i,j}^2}.\]
The maximal KR stratum $\mathscr{A}_k^{w_0}$ can be described as follows: $x \in \mathscr{A}_k^{w_0}(k) $ if and only if the following condition ( dependent on the type of the index \(i\)) holds:
\begin{itemize}
	\item Case (AL/AU): Under the natural decomposition \(\omega_{x,i} = \omega_{x,i,1} \bigoplus \omega_{x,i,2}\) given by the \(R_i / F_i\)-action, we require that \(\omega_{x,i,1} \simeq \omega_{\max,i}\).
	\item Case (C): We require that \(\omega_{x,i} \simeq \omega_{\max, i} = k[\ve_i]^{d_i} \subset k[\ve_i]^{2 d_i}\).
\end{itemize}
This condition comes from the explicit description of the maximal element in $\Adm(\mu_i)_{K_i}$.

Let $x\in \mathscr{A}_k^{w_0}(k)$. In the case (C) we have $a_{i,j}^l = d_{i,j}^{l}$ for all $l$,
then $\omega_{\max,i,j}[\varepsilon_i^l]$ has rank $d_{i,j}^1 + \cdots + d_{i,j}^l$.
This forces \[\scrF_{i,j}^l = \omega_{x,i,j}[\varepsilon^l_i].\]
Hence the splitting structure over $\omega_{x,i,j}$ is unique. The case (AL/AU) is similar. Therefore the splitting structure over $\omega_{x}$ is unique,
and $\pi: \pi^{-1}(\mathscr{A}_k^{w_0}) \simeq \mathscr{A}_k^{w_0}$ is an isomorphism.
\end{proof}

Since there is a unique maximal KR stratum in $\scrA_k$,
we have a unique maximal EKOR stratum in $\scrA_k$. Recall that for each $1\leq i\leq r$, we have $\calG_{i,0}^{\spl}\simeq \calG^{\rdt,e_i}_{i,0}$. We get a map $\calG_0^{\spl}\ra \calG_0^{\rdt}$ which is induced by the projection to the first factor $\calG_{i,0}^{\spl}\ra \calG^{\rdt}_{i,0}$ for each $i$. Let $\mu'$ be the cocharacter of $\calG^{\rdt}_0$ induced by $\mu$ under this map.
As a continuation of Proposition \ref{maxKR}, we have
\begin{proposition}\label{prop EO in max KR} 
	Let $U_{\rm KR}=\mathscr{A}_k^{w_0}$ be the maximal KR stratum of $\mathscr{A}_k$.
	Then the morphism \(\zeta: \scrA_k^{\spl} \to \Zip{\calG_0^{\spl}}{\mu}_k\) induces a morphism
	\[\zeta_1: \pi^{-1}(U_{\rm KR}) \to \Zip{\calG_0^{\rdt}}{\mu'}_k.\]
	Moreover, the natural morphism $\pi: \Spl{\mathscr{A}}_k \to \mathscr{A}_k$ induces a commutative diagram
	\[\xymatrix{
		\pi^{-1}(U_{\rm KR}) \ar[d]_{\pi}^{\simeq} \ar[r]^{\zeta_1} & \Zip{\calG_0^{\rdt}}{\mu'}_k \\
		U_{\rm KR} \ar[ur]_{\zeta_2}
	}\]
	where $\zeta_2=\zeta_{w_0}$ is the zip morphism on the maximal KR stratum of $\scrA_k$ constructed in subsection \ref{subsec EKOR loc}.
\end{proposition}
\begin{proof}
The projection \(\calG_0^{\spl} \to \calG_0^{\rdt}\)  induces a morphism $\Zip{\calG_0^{\spl}}{\mu}_k\ra \Zip{\calG_0^{\rdt}}{\mu'}_k$. Composing with $\zeta: \scrA_k^{\spl} \to \Zip{\calG_0^{\spl}}{\mu}_k$ and restricting to $\pi^{-1}(U_{\rm KR})$, we get a morphism \[\zeta_1: \pi^{-1}(U_{\rm KR}) \to \calG_0^{\rdt}\text{-}{\rm Zip}^{\mu'}_{k}.\]
By Proposition \ref{maxKR}, the splitting structure over \(U_{\rm KR}\) is unique, 
hence the EO strata in \(\pi^{-1}(U_{\rm KR})\) are uniquely determined by the fiber of $\zeta_1$. In other words, the restriction of the universal $\calG_0^{\spl}$-zip of type $\mu$ to \(\pi^{-1}(U_{\rm KR})\) is uniquely determined by the associated $\calG_0^{\rdt}$-zip of type $\mu'$. To prove that the above diagram commutes, we need to show that this $\calG_0^{\rdt}$-zip is the pullback (under $\pi$) of the $\calG_0^{\rdt}$-zip on $U_{\rm KR}$ given by $\zeta_2=\zeta_{w_0}$.

Let $x = (A, \lambda, \iota, \alpha) \in U_{\rm KR}(k)$ and $y = \pi^{-1}(x)$.
We will explicitly compare the $\calG_0^{\rdt}$-zips at $x$ and $y$. The argument will be dependent on the type of the indexes \(i\).
We first investigate the $\calG_0^{\rdt}$-zip at $x$.

Case (AL/AU):
Fix an isomorphism \[\calH_{x,i} \simeq \Lambda_{k,i} = k[\ve_i]^{\oplus2 d_i f_i}\] and 
$\Lambda_{k,i}$ decomposes as $W \oplus W'$, with $W = k[\ve_i]^{\oplus d_i f_i}$, $W'$ the dual of $W$. There are induced Frobenius morphism $F$ and Verschiebung morphism $V$ on $\Lambda_{k,i}$.
The $F$-zip associated to $x$ (with $\calG_{i,0}^{\rdt}$-structure) induces an $F$-zip with $\kappa_i$-action 
$M = (W, C, D, \iota)$ given by \[W = k[\ve_i]^{\oplus d_i f_i},\quad C = \Ker(V|_W),\quad D = \Ker(F|_W),\]
and $M$ determines the $F$-zip associated to $x$.

By the proof of Proposition \ref{maxKR}, we have $C = \bigoplus_j C_j$, where
\[C_j \simeq k[\ve_i]^{\oplus a_{i,j}^{e_i}} \bigoplus (\varepsilon_i)^{\oplus(a_{i,j}^{e_{i-1}} - a_{i,j}^{e_i})} \bigoplus \cdots \bigoplus (\ve_i^{e_i-1})^{\oplus (a_{i,j}^1 - a_{i,j}^2)}.\]
The EO strata inside  $U_{\rm KR}$ are determined by an $F$-zip \((M', C', D', \iota')\) with $\calG_{i,0}^{\rm rdt}$-structure (recall that \(\calG_{i,0} = \Res_{\kappa_i[\ve_i] | \mathbb{F}_p} H_i, \calG_{i,0}^{\rm rdt} = \Res_{\kappa_i |\mathbb{F}_p} H_i\)). More precisely, we have
\begin{itemize}
	\item \(M' = M / \ve_i M = k^{\oplus d_i}.\)
	\item \(C'\) is the natural image of \(C\) in \(M'\), 
	we have $C' = \bigoplus_j C'_j$ and \(\dim_k C'_j = a_{i,j}^{e_i}\).
	\item Under the natural isomorphism \(\ve_i^{e_i-1}M \simeq M'\), let \(D' := D \cap \ve_i^{e_i-1}M\), 
	we have \(D' = \bigoplus_j D'_j\) and \(\dim_k D'_j = d_i - a_{i,j}^{e_i}\).
	\item The morphisms \(F,V\) of $M$ restrict to the morphisms
	\[F: M' \to \ve_i^{-1}(\ve_i C) / (\ve_i C) \quad \tr{and}\quad V: \ve_i^{-1}(\ve_i C) / (\ve_i C) \to M'.\]
	Composing with the natural isomorphism (same as Lemma \ref{lem l=1})
	\[\ve_i^{-1}(\ve_i C) / (\ve_i C) \simeq M'\]
	give semi-linear isomorphisms $F', V'$ of $M'$, such that
	\[\Ker(F') = \Image(V') = D', \quad \Ker(V') = \Image(F') = C'.\]
	The morphism \(\iota' = (\iota'_0, \iota'_1)\) is induced from $F', V'$.
\end{itemize}
The type of a $\calG_{i,0}^{\rm rdt}$-zip is determined by a cocharacter $\mu'$ associated to the tuple \((a_{i,j}^{e_i})_j\) (up to conjugate).

Case (C):
Fix an isomorphism \[\calH_{x,i} \simeq \Lambda_{k,i} = k[\ve_i]^{\oplus 2 d_i f_i}.\]
The $F$-zip associated to $x$ (with $\calG_{i,0}^{\rdt}$-structure) induces an $F$-zip
$M = (\Lambda_{k,i}, C, D, \iota)$, $C = \Ker(V), D = \Ker(F)$.
In this case, we have $d_{i,j}^l = d_i$, and $C = \bigoplus_j C_j$ where
\[C_j \simeq k[\ve_i]^{\oplus d_i} \subset k[\ve_i]^{\oplus 2 d_i}.\]
Similar to the case (AL/AU), such data induce an $F$-zip \((M', C', D', \iota')\) with \(\calG_{i,0}^{\rm rdt}\)-structure, with type \((1^g, 0^g)^{e_i f_i}\).

On the other hand, fixing an isomorphism \(\calM_y \simeq \Lambda^{\rm spl}_k\) with associated 
\(F\)-zip \((\Lambda_k^{\rm spl}, C, D, \iota)\), we have
\[\calM_{i}^{e_i}  = \ve_i^{-1}(\ve_i C) / \ve_i C, \quad \calM_{i}^1 = \ve_i^{e_i-1} \Lambda_{k,i}.\]
By identifying \(\calM_{i}^{e_i}\) and \(\calM_{i}^1\), our construction in subsection \ref{subsection zip} induces the same $F$-zip with \(\calG_{i,0}^{\rm rdt}\)-structure $(M', C', D', \iota')$ as that at $x$. This finishes the proof.
\end{proof}

\begin{corollary}\label{cor mu-ordinary}
	The $\mu$-ordinary locus of $\mathscr{A}_k^\spl$ coincides with the maximal EO stratum, thus it is open dense in $\scrA_k^\spl$.
\end{corollary}
\begin{proof}
	Note that the Weil restriction of a quasi-split group remains quasi-split, 
	so the group $G_{\mathbb{Q}_p}$ is quasi-split.
	Let $U_{\rm EKOR}$ be the maximal EKOR stratum of $\mathscr{A}_k$. Then $U_{\rm EKOR}\subset U_{\rm KR}$.
	The main result of \cite{HeNie2017} shows that the $\mu$-ordinary locus of $\mathscr{A}_k$ is the same as $U_{\rm EKOR}$. 
	By the above proposition, we have \[\pi^{-1}(U_{\rm EKOR}) = U,\]
	hence the maximal EO stratum $U$ coincides with the $\mu$-ordinary locus of $\Spl{\mathscr{A}}_k$.
\end{proof} 

The open density of the $\mu$-ordinary locus was proved in
\cite{BijakowskiHernandez2022}*{Theorem 1.2} by a different method.

\subsection{Hodge strata and pullbacks of Kottwitz-Rapoport strata}\label{subsection Hodge and KR}
In \cite{BijakowskiHernandez2022},
Bijakowski-Hernandez introduced a Hodge stratification. Here,
we provide an interpretation of their Hodge stratification in our language.

Recall the local setup of \cites{BijakowskiHernandez2017, BijakowskiHernandez2022}.
Let $L$ be a finite extension of $\bbQ_p$ of degree $n = e f$,
where $e$ is the ramification index, and $f$ the residue degree of $L|\bbQ_p$.
Let $\pi$ be a uniformizer of $\calO_L$.
Let $\mathbf{G}$ be a $p$-divisible group over $k$ with $\calO_L$ action $\iota$, such that its height is $nh$. This is the local geometric datum in the case of type (AL). In the case of type (C) or (AU), we require moreover there to be a polarization $\lambda: \mathbf{G}\ra \mathbf{G}^\vee$ of $\mathbf{G}$. Note that even in the type (AL) case, we have a natural polarization $\lambda: \mathbf{G}\times\mathbf{G}^\vee\ra\mathbf{G}^\vee\times \mathbf{G}$.
Let $(M,F)$ be the contravariant $F$-crystal associated to the $p$-divisible group $\mathbf{G}$. 
Then $M$ is a finite free $W(k)$-module of rank $nh$ and $F: M \to M$ a $\sigma$-linear injective morphism.
The $\calO_L$-action induces a decomposition of $W(k)$-module
\[M = \bigoplus_{\tau \in J} M_{\tau},\]
where $J := \Hom(\calO_{L^{\ur}}, W(k))$,
and $M_\tau$ is the $W(k)$-submodule of $M$ such that $\calO_L$-acts through $\tau: \calO_{L^{\ur}} \to W(k)$. In the case of type (C) or (AU), there is an induced perfect pairing $M\times M\ra W(k)$ from $\lambda$, which is compatible with the decomposition.
As $F$ is $\sigma$-linear, we have
\[F(M_{\sigma^{-1}\tau}) \subset M_\tau.\]
For each $\tau \in J$, there is an isomorphism
\[M_\tau / F(M_{\sigma^{-1}\tau}) \simeq \bigoplus_{1 \leq i \leq h} W(k)/a_{\tau,i} W(k),\]
where each $a_{\tau,i} \in W(k)-\{0\}$ and we assume that
\[v(a_{\tau,1}) \geq v(a_{\tau,2}) \geq \dots \geq v({a_{\tau,h}}).\]
For each $0 \leq i \leq h$,
the $i$-th coordinate of the Hodge polygon is:
\[\Hdg_{\calO_L, \tau}(M,F)(i) := v(a_{\tau,1}) + \cdots + v(a_{\tau,h-i+1}).\]
Note that our polygon is upper convex, whereas \cite{BijakowskiHernandez2017} uses lower convex polygon.
All statements of loc. cit. about polygons still hold, after replacing minimal by maximal.
The Hodge polygon of $\ul{\mathbf{G}}=(\mathbf{G},\iota,\lambda)$ is defined as a polygon starting at $(0,0)$ and ending at $(h,d)$ such that for $0 \leq i \leq h$
\[\Hdg(\ul{\mathbf{G}})(i) = \frac{1}{f} \sum_{\tau \in J} \Hdg_{\calO_L, \tau}(M,F)(i),\]
where $d=\frac{1}{f}\sum_{\tau\in J}d_\tau$ and $d_\tau=\sum_{i=1}^h v(a_{\tau,i})$. The precise form of the pair of numbers $(h,d)$ depends on the type (C), (AL) or (AU) which is being studied.  Note that since we have an isomorphism \[\omega_\mathbf{G}\simeq M/FM\] as $\calO_L$-modules, the Hodge polygon of $\ul{\mathbf{G}}$ depends only on the $\calO_L$-module $\omega_\mathbf{G}$.

Now fix integers $0 \leq d_{\tau}^l \leq h$ for $\tau \in J, 1 \leq l \leq e$,
we say that $(M,F)$ satisfies the \emph{Pappas-Rapoport condition} with respect to $\mu := (d_{\tau}^l)$,
if for every $\tau$, there exists a filtration
\[F_\tau M_{\sigma^{-1}(\tau)} = \Fil^0 M_\tau \subset \Fil^1 M_\tau \subset \cdots \subset \Fil^{e} M_\tau = M_\tau, \]
such that
\begin{itemize}
	\item For each $0 \leq l \leq e$, $\Fil^l M_\tau$ is a sub $W(k)$-module of $M_\tau$,
	\item One has $\pi \cdot \Fil^{l} M_\tau \subset \Fil^{l-1} M_\tau$ for all $1 \leq l \leq e$,
	\item $\Fil^l M_\tau / \Fil^{l-1} M_\tau$ is a $k$-vector space of dimension $d_{\tau}^l$ for all $1 \leq l \leq e$.
\end{itemize}
Given the Pappas-Rapoport condition with respect to $\mu$,
the maximal possible Hodge polygon is given by the polygon $\PR(\mu)$,
where for each $\tau$, let
\[\PR_\tau(\mu)(s) := \frac{1}{e} \sum_{l=1}^e \max(s-h+d_{\tau}^l),\]
and $\PR(\mu)$ is the average of $\PR_\tau(\mu)$ for $\tau \in J$.
Recall the following result
\begin{theorem}[\cite{BijakowskiHernandez2017}]
	Let $\underline{\mathbf{G}} = (\mathbf{G}, \iota, \lambda,\Fil^{\bullet})$ be a $p$-divisible group over a perfect field $k$ of characteristic $p$,
	with an $\calO_L$-action,
	plus a Pappas-Rapoport condition given by $\mu = (d_\tau^l)_{\tau,l}$.
	We have three polygons associated to $\underline{\mathbf{G}}$,
	the Newton polygon (upper convex) $\Newt(\underline{\mathbf{G}})$ and the Hodge polygon $\Hdg(\underline{\mathbf{G}})$ with respect to $\underline{\mathbf{G}}$,
	and a PR polygon $\PR(\mu)$ determined by the Pappas-Rapoport condition with respect to $\mu$.
	One has the following inequalities:
	\[\Newt(\underline{\mathbf{G}}) \leq \Hdg(\underline{\mathbf{G}}) \leq \PR(\mu).\]
\end{theorem}

Using our language,
let $x \in \scrA^{\spl}(k)$,
and $\mathbf{G}_x$ the $p$-divisible group associated to $x$.
Then $\mathbf{G}_x$ decomposes as direct sum $\mathbf{G}_x=\bigoplus_{i=1}^r\mathbf{G}_{x,i}$,
and each $\mathbf{G}_{x,i}$ is a $p$-divisible group over $k$ with $\calO_{F_i}$-action and polarization $\lambda$,
plus a PR condition with respect to $\mu_{i}$.
So for each $i$,
there are three polygons $\Newt_i(x):=\Newt(\mathbf{G}_{x,i}), \Hdg_i(x):=\Hdg(\mathbf{G}_{x,i}), \PR(\mu_i)$ associated to $x$. Note that for each $i$, we have an isomorphism of $\calO_{F_i}$-module $\omega_{x,i}\simeq \omega_{\mathbf{G}_{x,i}}$.
\begin{corollary}\label{cor max Hodge}
	Let $U_{\rm KR}$ be the maximal open dense KR stratum of $\scrA_k$ and $\pi: \scrA_k^{\spl} \to \scrA_k$ the forgetful morphism.
	For each $x\in \scrA_k^{\spl}$, 
	we have $\Hdg_i(x) = \PR(\mu_i)$ for all $i$ if and only if $x \in \pi^{-1}(U_{\rm KR})$.
\end{corollary}
\begin{proof}
	This follows from the description of the maximal KR stratum in the proof of Proposition~\ref{maxKR}.
\end{proof}

Note that $\pi^{-1}(U_{\rm KR})$ is the generalized Rapoport locus in \cite{BijakowskiHernandez2022}, where it is defined as the maximal Hodge stratum of $\scrA_k^{\spl}$.
Also note in general the closure relation does not hold for the Hodge stratification, as claimed by \cite{BijakowskiHernandez2022}.

We give a group theoretic reformulation of Hodge polygon. For simplicity, we assume $r=1$ (the general case follows by combining the independent data for all $i\in I$) in our previous notations.
Let $F_1$ be a finite extension over $\bbQ_p$ of degree $n_1 = e_1 f_1$,
with $e_1$ its ramification index and $f_1$ its residue degree.
We also fix an unramified group $H$ over $F_1$ and let $G_1= \Res_{F_1 | \bbQ_p} H$.
The maximal parahoric subgroup of $G_1(\Q_p)$ is $K = H(\calO_{F_1})$, where we use the same $H$ as the reductive model of $H$.
Let $\breve{\bbQ}_p$ be the completion of maximal unramified extension of ${\bbQ}_p$.
For simplicity, we will write $\breve{G}_1 = G_1(\breve{\bbQ}_p)$ and similar symbols for other groups.

Choose a maximal torus $T \subset G_1$ and $N$ its normalizer.
Let $\breve{I}$ be the Iwahori subgroup of $\breve{G}_1$.
The Iwahori-Weyl group is 
\[\widetilde{W} = N(\breve{\bbQ}_p) / (T(\breve{\bbQ}_p) \cap \breve{I}). \]
We have the following isomorphism
\[\widetilde{W} \simeq X_*(T)_{\Gamma_0} \rtimes W_0\]
where $\Gamma_0 = \Gal(\overline{\bbQ}_p / \breve{\bbQ}_p)$ and $W_0$ is the relative Weyl group of $G_{\breve{\bbQ}_p}$.
Let $\breve{K}\subset \breve{G}_1$ be the induced parahoric subgroup and
\[W_K = (N(\breve{\bbQ}_p) \cap \breve{K}) / (T(\breve{\bbQ}_p) \cap \breve{I}) \subset \widetilde{W},\] 
then $W_K \simeq W_0$ is the section of $W_0$ in $\widetilde{W}$, since $K$ is a special parahoric subgroup.
The following natural bijection
\[W_K \backslash \widetilde{W} / W_K \simeq X_*(T)_{\Gamma_0} / W_0 \]
gives the natural injective map
\[h: W_K \backslash \widetilde{W} / W_K \to X_*(T)_{\Gamma_0, \bbQ}^{+}.\]
Since $G_1$ is a Weil restriction of $H / F_1$, 
this implies
\[\breve{G}_1= H(F_1 \otimes_{\bbQ_p} \breve{\bbQ}_p) = \prod_{\tau: F_1^{\ur} \to \breve{\bbQ}_p} H(\breve{F}_1).\]
For simplicity, we just write $\breve{G}_1= \breve{H}^{f_1}$, where $\breve{H} = H(\breve{F}_1)$.
Similarly we have $T_{\breve{\bbQ}_p} \simeq T_1^f$ with $T_1\subset H$ a maximal torus,
which induces natural maps
\[X_*(T)_{\Gamma_0,\Q} \simeq (X_*(T_1)_{\Gamma_0,\Q})^f \to X_*(T_1)_{\Gamma_0,\bbQ}.\]
The last map sends $(\chi_1, \dots, \chi_f)$ to $\frac{1}{f} (\chi_1 + \cdots + \chi_f)$.
Composing this map with $h$ gives us a map
\[\Hdg: W_K \backslash \widetilde{W} / W_K \to X_*(T_1)_{\Gamma_0, \bbQ}^+.\]
Recall the $\mu$-admissible set $\Adm(\mu)_K \subset W_K \backslash \widetilde{W} / W_K$.
Restricting $\Hdg$ to it gives the following proposition.

\begin{proposition}
	For $x\in \scrA^\spl(k)$, let 
    $w \in \Adm(\mu)_K$ be the KR type of $\pi(x)\in \scrA(k)$. Then the Hodge polygon of $x$ is given by $\Hdg(w) \in X_*(T_1)_{\Gamma_0, \bbQ}^+$.
\end{proposition}
\begin{proof}
	Let $(M,F)$ be the
   $F$-crystal  with additional structure attached to $x$, which depends only on $\pi(x)$.
    The isomorphism class of $M/FM$ as $\calO_L \otimes W(k)$-module is given by an element in $\Adm(\mu)_K$,
    consisting of a tuple in $X_*(T)_{\Gamma_0}$.
    The average of the  tuple in $X_*(T)_{\Gamma_0}$ gives the Hodge polygon of the $F$-crystal $(M,F)$, which is (as usual) viewed as an element of $X_*(T_1)_{\Gamma_0, \bbQ}^+$.
\end{proof}
\begin{corollary}
	The Hodge stratification of $\scrA_k^\spl$ descends to $\scrA_k$ under the map $\pi: \scrA_k^\spl\ra \scrA_k$. In particular,
each Hodge stratum of $\scrA_k^\spl$ is a finite disjoint union of preimages of Kottwitz-Rapoport strata of $\scrA_k$ under the natural map $\pi: \scrA_k^\spl\ra \scrA_k$.
\end{corollary}
Recall that Corollary \ref{cor max Hodge} states that the maximal Hodge stratum is exactly the preimage of the maximal Kottwitz-Rapoport stratum. One checks easily that $\Hdg(w_0)=\PR(\mu)$.

\section{Hasse invariants for EO strata}

In this section, we apply the general theory of group theoretical Hasse invariants of \cite{GoldringKoskivirta2018, GoldringKoskivirta2019} to the smooth scheme $\scrA_{0}^\spl$, by the map $\zeta: \scrA_{0}^\spl\ra \Zip{\G_0^\spl}{\mu}_\kappa$ constructed in the last section. We will discuss some examples.

\subsection{Group theoretical Hasse invariants}

We come back to the setting of subsection \ref{subsec G-zip}.
Let $(G,\mu)$ be a cocharacter datum over $\mathbb{F}_p$,
this means that $G$ is a connected reductive group over $\mathbb{F}_p$ and $\mu$ is a cocharacter of $G$ defined over a finite field $\kappa$ over $\mathbb{F}_p$. 
Recall that $(P_+,P_-)$ is the pair of opposite parabolic subgroups of $G_\kappa$ determined by $\mu$, and
$L := P_- \cap P_+$ the Levi subgroup of $G_\kappa$.
Let $U_+$ (resp. $U_-$) the unipotent radical of $P_+$ (resp. $P_-$).
The zip group $E_{G,\mu}$ is the subgroup of $G_\kappa \times G_\kappa$
\[ E_{G,\mu}= \{ (p_+ = l u_+, p_- = l^{(p)} u_-) \in P_+ \times P_-^{(p)} \mid l
   \in L, u_+ \in U_+, u_- \in U_-^{(p)} \} . \] 
It acts on $G_\kappa$ as $(p_+, p_-) \cdot g := p_+ g p_-^{-1}$.
By Theorem \ref{thm quotient zip stack}, there is a canonical isomorphism of stacks
\[\Zip{G}{\mu}_\kappa \cong [E_{G,\mu}\backslash G_\kappa].\]

Let $k=\ov{\kappa}$ and $\Zip{G}{\mu} =\Zip{G}{\mu}_k=[E_{G,\mu,k} \backslash G_k]$.
Given $\lambda \in X^*(L)$, which can be viewed as an element of $X^*(E_{G,\mu})$ through the projection $E_{G,\mu}\to L$,
one can associate a line bundle $\calV(\lambda)$ over $[E_{G,\mu,k}  \backslash G_k]$ such that
\[H^0([E_{G,\mu,k}  \backslash G_k], \calV(\lambda)) = \{f: G_k \to \bbA^1_k \mid f(gx) = \lambda(g)f(x) \text{ for all } g \in E_{G,\mu,k} , x \in G_k\}.\]
Let $S$ be a scheme or an algebraic stack over $k$ with a morphism $\zeta : S \rightarrow \Zip{G}{\mu}$.
\begin{definition}
    For every $w \in {}^J W$, denote by $S_w$ the EO stratum of $S$ associated to $w$, i.e. $S_w = \zeta^{-1}(w)$ and $\ov{S}_w$ its Zariski closure in $S$. A Hasse invariant for $(\lambda, S_w)$ is a section $h_w \in H^0(\overline{S}_w, \mathcal{V}(n \lambda))$ for some positive integer $n$, such that its non-vanishing locus $D(h_w)$ is $S_w$. If such $h_w$ exists for all w, then $\lambda$ is called a Hasse generator for $S$.
\end{definition}

The following proposition shows that the Hasse generator is unique up to $k^\times$ if it exists.
\begin{proposition}[\cite{GoldringKoskivirta2018}]
    Let $\lambda \in X^*(L)$ and $U_\mu$ the unique open dense $E$-orbit in $G_k$, then
    \begin{enumerate}
        \item One has $\dim_k(H^0(\Zip{G}{\mu}, \mathcal{V}(\lambda))) \leq \dim_k(H^0(U_\mu, \mathcal{V}(\lambda))) \leq 1$.
        \item There is a positive integer $N_\mu$ such that the space $H^0(U_\mu, \mathcal{V}(N_\mu \lambda))$ has dimension one.
    \end{enumerate}
\end{proposition}

Recall that a cocharacter datum $(G,\mu)$ is Hodge type if there is a symplectic embedding $G\hookrightarrow\GSp_{2g}$ over $\F_p$ such that $\mu$ induces  the standard minuscule cocharacter of $\GSp_{2g}$.
For such datum, 
there is a Hodge line bundle $\calV(\eta_\omega)$ over $[E_{G,\mu} \backslash G_\kappa]$ by pull-back from the given symplectic embedding of $G$ (in general $\calV(\eta_\omega)$ depends on the symplectic embedding).
We have the following theorem.
\begin{theorem}[\cite{GoldringKoskivirta2018}]
    Let $(G,\mu)$ be a cocharacter datum over $\mathbb{F}_p$ and assume $\mu$ is minuscule.
    For $p>2$, the stack $\Zip{G}{\mu}$ admits a Hasse generator.

    Moreover, if $(G,\mu)$ is Hodge type ($p=2$ allowed), then the Hodge line bundle is associated to a Hasse generator.
\end{theorem}

%\begin{corollary}
 % If $\lambda \in X^*(L)$ is a Hasse generator,
 %then there exists $d \geq 1$ such that for all $w \in {}^JW$,
   % the section $h_{w}^d$ extends to $[E\backslash \overline{G}_w]$ with non-vanishing locus $[E\backslash G_w]$.
  %  Moreover, if $(G,\mu)$ is Hodge type, then the Hodge character $\eta_\omega \in X^*(L)$ is a Hasse generator.
%\end{corollary}

\subsection{Length Hasse invariants}
Keep notations as above.
Let $d=\langle 2\rho,\mu\rangle$, where $\rho$ is the half sum of positive (absolute) roots of $G$. Recall the subsets $G_w\subset G_k$ as in Theorem \ref{thm PWZ11} attached to $w\in {}^JW$.
For any integer $0 \leq j \leq d$, define 
\[G_j = \bigcup_{\ell(w)=j} G_w\]
with the reduced subscheme structure. It is called the $j$-th length stratum of $G_k$.
By \cite[]{GoldringKoskivirta2019}, we have the following identities:
\[\overline{G}_j = \bigcup_{\ell(w)=j} \overline{G}_w = \bigcup_{\ell(w)\leq j} G_w.\]
In the rest of this section, assume that $(G,\mu)$  is of Hodge type.
\begin{proposition}[\cite{GoldringKoskivirta2019}*{Proposition 5.2.2}]\label{prop length Ha}
There exists an integer $N' \geq 1$ such that for each $0 \leq j \leq d$, a section $h_j \in H^0([E_{G,\mu,k} \backslash \overline{G}_j], \calV(\eta_\omega)^{N'})$ such that $D(h_j) = [E_{G,\mu,k} \backslash G_j]$.
\end{proposition}
Such $h_j$ will be called a length Hasse invariant.

Now we assume that $S$ is a $k$-scheme and $\zeta : S \rightarrow \Zip{G}{\mu}$ is a morphism of stacks (maybe not smooth). For a character $\lambda \in X^{\ast} (L)$, write $\mathcal{V}_S (\lambda) =\zeta^{\ast} (\mathcal{V} (\lambda))$.
For $w \in {}^J W$ and $j \in \{ 0,\ldots, d \}$, 
we define the locally-closed subsets of $S$: $S_w, S_w^{\ast}, S_j, S_j^{\ast}$ as preimage of $[E_{G,\mu,k} \backslash G_w], [E_{G,\mu,k} \backslash \overline{G}_w],
[E_{G,\mu,k} \backslash G_j], [E_{G,\mu,k} \backslash \overline{G}_j]$ respectively, 
equipped with reduced subscheme structure from $S$.

\begin{proposition}[\cite{GoldringKoskivirta2019}*{Proposition 5.2.3}]\label{prop length Hasse}
  Assume that
  \begin{enumerate}
    \item The scheme $S$ is equi-dimensional of dimension $d$.
    \item The stratum $S_w$ is non-empty for all $w \in {}^J W$.
    \item The stratum $S_e = S_0$ is zero-dimensional.
  \end{enumerate}
  Then we have:
  \begin{enumerate}
    \item The schemes $S_j$ and $S_j^{\ast}$ are equi-dimensional of dimension $j$.
    \item The sections $h_j$ are injective; equivalently $S_j$ is open dense in $S_j^{\ast}$.
    \item For $w \in {}^J W$, $S_w$ is equi-dimensional of dimension $\ell(w)$.
  \end{enumerate}
\end{proposition}

\subsection{Hasse invariants on splitting models}
Back to splitting models,
let $S = \scrA^{\spl}_k$. Recall that by Theorem \ref{thm EO}, there is a smooth surjective morphism
\[\zeta: S \to \Zip{\calG_0^{\spl}}{\mu}_k.\]
For each $w \in {}^J W$,
let $S_w = \zeta^{-1}(w)$ be the EO stratum of $S$.
By our construction, $\calG_0^{\spl}$ is the similitude group of the lattice $\Lambda^\spl_0$ with a pairing induced from that on $\Lambda$, hence the pair $(\calG_0^\spl,\mu)$ is of Hodge type  and the general theory of Hasse invariants applies. For simplicity, here in the following we denote $\mathsf{E}=E_{\calG_0^\spl,\mu,k}$.

\begin{corollary}
	There exists an integer $N \geq 1$ such that for every $w \in {}^J W$,
	there exists a section $h_w \in H^0([\mathsf{E} \backslash \overline{\calG_{0,w}^{\spl}}], \calV(N \eta_\omega))$ whose non-vanishing locus is precisely $[\mathsf{E} \backslash \calG_{0,w}^{\spl}]$.
\end{corollary}

\begin{corollary}\label{cor hasse inv}
	Fix a large enough integer $N$ as in the above corollary.
	For every EO stratum $S_w \subset S$,
	the section $\zeta^*(h_w) \in H^0(\overline{S}_w, \omega_{\Hdg}^N)$ is $G(\bbA_f^p)$-equivariant,
	and its non-vanishing locus is $D(\zeta^* (h_w)) = S_w$.
\end{corollary}

Here we denote $\omega_{\Hdg}$ as the Hodge line bundle with weight $\eta_\omega$ on $S=\scrA^{\spl}_k$.
We will use the same symbols $h_w$ as their pullbacks to $\overline{S}_w$ and 
call them the \emph{Hasse invariants} of $S$. We also have the length Hasse invariants on $S$, and by our previous results in section \ref{section Fzip}, Proposition \ref{prop length Hasse} holds for $S$.

\subsection{Example: Hilbert modular varieties}\label{subsection exa Hilbert}
Now consider the splitting models of Hilbert modular varieties. We would like to compare our construction with that of Reduzzi-Xiao in \cite{ReduzziXiao2017}. 

Let $L|\Q$ be a totally real field extension and
$G = (\Res_{L|\Q}\GL_2)^{\det \in \Q^\times}$ the similitude group associated to the Hilbert moduli space with respect to $L$.
Fix a prime number $p$ and a prime to $p$ level structure $K^p \subset G(\bbA_f^p)$,
which is defined in \cite[]{ReduzziXiao2017} as $\Gamma_{00}(\mathcal{N})$-level structures.
Let $F$ be a Galois extension of $\bbQ_p$ containing all the $p$-adic factors of $L$.
Then we have a splitting model $\scrA^{\spl}$ defined over $\calO_F$ with level $K^p$. Let $k$ be an algebraic closure of the residue field of $\calO_F$,
and $\scrA_0^{\spl} = \scrA^{\spl} \otimes k$.

For any $k$-scheme $S$,
$\scrA_0^{\spl}(S)$ classifies the isomorphism class of tuples $(A, \lambda, \alpha, \underline{\scrF}_\bullet)$, where
\begin{enumerate}
	\item $(A, \lambda)$ is a polarized abelian scheme over $S$ with $\calO_L$-action, with level structure $\alpha$.
	\item $\underline{\scrF_\bullet} = (\scrF_{i,j}^l)_{1 \leq i \leq r, 1 \leq j \leq f_i, 0 \leq l \leq e_i}$, 
		each $\scrF_{i,j}^l$ is a locally free sheaf over $S$ such that
		\begin{itemize}
			\item $0 = \scrF_{i,j}^0 \subset \scrF_{i,j}^1 \subset \cdots \subset \scrF_{i,j}^{e_i} = \omega_{A/S,i,j}$ and each $\scrF_{i,j}^l$ is stable under the $\calO_L$-action.
			\item each subquotient $\scrF_{i,j}^l / \scrF_{i,j}^{l-1}$ is a locally free $\calO_S$-module of rank one.
			\item the $\calO_F$-action on each subquotient $\scrF_{i,j}^l / \scrF_{i,j}^{l-1}$ factors through $\sigma_{i,j}^l: \calO_L \to \calO_F$.
		\end{itemize}
\end{enumerate} 
Consider $S = \scrA_0^{\spl}$ and $(A,\lambda,\alpha, \underline{\scrF_\bullet})$ the universal object over $S$.
The sheaf $\calH_{i,j}=H_{\dR}^1(A/S)_{i,j}$ is a locally free $\calO_S[\varepsilon_i]$-module ($\varepsilon_i^{e_i}=0$) of rank two, 
where the $\pi_i$-action is given by $\varepsilon_i$. For each $i, j$ denote $\omega_{i,j}=\omega_{A/S,i,j}$.

For each $i,j,l$, let $h_{i,j}^l$ be the partial Hasse invariant defined by \cite[]{ReduzziXiao2017},
which is a section of a certain line bundle over $S$.
For a point $x = (\underline{A}, \underline{\scrF_\bullet}) \in S(k)$, 
the construction of partial Hasse invariants shows that
\begin{itemize}
	\item For $2 \leq l \leq e_i$, $h_{i,j}^l(x) = 0$ if and only if $\varepsilon_i \scrF_{i,j}^l = \scrF_{i,j}^{l-2}$.
	\item For $l = 1$, $h_{i,j}^1(x) = 0$ if and only if $\scrF_{i,j}^1 = \Ker(\Ver_{i,j}^1)$.
\end{itemize}
The vanishing loci of these $h_{i,j}^l$ cut out the stratification in \cite{ReduzziXiao2017}.
We claim that their stratification coincides with the Ekedahl-Oort stratification defined in this paper.

For each $i,j,l$ we have a locally free sheaf $\calM_{i,j}^l = \varepsilon_i^{-1} \scrF_{i,j}^{l-1} /\scrF_{i,j}^{l-1} $. Let $\calM_i=\bigoplus_{j,l}\calM_{i,j}^l $.
By our construction in subsection \ref{subsection zip},  there is a natural $G_i:= \Res_{\kappa_i|\bbF_{p}} \GL_2^{e_i}$-zip
\[(\calM_{i}, C_{i}, D_{i}, \varphi_0, \varphi_1)\]
of type $\mu_i = (1,0)^{e_if_i}$ over $S$.
The maps $V_{i,j}^l: \calM_{i,j}^l \to \calM_{i,j}^{l-1}$ (for $l\geq 2$) and $V_{i,j}^1:\calM_{i,j}^1 \to \calM_{i,j-1}^{e_i,(p)}$ (for $l=1$)  induce maps \[V_{i,j}^l: \omega_{i,j}^l \to \omega_{i,j}^{l-1} \quad l\geq 2, \quad \tr{and}\quad V_{i,j}^1: \omega_{i,j}^1 \to \omega_{i,j-1}^{e_i,(p)} = \omega_{i,j-1}^{e_i, \otimes p} \]
by restriction to $\omega$. 
Such morphisms give sections of $\omega_{i,j}^{l-1}\bigotimes(\omega_{i,j}^1)^{-1}$ and $\omega_{i,j-1}^{e_i, \otimes p} \bigotimes (\omega_{i,j}^1)^{-1}$ respectively, 
which are the partial Hasse invariants of Reduzzi-Xiao. More precisely, 
$\omega_{i,j}^l$ is just a Frobenius twisted version of $\omega_{\tau_{\mathfrak{p}_i,j}^l}$ in \cite[]{ReduzziXiao2017},
i.e. $(\omega_{i,j}^{l})^{(p)} = \omega_{\tau_{\mathfrak{p}_i,j}^l}$. 
Now by \cite[]{ImaiKoskivirta2021}*{Proposition 5.2.3}, 
every codimension one closed Ekedahl-Oort stratum of zip stacks can be cut out from zip partial Hasse invariants.
So we at least have codimension one closed EO strata coincide with those in Reduzzi-Xiao's definition.

In the Hilbert case, any EO stratum can be cut out from codimension one strata as the associated Weyl group is isomorphic to $(\bbZ/2\bbZ)^n$.
So all the EO strata of splitting models of Hilbert modular varieties coincide with the strata constructed by the partial Hasse invariants in \cite[]{ReduzziXiao2017}.

\subsection{Example: Hilbert-Siegel case}\label{subsec ex Gsp}

Let $L$ be a totally real field over $\Q$ with $n=[L:\Q]>1$ and \[G = (\Res_{L|\Q} \GSp_{2g})^{\det \in \Q^\times}\] the similitude subgroup of $\Res_{L|\Q} \GSp_{2g}$. 
In fact one can also work with Shimura varieties associated to the sightly larger group $\Res_{L|\Q} \GSp_{2g}$, which is of abelian type, by taking some suitable group quotient which does not change the geometry at $p$, see for example \cite{BCGP, SZ, ERX, KS}.
Fix a prime to $p$ level $K^p\subset G(\mathbb{A}_f^p)$ which is sufficiently small.
Let $\scrA^{\spl}$ be the splitting model over $\calO_F$ corresponding to $G$. Note that when $g=1$, we return to the Hilbert case in the last subsection.

As $L = B$, we have 
\[L \otimes_{\bbQ} \bbQ_p \simeq \prod_{i=1}^r F_i\]
and
\[\calG_0^{\spl} \subset \prod_{i=1}^r (\Res_{\kappa_i | \bbF_p} \GSp_{2g})^{e_i}.\]
Recall that the EO index set of $\GSp_{2g}$ is $\{0,1\}^g$ (see \cite[\S 5.4]{ViehmannWedhorn2013}),
so the EO index set $^J W$ of $\scrA_0^{\spl}$ can be identified with $(\{0,1\}^g)^n$, under the isomorphism $\calG_{0,\kappa,der}^{\spl} \otimes_{\bbF_p} k \simeq (\mathrm{Sp}_{2g})^n$. Given $x = (\underline{A}, \underline{\scrF_\bullet}) \in \scrA_0^{\spl}(k)$, the EO type of $x$ can be written as a tuple $a = (a_{i,j}^l)$, where $a_{i,j}^l \in \{0,1\}^g$. For each index $i,j,l$, by the construction in subsection \ref{subsection zip} at the point $x$, we have an $F$-zip structure at the standard tuple $(\Lambda_{i,j}^l \otimes \kappa, C_{0,i,j}^l, D_{0,i,j}^l)$ (of type $\mu_{i,j}^l$) by trivializing the filtrations
\[\calM_{i,j}^l \supset C_{i,j}^l \supset 0, \quad \text{ and }\quad 0 \subset D_{i,j}^l \subset \calM_{i,j}^{l+1}.\]
The classification of $F$-zips of type $\mu_{i,j}^l$ is given by the set $\{0,1\}^g$. This gives the element $a=(a_{i,j}^l)\in (\{0,1\}^g)^n$ with respect to $x$.

We give a description of the $\mu$-ordinary Hasse invariants.
For each $i$ and $l \geq 2$, we have defined a map $V_i^{l}: \calM_i^l \to \calM_i^{l-1}$, which reduces to a morphism
\[V_i^l: \omega_i^{l} \to \omega_i^{l-1}, \quad\text{ where }\quad \omega_i^{l} := \bigoplus_j C_{i,j}^{l}.\]
For $l = 1$, we have the morphism
\[V_i^1: \omega_i^{1} \to (\omega_i^{e_i})^{\otimes p}.\]
Taking summation induces a morphism
\[\det(V_i): \det(\omega_i^{1}) \bigotimes \cdots \bigotimes \det(\omega_i^{e_i})  \to \det(\omega_i^{1}) \bigotimes \cdots \bigotimes \det(\omega_i^{e_i-1}) \bigotimes \det(\omega_i^{e_i})^{\otimes p}, \]
and a global section
\[h_i \in H^0(\scrA_0^{\spl}, \det(\omega_i^{e_i})^{\otimes (p-1)}).\]
Let $h$ be the product of the $h_i$, which is a section of the product of the line bundles $\det(\omega_i^{e_i})^{\otimes (p-1)}$. It coincides with the $\mu$-ordinary Hasse invariant coming from the zip stack.

\subsection{Example: unitary Shimura varieties}\label{subsec ex GU}

Let $L^+$ be a totally real field and $L|L^+$ a totally imaginary quadratic extension. We denote by $c\in \tr{Gal}(L/L^+)$ the non trivial element. Let $I=\Hom(L^+,\ov{\Q})$ and for any $\sigma\in I$ we fix a choice of extension $\tau: L\ra \ov{\Q}$ of $\sigma$. Then we have $\Hom(L,\ov{\Q})=I\coprod I\circ c$. 

Let $V$ be an $L$-vector space of dimension $n$ together with a hermitian form $\langle\cdot,\cdot\rangle$. We assume that this form is not totally definite. Let $G=GU(V,\langle\cdot,\cdot\rangle)$ be the associated reductive group of unitary similitudes  over $\Q$. Let $(p_\tau,q_\tau)_{\tau\in I}$ be the signature of $G_{\mathbb{R}}$ so that $G_{\mathbb{R}}=G(\prod_{\tau\in I}U(p_\tau,q_\tau))$ and we get a standard $h: \Res_{\mathbb{C}|\R}\mathbb{G}_m\ra G_{\mathbb{R}}$.
If we write $G_1=U(V,\langle\cdot,\cdot\rangle)$ as the corresponding unitary group over $\Q$, then $G_1=\Res_{L^+|\Q}U$, where $U$ is the unitary group over $L^+$ defined by $(V,\langle\cdot,\cdot\rangle)$. 

Let $p>2$ be a prime and $\Lambda\subset V_{\Q_p}$ be a PEL $\calO_L$-lattice. We get a parahoric group scheme $\G$ over $\Z_p$. Let $v_1,\cdots,v_r$ be the places of $L^+$ over $p$. For each $v_i$, as in subsection \ref{subsec PEL data} we assume $v_i$ is unramified in $L$, thus we have  the following two cases:
\begin{itemize}
	\item (AL): $v_i$ splits in $L$,
	\item (AU): $v_i$ is inert in $L$.
\end{itemize}
Fix a tame level $K^p\subset G(\mathbb{A}_f^p)$ and let $K=\G(\Z_p)K^p$. Let $E$ be the local reflex field and $F|\Q_p$ a sufficiently large field extension as before.
We get integral models over $\calO_F$ of the associated PEL moduli space
$\scrA_K^\spl\ra \scrA_K\otimes_{\calO_E}\calO_F$.

Over the algebraically closed field $k$,
the group $\calG_{0,k}^{\spl}$ is just the associated similitude subgroup of
$$ \prod_{i,j,l} \GL_{d_i}\times \bbG_m.$$
Moreover, the cocharacter has a decomposition $\mu = \prod_{i,j,l} \mu_{i,j}^l$, where each $\mu_{i,j}^l$ has type $(d_{i,j}^l,d_i - d_{i,j}^l)$, corresponding to some $(p_\tau, q_\tau)$ above.
By the classification of $F$-zips in \cite{MW} and the decomposition of $\calG_0^{\spl}$, the EO index set $^J W$ decomposes as 
\[ ^J W = \prod_{i,l} W_{i}^l, \quad\tr{ where }\quad W_i^l \simeq \prod_{1 \leq j \leq f_i}  (S_{d_{i,j}^l} \times S_{d_i - d_{i,j}^l})\setminus S_{d_i},\]
where $S_{d_i}$ is the permutation group of $d_i$ elements, and $(S_{d_{i,j}^l} \times S_{d_i - d_{i,j}^l})$ is the subgroup of $S_{d_i}$ consisting of the permutations of the first $d_{i,j}^l$ elements  and the permutations of the last $d_i - d_{i,j}^l$ elements. For a precise description of the minimal length representatives for the left coset $(S_{d_{i,j}^l} \times S_{d_i - d_{i,j}^l})\setminus S_{d_i}$, see \cite{MW} subsection 2.6.
The partial order $\preceq$ on $^J W$ is given by the product of the partial orders on $W_{i}^l$.
We remark that different types of $i$ induce different Frobenius actions on $W_i^l$, which induces different partial orders $\preceq$ on $W_i^l$. The construction of Hasse invariants is the same as above.

\section{Extensions to compactifications}
To study applications to cohomology, we need to extend the previous constructions to compactifications. In other words, we need to study the degeneration of $F$-zips with additional structure on splitting models.
Fortunately, the arithmetic compactifications for general splitting models have already been established by Lan in \cite{Lan18}. Here we single out the special case of smooth splitting models. It turns out the properties of these compactifications are as good as that in the unramified setting \cite{Lan13}.

\subsection{Modifications of moduli spaces and integral models}
Before talking about compactifications of splitting models, we need to slightly modify our integral models following \cite{Lan18}. Recall we have fixed a tame level $K^p\subset G(\mathbb{A}_f^p)$ and $K=K^p\G(\Z_p)$. Let $\M_K$ be the integral model over $\calO_E$ constructed as in \cite{Lan16} (denoted by $\overrightarrow{\M}_\H$ in \cite{Lan16, Lan17, Lan18}, where $\H$ denotes the level as in loc. cit.). Then we get an open and closed embedding (cf. \cite{Lan18} Corollary 2.4.8) of $\calO_E$-schemes
\[\M_K\hookrightarrow\scrA_{K^p},\]
thus composing with the closed immersion $\scrA_{K^p}\subset \scrA^{\naive}_{K^p}$ we get also a closed immersion $\M_K\subset \scrA^{\naive}_{K^p}$. On generic fibers, we have an open and closed embedding of $E$-schemes (cf. \cite{Lan13} Lemma 1.4.4.2)
\[\M_{K,E}\hookrightarrow\scrA_{K^p,E}.\]

Let $(\ul{\H},\ul{\scrF}, \ul{\iota})$ be the polarized $\calO_B\otimes\calO_{\M_K}$-modules (see Definition \ref{def triple modules}) associated to the pullback of the universal object over $\scrA^{\naive}$. Here in fact $\scrF=\omega$.
By the notation of \cite{Lan18} Proposition 2.3.10,
we define the corresponding splitting model as the moduli scheme of splitting structures of $(\ul{\H},\ul{\scrF}, \ul{\iota})$ over $\M_K\otimes \calO_F$
\[\M_K^\spl=\mathrm{Spl}^+_{(\ul{\H},\ul{\scrF},\ul{\iota})/\M_K\otimes \calO_F}.\]Then we get an induced open and closed embedding of $\calO_F$-schemes
\[\M_K^\spl\hookrightarrow \scrA_{K^p}^\spl.\]
The difference of these schemes is bounded by the size of failure of the Hasse principle (cf. \cite{Lan13} Remark 1.4.3.12). All the results in previous sections still hold for $\M_K^\spl$ and $\M_K$, as in fact our previous constructions can be done for the related integral models of the associated Shimura varieties.

\subsection{Arithmetic compactifications of splitting models}
We briefly review some basic properties of the toroidal and minimal compactifications of $\M_K^\spl$ constructed by Lan in \cite{Lan18}, which will be used in the following.

Recall that we have a morphism $\M_K^\spl\ra \M_K\otimes \calO_F$. By \cite{Lan16,Lan17}
we choose a toroidal compactification $\M^{\tor}_{K,\Sigma}$ of $\M_K$ over $\calO_E$ associated to a compatible collection $\Sigma$ of cone decompositions.
Over $\M_K$, we have the polarized $\calO_B\otimes\calO_{\M_K}$-modules $(\ul{\H},\ul{\mathscr{F}}, \ul{\iota})$ associated to the universal abelian scheme with additional structure as above. By \cite{Lan18} Proposition 3.1.2, the triple $(\ul{\H},\ul{\mathscr{F}}, \ul{\iota})$ uniquely extends to an $\scrL$-set of polarized $\calO_B\otimes\calO_{\M^\tor_{K,\Sigma}}$-modules $(\ul{\H}^\ext,\ul{\scrF}^\ext, \ul{\iota}^\ext)$. Then we define
\[\M_{K,\Sigma}^{\spl,\tor}=\mathrm{Spl}^+_{(\ul{\H}^\ext,\ul{\scrF}^\ext, \ul{\iota}^\ext)/\M^\tor_{K,\Sigma}\otimes \calO_F}\]
as the scheme of splitting structures of $(\ul{\H}^\ext,\ul{\scrF}^\ext, \ul{\iota}^\ext)$ over $\M^\tor_{K,\Sigma}\otimes \calO_F$.
By \cite{Lan18} Theorem 3.4.1, for suitable choice of $\Sigma$, the scheme $\M_{K,\Sigma}^{\spl,\tor}$ is normal, projective and flat over $\calO_F$, and admits a similar description as the usual compactifications $\M^{\tor}_{K,\Sigma}$. 
\begin{proposition}
For a projective smooth $\Sigma$, the scheme $\M_{K,\Sigma}^{\spl,\tor}$ is projective and smooth over $\calO_F$.
\end{proposition}
\begin{proof}
This follows from \cite{Lan18} Propositions 3.4.13, 3.4.14, and our Proposition \ref{prop smooth spl}.

Alternatively, the smoothness of $\M_{K,\Sigma}^{\spl,\tor}$  can also be seen from the extended local model diagram in the next subsection.
\end{proof}

From now on we assume that $\Sigma$ is projective and smooth. Thus the canonical morphism
\[ \M_{K,\Sigma}^{\spl,\tor}\ra \M_{K,\Sigma}^{\tor}\otimes \calO_F\]
is a resolution of singularities.

By \cite{Lan18} Propositions 4.1.22, 4.2.31 and 4.2.34, we have the minimal compactification $\M_K^{\spl,\min}$ together with a canonical morphism
$\oint: \M_{K,\Sigma}^{\spl,\tor}\ra \M_K^{\spl,\min}$, which fits into a commutative diagram of schemes over $\calO_F$:
\[\xymatrix{
\M_{K,\Sigma}^{\spl,\tor}\ar[r]\ar[d]& \M^\tor_{K,\Sigma}\otimes \calO_F\ar[d]\\
\M_K^{\spl,\min}\ar[r]& \M_K^{\min}\otimes \calO_F.
}\]
By loc. cit. Theorem 4.3.1, the minimal compactification $\M_K^{\spl,\min}$ 
admits a similar description to  $ \M_K^{\min}$.

Let $Z^\spl\subset \M_K^{\spl,\min}$ be a boundary stratum. There is a corresponding boundary stratum $Z\subset \M_K^{\min}$, which is an analogue of $\M_K$ for some boundary PEL datum $(\calO_B, \ast, \Lambda^Z,\langle\cdot,\cdot\rangle^Z, h^Z)$ so that we have the tautological abelian scheme $(B,\lambda,\iota)$ over $Z$ with associated $({}^\sharp\ul{\H},{}^\sharp\ul{\scrF}, {}^\sharp\ul{\iota})$. By construction, we have
\[Z^\spl=\tr{Spl}^+_{({}^\sharp\ul{\H},{}^\sharp\ul{\scrF}, {}^\sharp\ul{\iota})/Z\otimes \calO_F}.\]
Attached to $Z^\spl$ and $Z$, we have the following data:
\begin{itemize}
\item an admissible cone decomposition $\Sigma_Z$ of some cone $\mathbf{P}=\mathbf{P}_Z$, as well as a subset $\Sigma_Z^+$ of $\Sigma_Z$ which forms a cone decomposition of the interior $\mathbf{P}^+$ of $\mathbf{P}$.
\item an arithmetic group $\Gamma=\Gamma_Z$ acting on $\mathbf{P}$ and hence aslo on $\Sigma_Z$; the open cone $\mathbf{P}^+$ and the corresponding $\Sigma_Z^+$ are stable under the action of $\Gamma$. As in \cite{LS18} we may and we shall assume that for each $\sigma\in\Sigma_Z^+$, the stabilizer $\Gamma_\sigma$ is trivial.
\item a finite free abelian group $S$; let $\mathbf{E}=\mathbf{E}_Z$ be a split torus over $\Z$ with character group $S$.
\item a normal scheme $C^\spl$ which is flat over $\calO_F$, together with a proper surjective morphism $C^\spl\ra Z^\spl$. 
\item a morphism of schemes $\Xi^\spl\ra C^\spl$ which is an $\mathbf{E}$-torsor; for each $\sigma\in \Sigma_Z$, we have an affine toroidal embedding $\Xi^\spl\hookrightarrow\Xi^\spl(\sigma)$ over $C^\spl$ with a closed subscheme $\Xi_{\sigma}^\spl$.
\item for each representative $\sigma\in\Sigma_Z^+$ of an orbit $[\sigma]\in \Sigma_Z^+/\Gamma$, let $Z_{[\sigma]}^\spl\subset \M_{K, \Sigma}^{\spl,\tor}$ be the corresponding toroidal boundary stratum, and $\mathfrak{X}^\spl_\sigma:=(\Xi^\spl(\sigma))^{\wedge}_{\Xi^\spl_{\sigma}}$ the formal completion, then there is a canonical isomorphism of formal schemes \[\mathfrak{X}^\spl_\sigma\simeq (\M_{K, \Sigma}^{\spl,\tor})^{\wedge}_{Z^\spl_{[\sigma]}}.\]
\item let $\Xi_Z^\spl$ be the full toroidal embedding attached to $\Sigma_Z$ and \[\mathfrak{X}^\spl:=\mathfrak{X}_Z^\spl=(\Xi_Z^\spl)^{\wedge}_{\quad\cup_{\tau\in\Sigma_Z^+}\Xi_{\tau}^\spl}\] the formal completion, then there is a canonical isomorphism
	\[\mathfrak{X}^\spl/\Gamma\simeq (\M_{K, \Sigma}^{\spl,\tor})^{\wedge}_{\quad\cup_{\tau\in\Sigma_Z^+/\Gamma}Z_{[\tau]}^\spl}.\]
	Note the (disjoint) union $\cup_{\tau\in\Sigma_Z^+/\Gamma}Z_{[\tau]}^\spl$ is exactly the preimage of $Z^\spl$ under the natural projection $\oint: \M_{K,\Sigma}^{\spl,\tor}\ra \M_K^{\spl,\min}$.
	\item there are similar and parallel objects $C, \Xi$, etc. for the boundary stratum $Z$, such that $C^\spl=\tr{Spl}^+_{({}^\sharp\ul{\H},{}^\sharp\ul{\scrF}, {}^\sharp\ul{\iota})/C\otimes \calO_F}, \quad \Xi^\spl=\tr{Spl}^+_{({}^\sharp\ul{\H},{}^\sharp\ul{\scrF}, {}^\sharp\ul{\iota})/\Xi\otimes \calO_F}$, etc., see \cite{Lan18} Lemma 3.2.4.
\end{itemize}

The same arguments as in the proof of Proposition \ref{prop smooth spl} show that
\begin{proposition}\label{prop sm boundary}
Each boundary stratum $Z^\spl$ is smooth over $\calO_F$.
\end{proposition}

\subsection{Canonical extensions of automorphic vector bundles}\label{subsection can ext}
Since $G$ is split over $F$, it defines a reductive group scheme over $\calO_F$, denoted by $\G^\spl$ (or $G$) as before. Moreover, recall that as in the paragraph above Corollary \ref{coro ext vb} the parabolic subgroup $P_\mu$ extends to a paraholic group scheme of $\G^\spl$ over $\calO_F$. Thus
the flag variety $\Fl(G,\mu)_F$ extends canonically to a smooth scheme $\G^\spl/P_\mu$ (the integral flag variety) over $\calO_F$ which is nothing else but $\prod_{i\in I, 1\leq j\leq f_i, 1\leq l \leq e_i}\mathbb{M}_{i,j}^l$. We still denote by $\Fl(G,\mu)$ the integral flag variety over $\calO_F$. 

Now we have an extension of the local model diagram of schemes over $\calO_F$:
\[\xymatrix{
& \wt{\M_{K,\Sigma}^{\spl, \tor}}\ar[rd]^q \ar[ld]_\pi&\\
\M^{\spl,\tor}_{K,\Sigma}& &\Fl(G,\mu)
}\]
where $\pi$ is a $\calG^\spl$-torsor and $q$ is $\calG^\spl$-equivariant.
For any representation $(V,\eta)\in \Rep_{\calO_F}P_\mu$, we get the associated $G$-equivariant vector bundle $\ul{V}$ on $\Fl(G,\mu)$.  Via the above diagram, we get a vector bundle $\V^\can=\V_\eta^\can$ on $\M^{\spl,\tor}_{K,\Sigma}$, which we call the canonical extension of the vector bundle $\V=\V_\eta$ on $\M_K^{\spl}$. Let $D'$ be the effective Cartier divisor of \cite{Lan18} Corollary 4.4.4 with $D'_{red}=\M^{\spl,\tor}_{K,\Sigma}\setminus \M_K^{\spl}$ and set $\V^\sub=\V^\can(-D')$, which we call the sub-canonical extension of $\V=\V_\eta$. 

Let $\eta\in X^\ast(T)^+_L$ and $V\in \Rep_{\calO_F}L$ the irreducible representation of $L$ of highest weight $\eta$. Recall the map $\oint: \M^{\spl,\tor}_{K,\Sigma}\ra \M_K^{\spl,\min}$.
\begin{proposition}\label{prop rel vanish}
For any $i>0$, we have $R^i\oint_{\ast}\V_\eta^\sub=0$.
\end{proposition}
\begin{proof}
This will follow from \cite{Lan18} Theorem 4.4.9, once we have verified that $\V^\can=\V_\eta^\can$ is formally canonical in the sense of loc. cit. (The definition of this notion there is given by the corresponding analogue of \cite{Lan17} Definition 8.5.) Recall that this means: for any boundary stratum $Z^\spl$ of $\M_K^{\spl,\min}$, and any geometric point $\ov{x}$ over $Z^\spl$, there exists a coherent sheaf $
\V_{0,\ov{x}}$ over $C_{\ov{x}}^{\spl,\wedge}$ such that
\begin{enumerate}
	\item for any $\sigma\in\Sigma_Z^+$, the pullback $\V^{\can,\wedge}$ to $\mathfrak{X}_{\sigma,\ov{x}}^\wedge$ is of the form
	\[\wh{\bigoplus}_{\ell\in\sigma^\vee}\Big((\Psi^\spl_Z(\ell))^{\wedge}_{\ov{x}}\otimes\V_{0,\ov{x}}\Big),\]
	\item $\V_{0,\ov{x}}$ admits a finite filtration whose graded pieces are isomorphic to pullbacks of quasi-coherent sheaves over $\calO_F$ via the structural morphism $C_{\ov{x}}^{\spl,\wedge}\ra \Spec\,\calO_F$.
\end{enumerate}

One needs to check that in the smooth reduction case the sheaf $\V^\can$ satisfies the above two conditions. This can be achieved by modifying the arguments in the proof of \cite{Lan16a} Proposition 5.6, where we take account of the splitting structures everywhere.
\end{proof}

\subsection{Extensions of $\G_0^\spl$-zips and Hasse invariants}
Recall that $\kappa$ is the residue field of $\calO_F$ and $k=\ov{\kappa}$.
Let $\M^{\spl}_{K,0}=\M^{\spl}_{K}\otimes_{\calO_F} k$, $\M^{\spl,\tor}_{K,\Sigma,0}=\M^{\spl,\tor}_{K,\Sigma}\otimes_{\calO_F} k$ and $\M^{\spl,\min}_{K,0}=\M^{\spl,\min}_{K}\otimes_{\calO_F} k$ be the geometric special fibers. By section \ref{section Fzip} we have a smooth surjective morphism \[\zeta: \M^{\spl}_{K,0}\ra \calG^\spl_0\tr{-}\mathrm{Zip}^\mu_k.\]

Before discussing extensions of zips to the boundary, we need some group theoretical preparation. 
For each $Z^\spl\subset \M^{\spl,\min}_{K}$ and the associated $Z\subset \M^{\min}_K$, recall that we have the boundary PEL-type $\calO_B$-lattice $(\Lambda^Z, \langle\cdot,\cdot\rangle^Z,h^Z)$. Over $\Q$, we have the associated parabolic subgroups \[Q=Q_Z\subset P=P_Z\subset G,\]
cf. \cite{LS18} Lemma 3.3.6 and Definition 3.3.8. More precisely, attached to the boundary stratum $Z\subset \M^{\min}_K$ we have a symplectic filtration $0\subset V_{-2}\subset V_{-1}\subset V_0=V$ satisfying the condition as in Lemma 3.3.6 of \cite{LS18}. The group $P=P_Z$ is defined as the stabilizer of this filtration. The group $Q=Q_Z$ is defined as the kernel of the homomorphism $P\ra \GL(gr_{-2}V)\times\GL(gr_{0}V),\; (g,r)\mapsto (r^{-1}gr_{-2}(g),gr_0(g))$, where $P\ra \mathbb{G}_m, \; (g,r)\mapsto r$ is the similitude character. The groups $P$ and $Q$ have the same unipotent radical $U$, such that $M_h:=Q/U$ is the reductive group for the rational PEL datum corresponding to $Z$, and $M:=P/U=M_h\times M_l$ where $M_l\simeq P/Q$ is a reductive group factor of the Levi $M$, see \cite{LS18} Definition 3.3.8. 

Consider the reductive group $\G_0^\spl$ over $\bbF_p$. We want to adapt the construction of \cite{And22} subsection 3.3 to our setting.
Consider the induced symplectic filtration \[0\subset \Lambda_{-2}\subset \Lambda_{-1}\subset \Lambda_0=\Lambda\] so that \[
\Lambda^Z=gr_{-1}\Lambda=\Lambda_{-1}/\Lambda_{-2},\] and the decomposition $\Lambda=\bigoplus_{1 \leq i \leq r} \Lambda_i^{m_i}$, we get similar filtrations on each $\Lambda_i$. Using the fact that each $\Lambda_i$ is self-dual over $\calO_{F_i}$, by similar construction as in subsection \ref{subsec PEL data}, we have parabolic subgroups
\[\mathcal{Q}_0^\spl\subset \mathcal{P}_0^\spl\subset \G_0^\spl \]
with the same unipotent radical $U^\spl_0$. Let $M^\spl_0=\mathcal{P}_0^\spl/U^\spl_0$ and $M^\spl_{h,0}=\mathcal{Q}_0^\spl/U^\spl_0$. Then $M^\spl_0=M^\spl_{h,0}\times M^\spl_{l,0}$ where $M^\spl_{l,0}=\mathcal{P}_0^\spl/\mathcal{Q}_0^\spl$.
Recall that we have the cocharacter $\mu$ of $\G_0^\spl$ which is defined over $\kappa$, and $P_+:=P_\mu\subset \G_{0,\kappa}^\spl$ the associated parabolic subgroup. Set $\ov{P}:= \mathcal{P}_{0,\kappa}^\spl \cap P_+$. We get the inclusions \[\ov{P}\subset P_+\subset  \G_{0,\kappa}^\spl.\] Consider the induced cocharacter $\mu_Z$ of $M^\spl_{h,0}$ from $h^Z$. We get a natural morphism \[\gamma: M^\spl_{h,0}\tr{-Zip}^{\mu_Z}_\kappa\ra M^\spl_{0}\tr{-Zip}^{\mu_Z}_\kappa,\] where we denote by the same notation $\mu_Z$ for  the induced cocharacter of $M^\spl_{0}$. Identifying $M^\spl_{0}$ as the standard Levi of $\mathcal{P}_0^\spl$, the inclusion $M^\spl_{0}\subset \G_0^\spl$ induces a morphism \[\upsilon: M^\spl_{0}\tr{-Zip}^{\mu_Z}_\kappa\ra [E_{\ov{\mathcal{Z}}}\setminus \G_{0,\kappa}^\spl],\] where $\ov{\mathcal{Z}}=(\ov{P}, \ov{P}^{(p)}, \varphi)$ is the algebraic zip data for $\G_{0,\kappa}^\spl$ associated to $\ov{P}$. Then the natural inclusion $\ov{P}\subset P_+$ induces a morphism \[\epsilon: [E_{\ov{\mathcal{Z}}}\setminus \G_{0,\kappa}^\spl]\ra \G^\spl_0\tr{-}\mathrm{Zip}^\mu_\kappa.\]
We define \[\tau_Z=\epsilon\circ\upsilon\circ\gamma: M^\spl_{h,0}\tr{-Zip}^{\mu_Z}_\kappa\ra \G^\spl_0\tr{-}\mathrm{Zip}^\mu_\kappa.\]
In the following we also write $\G^\spl_{Z,0}=M^\spl_{h,0}$ to indicate its analogue with $\G^\spl_0$.

Given the canonical extensions of automorphic vector bundles to $\M^{\spl,\tor}_{K,\Sigma,0}$, the constructions of subsection \ref{subsection zip} can be generalized.
\begin{theorem}\label{thm zip map ext}
\begin{enumerate}
	\item 
The $\calG^\spl_0$-zip of type $\mu$ on $\M^{\spl}_{K,0}$ extends to a $G(\bbA_f^p)$-equivariant $\calG^\spl_0$-zip of type $\mu$ on $\M^{\spl,\tor}_{K,\Sigma,0}$.
\item The induced map $\zeta^\tor: \M^{\spl,\tor}_{K,\Sigma,0}\ra \calG^\spl_0\tr{-}\mathrm{Zip}^\mu_k$ is smooth.
\end{enumerate}
\end{theorem}
\begin{proof}
For (1), we just repeat the construction in subsection \ref{subsection zip} starting from $(\ul{\H}^\ext,\ul{\scrF}^\ext, \ul{\iota}^\ext)$ together with its universal splitting structure.

For (2), we follow the idea in the proof of \cite{And22} Theorem 3.1. 
	Let $x\in \M^{\spl,\tor}_{K,\Sigma,0}$ be a closed point and $C=\wh{\calO}_{\M^{\spl,\tor}_{K,\Sigma,0},x}$. We need only check that the induced morphism \[\zeta_C:  \Spec\,C\ra \calG^\spl_0\tr{-}\mathrm{Zip}^\mu_k\] is smooth. If $x\in \M^{\spl}_{K,0}$, this has been done in Proposition \ref{prop sm zeta}. So we may assume $x\in Z_{[\sigma]}^\spl$ for a boundary stratum. Then there is a corresponding point $x'\in \Xi^\spl(\sigma)$ such that $C=\wh{\calO}_{\Xi^\spl(\sigma),x'}$. Let $y\in Z^\spl$ be the image of $x$ and $D=\wh{\calO}_{Z^\spl,y}$. We get an induced morphism \[\oint: \Spec\,C\ra \Spec\,D.\] Moreover, we have the boundary version $\zeta_Z: Z^\spl\ra \G_{0,Z}^\spl\tr{-Zip}^{\mu_Z}_k$ and the analogue \[\zeta_D: \Spec\,D\ra \G_{0,Z}^\spl\tr{-Zip}^{\mu_Z}_k\] of $\zeta_C$. Recall that by Proposition \ref{prop sm boundary}, $Z^\spl$ is smooth. We have also a natural morphism \[\tau_Z: \G_{0,Z}^\spl\tr{-Zip}^{\mu_Z}_\kappa\ra \calG^\spl_0\tr{-}\mathrm{Zip}^\mu_\kappa.\] To show the smoothness of $\zeta_C:  \Spec\,C\ra \calG^\spl_0\tr{-}\mathrm{Zip}^\mu_\kappa$, we proceed in four steps.
	
	\emph{Step 1.} We show $\zeta_C=\tau_Z\circ\zeta_D\circ\oint$. In other words, the following diagram commutes
	\[\xymatrix{\Spec\,C\ar[d]_\oint\ar[r]^{\zeta_C}& \calG^\spl_0\tr{-}\mathrm{Zip}^\mu_k.\\
		\Spec\,D\ar[r]^{\zeta_D}& \G_{0,Z}^\spl\tr{-Zip}^{\mu_Z}_k\ar[u]_{\tau_Z}
}\] We need to analyze the structure of the $\G_0^\spl$-zip on $C$ in terms of the associated $\G_{0,Z}^\spl$-zip. Consider the restriction of $(\ul{\H}^\ext,\ul{\scrF}^\ext, \ul{\iota}^\ext)$ on $Z_{[\sigma]}^\spl$, which we denote by $(\ul{\H}^\natural,\ul{\scrF}^\natural, \ul{\iota}^\natural)$. By \cite{Lan18} proof of Proposition 3.3.21, there are isomorphisms \[\scrF^\natural\simeq {}^\sharp \scrF\oplus {}^\flat\scrF, \quad \tr{where} \quad {}^\flat\scrF=\omega_{T^\vee/Z_{[\sigma]}}.\] Moreover, the filtration $\ul{\scrF_\bullet}^\natural$ on $\scrF^\natural$ induces filtrations on $ {}^\sharp \scrF$ and ${}^\flat\scrF$ as \cite{Lan18} Lemma 3.3.11, and by loc. cit. Corollary 3.3.16 the filtration on ${}^\flat\scrF$ is independent of the filtration on $\scrF^\natural$. By the construction in section 3, the $F$-zip $\mathcal{M}^\natural$ (constructed from $(\ul{\H}^\natural,\ul{\scrF}^\natural, \ul{\iota}^\natural)$) is a direct sum of the $F$-zip $\mathcal{M}_Z$ constructed from $({}^\sharp\ul{\H},{}^\sharp\ul{\scrF}, {}^\sharp\ul{\iota})$ and the $F$-zip constructed from ${}^\flat\scrF$. Translating further into the language of $\calG^\spl_0$-zips, we have the factorization of $\zeta_C$.
	
	\emph{Step 2.} $\oint: \Spec\,C\ra \Spec\,D$ is smooth. For this, we apply the local diagrams for $\M^{\spl,\tor}_{K,\Sigma}$ and $Z^\spl$ to realize $C$ and $D$ as local rings of the corresponding local models, which are also completions along identities of some unipotent subgroups, see the proof of Proposition \ref{prop sm zeta}. Then one can prove the smoothness of $\oint$ by checking that on tangent spaces it induces a projection.
	
	\emph{Step 3.} $\zeta_D: \Spec\,D\ra \G_{0,Z}^\spl\tr{-Zip}^{\mu_Z}_k$ is smooth. As $Z^\spl$ is smooth, the smoothness of $\zeta_D$ follows from the boundary version of Proposition \ref{prop sm zeta}.
	
	\emph{Step 4.} $\tau_Z: \G_{0,Z}^\spl\tr{-Zip}^{\mu_Z}_k\ra \calG^\spl_0\tr{-}\mathrm{Zip}^\mu_k$ is smooth.
	This follows from a similar argument as the proof of \cite{And22} Lemma 3.4.
\end{proof}

We write $X^\tor=\M^{\spl,\tor}_{K,\Sigma,0}$ and the fiber of $\zeta^\tor$ at $w$ as $X^\tor_w$. Then we get the EO stratification \[X^\tor=\coprod_{w\in {}^JW}X^\tor_w.\] For each $w\in {}^JW$ let $\ov{X^\tor_w}$ be the Zariski closure of $X^\tor_w$ in $X^\tor$.
\begin{corollary}\label{cor ext hasse}
	The Hasse invariant $h_w\in H^0(\ov{X_w}, \omega_{\Hdg}^{N_w})$ of Corollary \ref{cor hasse inv} extends to a $G(\bbA_f^p)$-equivariant section $h_w^\tor \in H^0(\ov{X^\tor_w}, \omega_{\Hdg}^{N_w})$ with non-vanishing locus being precisely $X_w^\tor$.
\end{corollary}

The proof of Theorem \ref{thm zip map ext} (2) actually gives us the following (expected) description of $\zeta^\tor$ on the boundary. Let $Z_0^\spl\subset \M^{\spl,\min}_{K,0}$ be a boundary stratum and $Z_0^{\spl,\tor}:=\oint^{-1}(Z_0^\spl)$. We get an induced morphism $\oint: Z_0^{\spl,\tor}\ra Z_0^\spl$.
\begin{proposition}
The restriction \[\zeta^\tor|_{Z_0^{\spl,\tor}}= \tau_Z\circ\zeta_Z\circ\oint,\] where $\tau_Z$ and $\zeta_Z$ are as in the proof of Theorem \ref{thm zip map ext} (2).
\end{proposition}

\begin{corollary}\label{cor min stratum non boundary}
Let $e\in {}^JW$ be the minimal element (then $\ell(e)=0$). The associated EO stratum $X_e$ does not intersect with the boundary of $X^\tor$, in other words we have $X_e=X_e^\tor$. In particular, the conclusions of Proposition \ref{prop length Hasse} hold for $X^\tor$.
\end{corollary}
\begin{proof}
Similar to \cite{And22} Corollary 3.7, this follows from the fact that for each boundary $Z$, $\tau_Z: \G_{0,Z}^\spl\tr{-Zip}^{\mu_Z}_k\ra \calG^\spl_0\tr{-}\mathrm{Zip}^\mu_k$ is smooth, thus open, therefore its image does not contain the closed point of $\calG^\spl_0\tr{-}\mathrm{Zip}^\mu_k$. Indeed, if its image contained the closed point, it would be surjective and any maximal chain of points of $|\calG^\spl_0\tr{-}\mathrm{Zip}^\mu_k|$ would be the image of a maximal chain of points of $|\G_{0,Z}^\spl\tr{-Zip}^{\mu_Z}_k|$. This is impossible since $\dim\,Z<\dim\,X$ with $X=\M^\spl_{K,0}$ (see also the proof of \cite{And22} Corollary 3.7).
\end{proof}

Next, we discuss the well-positionedness of EO strata in the sense of  \cite{LS18} Definition 2.2.1, which is in fact closely related to the smoothness of $\zeta^\tor$. Let $Z^\spl$ be a boundary stratum of $\M^{\spl,\min}_{K}$. Recall for each $\sigma\in \Sigma_Z^+$, we have $\Xi^\spl(\sigma)$ and its closed subscheme $\Xi^\spl_{\sigma}$. Consider the closed subscheme \[\Xi^\spl(\sigma)^+:=\cup_{\tau\in\Sigma_Z^+,\ov{\tau}\subset\ov{\sigma}}\Xi^\spl_\tau\] and the formal completion \[\mathfrak{X}_\sigma^{\spl,\circ}:=(\Xi^\spl(\sigma))^\wedge_{\Xi^\spl(\sigma)^+}.\] By \cite{LS18} Proposition 2.1.3, the formal scheme $\Xi^{\spl}_Z$ admits an open covering by $\mathfrak{X}_\sigma^{\spl,\circ}$ for $\sigma$ running through elements of $\Sigma_Z^+$, and for each $\sigma$ we have an isomorphism \[\mathfrak{X}_\sigma^{\spl,\circ}\simeq (\M_{K, \Sigma}^{\spl,\tor})^\wedge_{\quad\cup_{\tau\in\Sigma_Z^+,\ov{\tau}\subset\ov{\sigma}}Z_{[\tau]}}.\] For any open affine formal subscheme $\mathrm{Spf}\,R$ of $\mathfrak{X}_\sigma^{\spl,\circ}$, let $\mathbf{W}=\Spec\,R$, then we get induced morphisms \[\mathbf{W}\ra \M_{K, \Sigma}^{\spl,\tor}, \quad \mathbf{W}\ra \Xi^\spl(\sigma).\] By these morphisms, the two stratifications of $\mathbf{W}$ induced respectively by those of $\M_{K, \Sigma}^{\spl,\tor}$ and $\Xi^\spl(\sigma)$ coincide. Let \[\mathbf{W}^0\subset \mathbf{W}\] be the open stratum, which is the preimage of $\M_{K}^\spl$ and $\Xi^\spl$ under the above morphisms. Now consider the geometric special fiber $\M_{K, \Sigma,0}^{\spl,\tor}$ and we denote by the same notations $Z^\spl, C^\spl, \Xi^\spl, \mathbf{W}$, etc. the corresponding objects base-changed to $k$. By \cite{LS18} Definition 2.2.1, a locally closed subset \[Y\subset \M_{K,0}^\spl\] is called well-positioned if there exists a collection  \[(Y_Z^\sharp)_Z\] indexed by the boundary strata of $\M_{K,0}^{\spl,\min}$, where $Y_Z^\sharp\subset Z^\spl$ is a locally closed subset (which may be empty), such that for any $\Sigma$, any $\sigma\in\Sigma_Z^+$ and any $\mathrm{Spf}\,R$ as above, if $Y_Z^\sharp\neq\emptyset$, then under the induced morphisms \[\mathbf{W}^0\ra\M_{K,0}^\spl,\quad \mathbf{W}^0\ra Z^\spl\]  the preimages of $Y$ and $Y_Z^\sharp$ in $\mathbf{W}^0$ coincide. Here $\mathbf{W}^0\ra Z^\spl$ is the composition
	\[\mathbf{W}^0\ra \Xi^\spl\ra C^\spl\ra Z^\spl.\] By \cite{LS18} Lemma 2.2.2, it suffices to verify the condition for just one collection of cone decompositions $\Sigma$ and some affine open covering $\mathrm{Spf}\,R$ of each $\mathfrak{X}_\sigma^{\spl,\circ}$.
\begin{proposition}\label{prop well-position}
For each $w\in {}^JW$, the locally closed subset $X_w$ of $X=\M_{K,0}^\spl$ is well-positioned.
\end{proposition}
\begin{proof}
We can translate \cite{LS18} Lemma 3.4.3 into zips. 
Indeed, the setting of loc. cit. includes the splitting model case. Specializing to the case of $p$-torsion ($n=1$ there), we see that the triple $(A[p], \lambda, \iota)$ determines and is determined by the isomorphism classes of $(X, Y, \phi: Y\ra X)$ and of $(B[p], \lambda_B,\iota_B)$. The splitting structure on $\omega_A$ is induced by those on $X, Y$ and $\omega_B$.
For the torus part, the splitting structure is unique, cf. \cite{Lan18} Corollary 3.3.16. Thus by construction, the $\G_0^\spl$-zip attached to $(A[p], \lambda, \iota)$ determines and is determined by the isomorphism classes by the $\G_{0,Z}^\spl$-zip attached to $(B[p], \lambda_B,\iota_B)$. See also the arguments of Step 1 in the proof of Theorem \ref{thm zip map ext}.
Therefore the EO strata are well-positioned.
\end{proof}

\begin{remark}
In \cite{LS18} subsection 3.5, Lan and Stroh firstly proved that EO strata in their case (Nm) for a good prime $p$ are well-positioned. They then deduced that $\zeta^\tor$ is smooth if $\zeta$ is, cf. loc. cit. Corollary 3.5.8. Following their idea, we sketch how to deduce the smoothness of $\zeta^\tor$ from Proposition \ref{prop well-position} as follows:

By Proposition \ref{prop sm zeta}, each $X_w$ is smooth. As $X_w$ is well-positioned by Proposition \ref{prop well-position}, the partial toroidal  compactification $(X_{w})^\tor_\Sigma$ of $X_w$ (in the sense of \cite{LS18} Definition 2.3.1 and Theorem 2.3.2)  is also smooth by \cite{LS18} Proposition 2.3.13. By construction, we see that $(X_{w})^\tor_\Sigma$ equals to the fiber of $\zeta^\tor$ at $w$. As each fiber of $\zeta^\tor$ is smooth, we get $\zeta^\tor$ is smooth once we know it is flat. The actual argument of \cite{LS18} (proof of Corollary 3.5.8 there) says that \'etale locally $\zeta^\tor$ factors through $C^\spl$ (and $Z^\spl$), which follows from Proposition \ref{prop well-position}.
\end{remark}

\section{Application to Galois representations}
In this section we study the coherent cohomology of the smooth schemes $\M_{K,\Sigma}^{\spl,\tor}$.
We deduce some consequences to Hecke algebras and Galois representations following the same treatments of \cite{GoldringKoskivirta2019}.

\subsection{Hecke actions on coherent cohomology}

Recall that $K=K^p\G(\Z_p)$ is our level of moduli spaces.
Let $S$ be the finite set of primes $\ell$ where $K_\ell$ is not hyperspecial. Consider the Hecke algebra \[\H^S=\bigotimes_{v\notin S}'\H_v,\] the restricted tensor product of the spherical Hecke algebras $\H_v=\Z_p[K_v\backslash G(\Q_v)/K_v]$ outside $S$.
We assume that $G$ is ramified over $\Q_p$, i.e. $p\in S$, since otherwise all the following discussions are covered by \cite{GoldringKoskivirta2019}.

Consider also the Hecke algebra $\H_K$. There is a natural morphism $\H^S\ra\H_K$. Recall for any $\calO_F$-representation $V$ of $L$, we have the automorphic vector bundle $\V^\sub$ on the smooth toroidal compactification $\M_{K,\Sigma}^{\spl,\tor}$ as in subsection \ref{subsection can ext}.
In the following
we describe the action of $\H_K$ on the coherent cohomology groups $H^i(\M_{K,\Sigma}^{\spl,\tor}, \V^\sub)$, so that we get an induced action of $\H^S$.
Since the Hecke algebra $\H_K$ is generated by characteristic functions of $KgK\in K\backslash G(\bbA_f)/K$ with $g\in G(\bbA_f)$, it suffices to describe the action of $KgK$ on $H^i(\M_{K,\Sigma}^{\spl,\tor}, \V^\sub)$. Let $K_g=K\cap gKg^{-1}$. By \cite{Lan18} Proposition 2.4.17, we get the associated Hecke correspondence 
\[\xymatrix{
&\M_{K_g}^\spl\ar[ld]_{p_1}\ar[rd]^{p_2}&\\
\M_{K}^\spl& &\M_{K}^\spl,
}\]
where $p_1$ is the natural projection, $p_2$ is the composition of natural projection with $g: \M_{K_g}^\spl\st{\sim}{\ra}\M_{g^{-1}K_gg}^\spl$.
For $i=1,2$, the induced $\Sigma_g^i=p_i^\ast\Sigma$ are admissible finite rpcd for the level $K_g$. Let $\Sigma_g$ be a common smooth refinement of $\Sigma_g^1$ and $\Sigma_g^2$. Then by \cite{Lan18} Proposition 3.4.10, we get an extended Hecke correspondence
\[\xymatrix{
	&\M_{K_g,\Sigma_g}^{\spl,\tor}\ar[ld]_{q_1}\ar[rd]^{q_2}&\\
	\M_{K,\Sigma}^{\spl,\tor}& &\M_{K,\Sigma}^{\spl,\tor},
}\]
where $q_i$ is the composition  \[ \M_{K_g,\Sigma_g}^{\spl,\tor}\st{r_i}{\lra} \M_{K_g,\Sigma_g^i}^{\spl,\tor}\st{\pi_i}{\lra} \M_{K,\Sigma}^{\spl,\tor}.\]
The morphisms $r_i$ satisfy $R^jr_{i,\ast}\calO_{\M_{K_g,\Sigma_g}^{\spl,\tor}}=0$ for $j>0$ and $r_{i,\ast}\calO_{\M_{K_g,\Sigma_g}^{\spl,\tor}}=\calO_{\M_{K_g,\Sigma_g^i}^{\spl,\tor}}$ by \cite{Lan18} Proposition 3.4.10 and \cite{Lan17} Proposition 7.5. By \cite{Lan18} Proposition 3.4.14, the schemes $\M_{K_g,\Sigma_g^i}^{\spl,\tor}$ are Cohen-Macaulay as $\M_{K_g}^{\spl}$ is. Arguing as \cite{GoldringKoskivirta2019} 8.1.5, we get that the morphisms $\pi_i$ are finite flat. 
Thus we get a trace map \[\tr{tr}\pi_i: \pi_{i,\ast}\calO_{\M_{K_g,\Sigma_g^i}^{\spl,\tor}}\lra \calO_{\M^{\spl,\tor}_{K, \Sigma}},\]
 which induces the associated Hecke operator: for each $i\geq 0$
\[T_g: H^i(\M_{K,\Sigma}^{\spl,\tor}, \V^\sub)\lra H^i(\M_{K,\Sigma}^{\spl,\tor}, \V^\sub).\]

Let $\varpi$ be a uniformizer of $\calO_F$. Recall the Levi subgroup $L$ of $P=P_\mu$ over $\calO_F$. For any $i\geq 0, n\geq 1$ and $(V,\eta)\in \Rep_{\calO_F}L$, consider the vector bundle $\V_\eta^\sub$ on $\M_{K, \Sigma,\calO_F/\varpi^n}^{\spl,\tor}$.
we get an action
\[\H^S\lra \End\big(H^i(\M_{K, \Sigma,\calO_F/\varpi^n}^{\spl,\tor}, \V_\eta^\sub)\big). \]
Let $\H^{i,n}_\eta$ be its image.

\subsection{Factorizations to $H^0$}

For any $i\geq 0, n\geq 1$ and $\eta\in X^\ast(T)_L^+$, as in \cite{GoldringKoskivirta2019} we consider the following set
\[F(i,n,\eta)=\{\eta'\in X^\ast(T)_L^+\,|\,\H^S\ra \H^{i,n}_\eta\,\tr{factors through}\,\H^S\ra \H^{0,n}_{\eta'}\}.\]With all the ingredients at hand, by the method of \cite{GoldringKoskivirta2019} we have the same consequences as Theorem 8.2.1 of loc. cit.
\begin{theorem}\label{Thm Hecke alg}
	For any $(i,n,\eta)$ as above, we have
	\begin{enumerate}
		\item There exists an arithmetic progression $A$ such that $\eta+a\eta_\omega\in F(i,n,\eta)$ for all $a\in A\cap \Z_{\geq 1}$.
		\item Let $\mathcal{C}$ be the cone defined in \cite{GoldringKoskivirta2019} 3.4.3. Then for all $\nu\in\mathcal{C}$ and $\eta_1\in  F(i,n,\eta)$, there exists $m=m(\nu,n)\in\Z_{\geq 1}$ such that for all $j\in\Z_{\geq 1}$, we have $\eta_1+jm\nu\in F(i,n,\eta)$.
		\item For all $\delta\in \mathbb{R}_{\geq 0}$, $F(i,n,\eta)$ contains a $\delta$-regular character in the sense of \cite{GoldringKoskivirta2019} Definition N.5.5.
	\end{enumerate}
\end{theorem}
The proof of the above theorem is by the same arguments as in  \cite{GoldringKoskivirta2019} subsections 7.2, 7.3, 8.3 and 9.2. In particular, one plays with the machinery of Hasse-regular sequences based on our Corollaries \ref{cor hasse inv}, \ref{cor ext hasse} and \ref{cor min stratum non boundary}, Propositions \ref{prop length Hasse} and \ref{prop rel vanish},  and one applies the associated flag space to increase the regularity. For the reader's convenience, we recall that a Hasse regular sequence of length $r$ with $0\leq r\leq d=\dim \M^{\spl}_{K,F}$ on $X:=\M^{\spl,\tor}_{K,\Sigma, \calO_F/\varpi^n}$ (in the sense of \cite{GoldringKoskivirta2019} Definition 7.2.1) is given by a filtration of  closed subschemes
\[X=Z_0\supset Z_1\supset\cdots\supset Z_r \] together with some integers $a_j$ and global sections $f_j\in H^0(Z_j, \omega_{\Hdg}^{a_j})$,
such that each $f_j$ is a lifting of a length Hasse invariant, and  for each $0\leq j\leq r-1$, we have $Z_{j+1}=V(f_j)$. More precisely,
on the reduced locus
\[X_{\tr{red}}=\coprod_{w\in {}^JW}X_w\supset Z_{j,\tr{red}}=\coprod_{w\in {}^JW, \ell(w)\leq d-j}X_w.\] By Proposition \ref{prop length Ha}, there exists a large integer $N_{d-j}$ and a length $d-j$ Hasse invariant \[h_{d-j}\in H^0(Z_{j,\tr{red}}, \omega_{\Hdg}^{N_{d-j}})\] such that its vanishing locus is $Z_{j+1,\tr{red}}=\coprod_{w\in {}^JW, \ell(w)\leq d-j-1}X_w$. Then one requires \[f_j\in H^0(Z_j, \omega_{\Hdg}^{a_j})\] to be a lifting of certain power of $h_{d-j}$ (which exists by \cite{GoldringKoskivirta2019} Theorem 5.1.1) for some integer $a_j\geq N_{d-j}$. From a regular Hasse sequence of length $r$, we get an exact sequence of sheaves over $Z_{r-1}$: 
\[ 0\lra \V_\eta^\sub\otimes\omega_{\Hdg}^s\st{\cdot f_{r-1}}{\longrightarrow} \V_\eta^\sub\otimes\omega_{\Hdg}^{\otimes (a_{r-1}+s)} \longrightarrow \V_\eta^\sub\otimes\omega_{\Hdg}^{\otimes (a_{r-1}+s)}|_{Z_r} \lra 0,\]
where $\omega_{\Hdg}$ is the Hodge line bundle and $s$ is an integer.
From here,  one gets congruences between cohomology of different degrees using the vanishing result $H^i(Z_r, \V_\eta^\sub\otimes\omega_{\Hdg}^m)=0$ for $m\gg 0$ and $i>0$ (cf. \cite{GoldringKoskivirta2019} Lemma 7.1.4). In particular, to study $H^i(X, \V_\eta^\sub)$, one performs a Hasse regular sequence of length $i$ to finally reduce to $H^0$.

\begin{remark}
	\begin{enumerate}
		\item As mentioned above, in the unramified case Theorem \ref{Thm Hecke alg} was proved by Goldring-Koskivirta in \cite{GoldringKoskivirta2019}; in this case part (1) also follows from the work of Boxer \cite{Box}. 
		\item In the Hodge type case (which may be ramified), part (1) was proved by Pilloni-Stroh in \cite{PS16} (Th\'eor\`eme 3.5 and Remarque 3.9) for some quite different integral models (constructed by Scholze's method). As they remarked there,  the torsion classes for these integral models seem to be quite different from those associated to the unramified Kottwitz or Kisin models. Their torsion classes also seem to be rather different from these associated to the smooth splitting models here.
		\item By \cite{GoldringKoskivirta2019} Remark 8.2.4, parts (2) and (3) of Theorem \ref{Thm Hecke alg} do not follow the methods of \cite{Box} and \cite{PS16}.
	\end{enumerate}
\end{remark}

\subsection{Galois representations}\label{subsec Galois}
We can now deduce some consequences on Galois representations from Theorem \ref{Thm Hecke alg} as \cite{GoldringKoskivirta2019} section 10.

Let $v\neq p$ be a finite unramified place of $\Q$ for $G$. Let $\Frob_v$ be a geometric Frobenius at $v$ and $\H_v$ the unramified Hecke algebra at $v$. Then we have the Satake isomorphism (cf. \cite{GoldringKoskivirta2019} (10.2.1))
\[\H_v[\sqrt{v}]\st{\sim}{\ra} R({}^LG_v)[\sqrt{v}],\]
where $R({}^LG_v)$ is the algebra obtained by restricting character representations of ${}^LG_v$ to semisimple ${}^LG_v^\circ(\ov{\Q}_p)$-conjugacy classes in ${}^LG_v\rtimes\Frob_v$.
If $\pi_v$ is an unramified irreducible smooth representation of $G(\Q_v)$, we get the corresponding semisimple ${}^LG_v^\circ(\ov{\Q}_p)$-conjugacy class $\Sat(\pi_v)$, the Satake parameter of $\pi_v$.

Let \[r: {}^LG\ra \GL_m\] be a representation of the Langlands dual group. For each place $v$ as above, we get an induced representation $r_v$ of ${}^LG_v$.   Let $\pi$ be a $C$-algebraic cuspidal automorphic representation of $G$. If $v$ is an unramified place of $\pi$, then we get \[r_v(\Sat(\pi_v))\in \GL_m(\ov{\Q}_p).\] Let $Ram(\pi)$ be the set of ramified places of $\pi$.
We say that $(\pi,r)$ satisfies \emph{$LC_p$}, if there exists a continuous semisimple Galois representation 
\[\rho(\pi,r): \mathrm{Gal}(\ov{\Q}/\Q)\ra \GL_m(\ov{\Q}_p)\] such that for any $v\notin Ram(\pi)\cup\{p\}$, we have \[\rho(\pi,r)(\Frob_v)=r_v(\Sat(\pi_v))\] as $\GL_m(\ov{\Q}_p)$-conjugacy classes.

For any $j\geq 1$, the function ${}^LG(\ov{\Q}_p)\ra \mathbb{\ov{\Q}}_p,\, g\mapsto \tr{tr}(r(g)^j)$ defines an element of $R({}^LG_v)$ and thus an element $T_v^{(j)}(r)\in \H_v[\sqrt{v}]$. For any $i\geq 0, n\geq 1, \eta\in X^\ast(T)_L^+$, let \[T_v^{(j)}=T_v^{(j)}(r;i,n,\eta)\in \H^{i,n}_\eta\] be its image in $\H^{i,n}_\eta$.
In the following we fix $\delta\in\mathbb{R}_{\geq 0}$ and $r: {}^LG\ra \GL_m$. Here is the version of \cite{GoldringKoskivirta2019} Theorems 10.4.1 and 10.5.1 in the ramified setting. The proof is identical to loc. cit. by applying Theorem \ref{Thm Hecke alg} here. Part (1) also generalizes Theorem 1.1 of \cite{ReduzziXiao2017}.
\begin{theorem}\label{thm Galois}
Suppose that for any $\delta$-regular, $C$-algebraic cuspidal automorphic representation $\pi'$ with $\pi_{\infty}'$ discrete series, the pair $(\pi',r)$ satisfies $LC_p$. 
\begin{enumerate}
	\item 
For any $i\geq 0, n\geq 1, \eta\in X^\ast(T)_L^+$,
there exists a continuous Galois pseudo-representation
\[\rho: \mathrm{Gal}(\ov{\Q}/\Q)\lra \H^{i,n}_\eta, \]such that $\rho(\Frob_v^j)=T_v^{(j)}$ for all $v\notin S$.

\item Let $\pi$ be a $C$-algebraic cuspidal automorphic representation of $G$ such that $\pi_\infty$ is a ($C$-algebraic) non-degenerate limit of discrete series and $\pi_p^{K_p}\neq 0$. Then $(\pi,r)$ also satisfies $LC_p$.
\end{enumerate}
\end{theorem}
Recall that $K_p\subset G(\Q_p)$ is a very special parahoric subgroup, and irreducible smooth representations $\pi_p$ of $G(\Q_p)$ such that $\pi^{K_p}_p\neq 0$ can be classified by their spherical parameters, see \cite{Zhu} section 6.

\subsection{Examples}
We discuss some concrete examples where the condition $LC_p$ is essentially known. In these examples, the notation $L$ is also used as certain number fields. Thus
to avoid confusion, we denote $\mathbf{L}\subset P_\mu$ for the Levi subgroup.
\subsubsection{Unitary case}
We use the notations of subsection \ref{subsec ex GU}.

For any regular $C$-algebraic cuspidal automorphic representation $\pi$ of $G$, the associated Galois representation satisfying the condition $LC_p$ is known to exist by the works of many people. We only mention
\cite{HT, Shin, CH, Clo}.
\begin{corollary}
For any $i\geq 0, n\geq 1, \eta\in X^\ast(T)_\mathbf{L}^+$,
there exists a continuous Galois pseudo-representation
\[\rho: \mathrm{Gal}(\ov{L}/L)\lra \H^{i,n}_\eta, \]such that $\rho(\Frob_v^j)=T_v^{(j)}$ for all $v\notin S$.
\end{corollary}
Here, when applying Theorem \ref{thm Galois} we can replace $\Q$ by the totally real field $L$, cf. \cite{GoldringKoskivirta2019}  10.6.
One may also state and prove a similar version of \cite{GoldringKoskivirta2019} Theorem 10.5.3 for Galois representations associated to automorphic representations $\pi$ with non degenerate limit of discrete series $\pi_{\infty}$. 

\subsubsection{Hilbert-Siegel case}
We use the notations of subsection \ref{subsec ex Gsp}.

For the group $G=\Res_{L|\Q}\GSp_{2g}$,  any regular $C$-algebraic cuspidal automorphic representation $\pi$ of $G$, the associated Galois representation satisfying the condition $LC_p$ is known to exist for $g\leq 2$:
\begin{itemize}
	\item for $g=1$, see \cite{ERX, ReduzziXiao2017} and the references therein for the related classical works, 
	\item for $g=2$ this has been intensively studied, see \cite{Tay, Lau, Wei, Sor} for example,
	\item for general $g$, see \cite{KS, Xu} for some recent progress. 
\end{itemize}
\begin{corollary}
	Assume the condition $LC_p$ holds.
	For any $i\geq 0, n\geq 1, \eta\in X^\ast(T)_\mathbf{L}^+$,
	there exists a continuous Galois pseudo-representation
	\[\rho: \mathrm{Gal}(\ov{L}/L)\lra \H^{i,n}_\eta, \]such that $\rho(\Frob_v^j)=T_v^{(j)}$ for all $v\notin S$.
\end{corollary}
Here, as above, when applying Theorem \ref{thm Galois} we can replace $\Q$ by the totally real field $L$.
One may also state and prove a version for Galois representations associated to automorphic representations $\pi$ with non degenerate limit of discrete series $\pi_{\infty}$.

\appendix

\section{Local models and EKOR stratifications in ramified PEL-type case}\label{section loc mod}
In this appendix, we first review the related local model diagrams for integral models of PEL-type Shimura varieties with general parahoric level at $p$, following \cite{PappasRapoport2005}. Then we briefly explain how to extend the construction of \cite{ShenYuZhang2021} to this setting (the groups of \cite{ShenYuZhang2021} are supposed to be tamely ramified at $p$ as those in \cite{KisinPappas2018}, but the construction there only needs local model diagrams as the input).

\subsection{Integral models of PEL-type Shimura varieties}
We keep the notations and assumptions of section \ref{sec:splitting}.
Let $G$ be the connected reductive group defined by the rational PEL datum and $\mathscr{L}$ a multichain of $\calO_B$-lattices (see Definition \ref{def multichain}).
By \cite{PappasRapoport2005}, there are three integral models $\mathscr{A}^{\naive}, \mathscr{A}$ and $\mathscr{A}^{\spl}$ of 
the PEL moduli space over $E$ (with respect to the multichain $\mathscr{L}$). 
The models $\mathscr{A}^{\naive}, \mathscr{A}$ are defined over $\calO_E$ and $\mathscr{A}^{\spl}$ is defined over $\calO_F$.
We first recall the definition of PEL datum with parahoric level structure following \cite{RapoportZink1996}, 
see also the appendix of \cite{ShenYuZhang2021} or \cite{Hartw} section 2.

\subsubsection{Parahoric data at $p$}

To simplify the notation, 
we will write $(B,*, V, \psi=\Pair{\cdot}{\cdot}, \calO_B, \Lambda)$ for the base change of such data in section~\ref{sec:splitting} to $\mathbb{Q}_p$.
So we have
\[
	B  \simeq \prod_{i=1}^r \Mat_{m_i}(R_i), \quad 
	\calO_B  \simeq  \prod_{i=1}^r \Mat_{m_i}(\calO_{R_i}).
\]
By Morita equivalence, we can decompose the $B$-module $V$ (resp. any $\calO_B$-lattice $\Lambda$ in $V$) as
\[V =\bigoplus_{i=1}^r V_i^{m_i} \quad (\tr{ resp. } \Lambda \simeq \bigoplus_{i=1}^r \Lambda_i^{m_i}),\]
where each factor $V_i$ is a free $R_i$-module (resp. $\Lambda_i$ is an $\mathcal{O}_{B_i}$-lattice in $V_i$).
Write $2d_i = \rank_{F_i} V_i$.

\begin{definition}\label{def multichain}
	\begin{enumerate}
		\item A chain of $\calO_B$-lattices in $V$ is a set of totally ordered $\calO_B$-lattices $\mathscr{L}$ such that
			for every element $x \in B^\times$ which normalizes $\calO_B$, one has
			\[\Lambda \in \mathscr{L} \implies x \Lambda \in \mathscr{L}.  \]
		\item A set $\mathscr{L}$ of $\calO_B$-lattices in $V$ is said to be a \emph{multichain of $\calO_B$-lattices} if there
			exists a chain of $\mathcal{O}_{B_i}$-lattices $\mathscr{L}_i$ in $V_i$ for each $i=1,\dots, m$ such that for any member
			$\Lambda \in \mathscr{L}$ one has $\Lambda_i \in \mathscr{L}_i$ for all $i=1,\dots,m$.
		\item A multichain of $\calO_B$-lattices $\mathscr{L}$ is called \emph{self-dual} if for every member $\Lambda \in \mathscr{L}$,
			its dual lattice $\Lambda^\vee$	also belongs to $\mathscr{L}$, where
			\[\Lambda^\vee := \{x \in V \mid \psi(x,\Lambda) \subset \mathbb{Z}_p\}.\]
	\end{enumerate}
\end{definition}

For a multichain $\mathscr{L}$ of $\calO_B$-lattices, 
we write \[\calG = \calG_\mathscr{L}\] for the group scheme over $\mathbb{Z}_p$ defined by the stabilizer of the multichain $\mathscr{L}$ as in \cite{RapoportZink1996}*{\S 6}.
By \cite{HainesRicharz2020}*{Corollary 4.8},
for a reductive group of the form $G = \Res_{F|\bbQ_p} G'$,
every parahoric group scheme of $G$ is of the form $\Res_{\calO_F|\bbZ_p} \calG'$ for a unique parahoric group scheme $\calG'$ of $G'$.
If we assume that $G'$ unramified, 
the stabilizer group scheme $\calG = \calG_\mathscr{L}$  associated to $\mathscr{L}$ is always connected. 
So $\calG$ is a parahoric group scheme of $G$ and (since the prime to $p$ level $K^p$ is fixed) we write $K=\G(\Z_p)$.

\subsubsection{Integral models}

Now let $\AV$ be the category of abelian varieties with $\calO_B$-actions, where morphisms are prime to $p$ isogenies.

\begin{definition}
	Let $\mathscr{L}$ be a multichain of $\calO_B$-lattices in $V$. 
	A \emph{$\mathscr{L}$-set of abelian varieties} over a $\mathbb{Z}_{p}$-scheme $S$ is a functor
	\[\mathscr{L} \to \AV,\quad \Lambda \mapsto A_\Lambda,\]
	satisfying the following.
	\begin{enumerate}
		\item For each inclusion $\Lambda \subset \Lambda'$ in $\mathscr{L}$, a quasi-isogeny $A_\Lambda \to A_{\Lambda'}$.
		\item For any element $a \in B^\times \cap \calO_B$ which normalizes $\calO_B$
			and any member $\Lambda \in \mathscr{L}$, there exists an isomorphism $\theta_{a,\Lambda}: A_\Lambda^a \to A_{a\Lambda}$
			such that the following diagram commutes
			\[\xymatrix{
				{A_{\Lambda}^a} \ar[r]^{\theta_{a,\Lambda}} \ar[rd] & {A_{a \Lambda}} \ar[d]^{\rho_{a\Lambda,\Lambda}} \\
				& {A_{\Lambda}}
			}\]
	\end{enumerate}
\end{definition}

Let $\mathscr{L}$ be a self-dual multichain of $\calO_B$-lattices in $V$,
and fix a sufficiently small open compact subgroup $K^p \subset G(\mathbb{A}_f^p)$. Recall that $E$ is the local reflex field. 
There is a naive integral model $\mathscr{A}^{\naive}_\mathscr{L}$ over $\calO_E$,
which is a moduli scheme classifying the objects 
$(A_\mathscr{L}, \overline{\lambda},\overline{\eta})/S$ for each scheme $S/\calO_E$, where
\begin{itemize}
	\item $A_\mathscr{L} = (A_\Lambda)_{\Lambda\in\mathscr{L}}$ is a $\mathscr{L}$-set of abelian schemes over $S$ in $\AV$;
	\item $\overline{\lambda} = \mathbb{Q}^\times \cdot \lambda$ is a $\mathbb{Q}$-homogeneous principal polarization on $A_\mathscr{L}$;
	\item $\overline{\eta}$ is a $\pi_1(S,\overline{s})$-invariant $K^p$-orbit of isomorphism
		$\eta: V \otimes \mathbb{A}_f^p \to T^p(A_{\overline{s}})$
		which preserves the pairings up to a scalar in $(\mathbb{A}_f^p)^\times$. 
		Here $T^p(A_{\overline{s}})$ is the prime to $p$ Tate module of $A_{\overline{s}}$ and for simplicity we assume that $S$ is connected.
\end{itemize}
For the exact meaning of the above terms, we refer to \cites{PappasRapoport2005,RapoportZink1996,Lan18}
and appendix of \cite{ShenYuZhang2021}.

In general the naive model $\mathscr{A}^{\naive}_\mathscr{L}$ is not flat over $\calO_E$, cf. \cite{Pappas1}.
In order to define a good integral model, 
Pappas and Rapoport introduced an alternative integral model of 
$\mathscr{A}^{\naive}_\mathscr{L} \otimes_{\calO_E} F$ in \cite{PappasRapoport2005},
where $F$ is a large enough extension of $E$ (which contains the Galois closure of $E$ over $\mathbb{Q}_p$).
We briefly recall their construction. To this end, we find it convenient to use Lan's formulation of splitting structures (\cite{Lan18}).

\begin{definition}[\cite{Lan18}*{Definition 2.1.12}]\label{def triple modules}
	Suppose that $S$ is a scheme over $\calO_F$.
	A $\mathscr{L}$-set of polarized $\calO_B \otimes \calO_S$-modules is a triple 
	$(\underline{\calH}, \underline{\mathscr{F}},\underline{\iota})$, where:
	\begin{enumerate}
		\item $\underline{\calH}: \Lambda \mapsto \calH_\Lambda$ and $\underline{\mathscr{F}}: \Lambda \mapsto \mathscr{F}_\Lambda$
			are functors from the category $\mathscr{L}$ to the category of $\calO_B\otimes \calO_S$-modules.
		\item For each $\Lambda\in\mathscr{L}$,  both $\mathscr{F}_{\Lambda}$ and 
			$\calH_{\Lambda}/\mathscr{F}_{\Lambda}$ are finite locally free $\calO_S$-modules,
			such that $\calH_{\Lambda}/\mathscr{F}_{\Lambda}$ satisfies the determinant condition.
		\item For other conditions, we refer to Lan's paper \cite{Lan18}.
	\end{enumerate}
\end{definition}

\begin{definition}[\cite{Lan18}*{Definition 2.3.3}; see also \cite{PappasRapoport2005} Definition 14.1]\label{def gen splitting str}
	Suppose that $S$ is a scheme over $\calO_F$, and that $(\underline{\mathcal{H}},\underline{\mathscr{F}},\underline{\iota})$ is
	a $\mathscr{L}$-set of polarized $\calO_B \otimes \mathcal{O}_S$-modules. 
	A \emph{splitting structure} for
	$(\underline{\mathcal{H}},\underline{\mathscr{F}},\underline{\iota})$ is a collection
	\[\ul{\scrF_\bullet}=(\mathscr{F}_{i,j}^l, \iota_{i,j}^l)_{1\leq i\leq r, 1\leq j \leq f_i, 0 \leq l \leq e_i},\]
	where each $\mathscr{F}_{i,j}^l: \Lambda \to \mathscr{F}_{\Lambda,i,j}^l$ is a functor from the category $\mathscr{L}$
	to the category of $\calO_B \otimes \mathcal{O}_S$-modules, and each $\iota_{i,j}^l: \mathscr{F}_{\Lambda, i,j}^l \to \calH_{\Lambda,i,j}$
	is an injective morphism satisfying the following conditions (identify $\mathscr{F}_{i,j}^l$ with its image under $\iota$):
	\begin{enumerate}
		\item For each $\Lambda \in \mathscr{L}$, we require both $\mathscr{F}_{\Lambda,i,j}^l$ and 
			$\calH_{\Lambda,i,j} / \mathscr{F}_{\Lambda,i,j}^l$ to be finite locally free $\mathcal{O}_S$-modules.
		\item For each $\Lambda \in \mathscr{L}$ and $1 \leq i \leq r, 1\leq j \leq f_i$, we have a filtration
			\[0 = \mathscr{F}_{\Lambda,i,j}^0 \subset \mathscr{F}_{\Lambda,i,j}^1 \subset \cdots \subset \mathscr{F}_{\Lambda,i,j}^{e_i} = \mathscr{F}_{\Lambda,i,j}\]
			as $\calO_B \otimes \mathcal{O}_S$-submodule of $\calH_{\Lambda,i,j}$. 
			For each integer $0 < l \leq e_i$, the quotient $\mathscr{F}_{\Lambda,i,j}^l / \mathscr{F}_{\Lambda,i,j}^{l-1}$ is a
			locally free $\mathcal{O}_S$-module of rank $d_{i,j}^l$, annihilated by $b \otimes 1 - 1 \otimes \sigma_{i,j}^l(b)$ for all $b \in \calO_{F_i}$.
		\item For each $\Lambda \in \mathscr{L}$ and $i \in \mathscr{I}$, there are \emph{periodicity isomorphisms}, cf. \cite{Lan18} p. 2475 for more details.
		\item For each $\Lambda\in\mathscr{L}$ and tuples $(i,j,l)$, let $(\mathscr{F}_{\Lambda,i,j}^l)^\perp$
			denote the orthogonal complement of $\mathscr{F}_{\Lambda,i,j}^l$ in $\calH_{\Lambda^\vee,i,j}$
			with respect to the perfect pairing 
			$\calH_{\Lambda,i,j} \times \calH_{\Lambda^\vee,i,j} \to \calO_S$. Then
			\[\prod_{0\leq k < l} (b \otimes 1 - 1 \otimes \sigma_{i,j}^{k}(b)) ((\mathscr{F}_{\Lambda,i,j}^{l})^\perp) \subset \mathscr{F}_{\Lambda^\vee,i,j}^l.\]
	\end{enumerate}
\end{definition}

Let $\mathscr{A}^{\spl}_\mathscr{L}$ be the moduli scheme over $\calO_F$ classifying the objects 
$(A_\mathscr{L},\overline{\lambda},\overline{\eta}, \underline{\mathscr{F}_\bullet})$ for each scheme $S/\calO_F$, 
where $(A_\mathscr{L},\overline{\lambda},\overline{\eta})$ is an object in $\mathscr{A}^{\naive}_\mathscr{L}(S)$ and
$\underline{\mathscr{F}_\bullet}$ is a splitting structure for the $\mathscr{L}$-set of polarized $\calO_B \otimes \calO_S$-modules
associated with $(A_\mathscr{L},\overline{\lambda},\overline{\eta})$. 
This is the \emph{splitting model} associated to the PEL datum. With the notation of \cite{Lan18}, we have
\[ \mathscr{A}^{\spl}_\mathscr{L}=\mathrm{Spl}^+_{(\ul{\H},\ul{\scrF},\ul{\iota})/\mathscr{A}^{\naive}_\mathscr{L}\otimes \calO_F}. \]

In the following, we will simply write $\scrA^\spl$ and $\scrA^{\naive}$ for the schemes $\mathscr{A}^{\spl}_\mathscr{L}$  and $\mathscr{A}^{\naive}_\mathscr{L}$  respectively.
Let $\mathscr{A}$ be the scheme-theoretic image of the natural morphism 
\[\mathscr{A}^{\spl} \to \scrA^{\naive} \otimes_{\calO_E} \calO_F \to \mathscr{A}^{\naive}.\]
The natural morphism \[\scrA^\spl \ra \scrA\] is projective by \cite{Lan18} Proposition 2.3.7.
Both $\scrA^\spl$ and $\mathscr{A}$ admit better geometric properties than the naive integral model $\scrA^{\naive}$.

\subsection{Local models}
The local structure of integral models of Shimura varieties is controlled by the associated local models.
There is a naive local model $\mathbb{M}^{\naive}$ associated to the moduli scheme $\mathscr{A}^{\naive}$.
Recall that $\bbM^{\naive} := \bbM^{\naive}(\G,\mu)$ is the moduli scheme over $\calO_E$ given by the following definition.
\begin{definition}[\cite{RapoportZink1996}]
	A point of $\bbM^{\naive}$ with values in an $\calO_E$-scheme $S$ is given by the following data.
	\begin{enumerate}
		\item A functor from the category $\scrL$ to the category of $\calO_B \otimes \calO_S$-modules on $S$:
			\[\Lambda \to t_\Lambda, \quad\Lambda \in \scrL.\]
		\item A morphism of functors
			\[\varphi_\Lambda: \Lambda \otimes_{\bbZ_p} \calO_S \to t_\Lambda.\]
	\end{enumerate}
	We require the following conditions to be satisfied:
	\begin{enumerate}
		\item $t_\Lambda$ is a finite locally free $\calO_S$-module.
			The $\calO_B$-action on $t_\Lambda$ satisfies the determinant condition
			\[\det_{\calO_S} (a; t_\Lambda) = \det_F (a; W), \quad a \in \calO_B.\]
		\item The morphisms $\varphi_\Lambda$ are surjective.
		\item The composition of the following maps is zero for each $\Lambda$:
			\[t_\Lambda^* \to (\Lambda \otimes \calO_S)^* \cong \hat{\Lambda} \otimes \calO_S \to t_{\hat{\Lambda}}.\]
	\end{enumerate}
\end{definition}
The naive local model $\mathbb{M}^{\naive}$ is usually not flat in the ramified case, cf. \cite{Pappas1}.
We define the \emph{splitting local model} \[\mathbb{M}^{\spl}=\mathbb{M}^{\spl}(\G,\mu)\] as the moduli scheme over $\calO_F$
classifying the splitting structures over $\mathbb{M}^{\naive} \otimes_{\calO_E} \calO_F$.
The splitting model $\mathbb{M}^{\spl}$ admits good properties. Recall we assume that there is no type (AR) local factors.

\begin{proposition}\label{prop:flatness}
	The splitting local model $\mathbb{M}^{\spl}$	is flat over $\calO_F$.
\end{proposition}
\begin{proof}
	This is a direct generalization of \cite{PappasRapoport2005}*{Theorems 5.3 and 9.4}. For the reader's convenience, we briefly recall their proofs.
Without loss of generality, we may assume $r=1$ and we slightly change the notation: let $G=G_1$ be the associated reductive group over $\Q_p$.
So we have
\[G = \Res_{{F_1}|\bbQ_p} G'\]
for an unramified group $G'$ over a local field ${F_1}|\bbQ_p$.
Since $F|\bbQ_p$ contains the Galois closure of ${F_1}$, we have
\[G_F = \prod_{\tau: {F_1} \to F} G'_\tau, \text{ where } G'_\tau := G' \otimes_{{F_1},\tau} F.\]
The projection of
\[\mu: \bbG_{m,F} \to G_F = \prod_{\tau: {F_1} \to F} G'_\tau\]
to each factor $G'_\tau$ gives a cocharacter $\mu_\tau: \bbG_{m,F} \to G'_\tau$.
Moreover, for each $\tau: {F_1}\to F$, one can associate an $\calO_F$-multichain by
\[\scrL_\tau := \{\Lambda_\tau = \Lambda \otimes_{\calO_{F_1}} \calO_F \mid \Lambda \in \scrL\}.\]
Such a multichain determines a parahoric subgroup of $G'_\tau$.
This gives a local model (over $\calO_F$) 
\[\bbM^{\loc}_\tau := \bbM^{\loc}(G'_\tau, \mu_\tau)_{\scrL_\tau}.\]
By our assumption, $G'$ is unramified over $F$,
so $\bbM^{\loc}$ is the same as the naive local model.
This means that for each $\calO_F$-scheme $S$,
$\bbM^{\loc}(S)$ classifies the set of multichains $\{\scrF_{\Lambda,\tau}\}_{\Lambda \in \scrL} \subset \scrL_\tau$ which is compatible with transition maps,
and each $\scrF_{\Lambda,\tau}$ is Zariski locally on $S$ an $\calO_S$-direct summand of $\Lambda_\tau$ of rank $d_{\tau}$.

Given the splitting local model $\bbM^{\spl}=\bbM^{\spl}(\G,\mu)$ over $\calO_F$,
one has the following  diagram (which is a modified version of the diagrams (5.10) and (9.13)  in \cite{PappasRapoport2005})
\[\xymatrix{
	& {\widetilde{\bbM}^{\spl}} \ar[ld]_{\pi_1} \ar[rd]^{\pi_2}\\
	{\bbM^{\spl}} && \prod_{\tau:F_1 \to F}{\bbM^{\loc}_\tau}.
}\]\label{spl to loc}
For each $\calO_F$-scheme $S$,
$\widetilde{\bbM}^{\spl}(S)$ classifies 
\[(\scrF_{\Lambda,j}^l, {\varphi_{\Lambda,j}^l: \Upsilon_{\Lambda,j}^{l} \simeq \Lambda_{j}^l}\otimes_{\calO_F} \calO_S)_{\Lambda \in \scrL},\]
where $(\scrF_{\Lambda,j}^l)_{\Lambda\in \scrL,j}^l \in \bbM^{\spl}(S)$,
$\varphi_j^l$ is an isomorphism from the $\calO_S$-module
\[\Upsilon_{\Lambda,j}^l := \Ker((\pi \otimes 1 - 1 \otimes \sigma_j^l(\pi)) |_{ \Lambda_{j,S}^l / \scrF_{\Lambda,j}^{l-1}})\]
to the $\calO_S$-module $\Lambda_{j}^l\otimes_{\calO_F} \calO_S$ with
\[\Lambda_j^l := \Lambda \otimes_{\calO_F,\sigma_j^l} \calO_F.\]
Such a trivialization exists Zariski locally on $S$ by \cite{PappasRapoport2005}*{Propositions 5.2, 9.2} (this reference only covers the case of type (C) and (AL); the case of type (AU) can be proved by the same argument).
The map $\pi_1$ is the natural forgetful morphism,
and $\pi_2$ sends $(\scrF_{\Lambda,j}^l, {\varphi_{\Lambda,j}^l})$ to $\varphi_{\Lambda,j}^l(\scrF_{\Lambda,j}^{l} / \scrF_{\Lambda,j}^{l-1}) \in \bbM^{\loc}_{\sigma_{j}^l}$.

Let $\calG_j^l$ be the subgroup of $\prod_{\Lambda \in \scrL} \Aut(\Lambda_{j}^l)$ compatible with the transition maps of $\scrL$,
and $\calG^{\spl} := \prod_{j,l}\calG_{j}^l$.
Then the action
\[g_j^l \cdot (\scrF_{\Lambda,j}^l, {\varphi_{\Lambda,j}^l}) = (\scrF_{\Lambda,j}^l, {g_j^l \varphi_{\Lambda,j}^l})\]
makes $\pi_1$ a $\prod_{j,l\geq 2}\calG_{j}^l$-torsor.
The other action
\[g_j^l \cdot (\scrF_{\Lambda,j}^l, {\varphi_{\Lambda,j}^l}) = ((\varphi_{\Lambda,j}^l)^{-1} g_j^l \varphi_{\Lambda,j}^l(\scrF_{\Lambda,j}^l), {g_j^l \varphi_{\Lambda,j}^l})\]
makes $\pi_2$ a $\prod_{j,l\geq 2}\calG_{j}^l$-torsor. 
Now the existence of such a diagram of torsors for a smooth group scheme and flatness of unramified local models implies the flatness of $\bbM^{\spl}$.
\end{proof}

\begin{proposition}\label{Prop loc mod diag}
	\begin{enumerate}
		\item Let $\mathbb{M}^\loc=\mathbb{M}^\loc(\G,\mu)$ be the scheme-theoretic image of the natural forgetful morphism 
			$\mathbb{M}^{\spl} \to \mathbb{M}^{\naive}$.
			Then $\mathbb{M}^\loc$ coincides with the local model $\bbM$ defined in \cite{Levin2016} with respect to the triple 
			$(G,\mu,K_\mathscr{L})$.
		\item We have the following naive local model diagram:
			
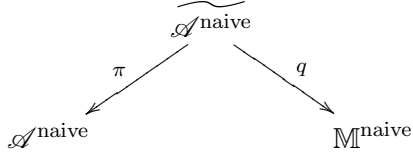
\[\xymatrix{
			& \widetilde{\mathscr{A}^{\naive}} \ar[ld]_{\pi} \ar[rd]^{q} \\
			{\mathscr{A}^{\naive}} && {\bbM^{\naive}}
			}\]
		\item 
			We have a local model diagram:
			
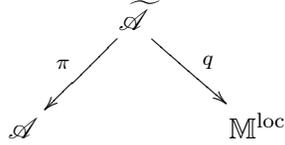
\[\xymatrix{
			& {\widetilde{\mathscr{A}}} \ar[ld]_{\pi} \ar[rd]^{q} \\
			{\mathscr{A}} && {\bbM^\loc}
			}\]
		\item The pullback of the natural morphism $\bbM^{\spl} \to \bbM^{\naive}$ gives the splitting local model diagram with respect to the group $\calG$ over $\calO_F$ (\cite{PappasRapoport2005}*{\S 15})
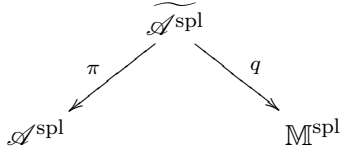
			\[\xymatrix{
				& {\widetilde{\mathscr{A}^{\spl}}} \ar[ld]_{\pi} \ar[rd]^{q} \\
				{\mathscr{A}}^{\spl} && {\bbM^{\spl}}
			}\]
	\end{enumerate}
\end{proposition}
\begin{proof}
	We only need to show that $\bbM$ coincides with $\bbM^{\rm loc}$, as the other statements are contained in \cite{PappasRapoport2005}.
	The image $\bbM^{\rm loc}$ is flat by the flatness of $\bbM^{\spl}$.
	As the generic fiber of $\bbM$ agrees with $\bbM^{\naive}$, 
	$\bbM$ is the flat closure of the generic fiber of $\bbM^{\naive}$,
	which agrees with $\bbM^{\rm \loc}$ by definition.
\end{proof}

In order to study the geometry of $\scrA_0^{\spl}$,
we would like to show that there is a local model diagram of splitting models with respect to the splitting group $\calG^{\spl}$ in the next paragraph.
This means that there is a $\calG^{\spl}$-torsor $\widetilde{\scrA^{\spl}}$ over $\scrA^{\spl}$,
and a $\calG^{\spl}$-equivariant morphism $\widetilde{\scrA^{\spl}} \to \bbM^{\loc}(\calG^{\spl}, \mu)$.

For every scheme $S$ over $\calO_F$ and each $\Lambda \in \scrL$,
we define the $\calO_S$-module $\Upsilon_{\Lambda,i,j}^l$ as
\[\Upsilon_{\Lambda,i,j}^l := \Ker( (\pi_i \otimes 1 - 1 \otimes \sigma_{i,j}^l(\pi_i)) |_{\H_{i,j}/\scrF_{i,j}^l}).\]
Consider also the $\calO_F$-lattice $\Lambda_{i,j}^l := \Lambda_i \otimes_{\calO_{F_i}, \sigma_{i,j}^l} \calO_F$.
Then
Zariski locally there exists an isomorphism
\[\Upsilon_{\Lambda,i,j}^l  \simeq \Lambda_{i,j}^l\otimes_{\calO_F}\calO_S,\]
see the proof of Proposition \ref{prop:flatness}.
For every $\Lambda \in \scrL$,
we can define an $\calO_S$-module $\calM_\Lambda$ as
\[\calM_\Lambda :=\bigoplus_i\calM_{\Lambda,i}^{m_i},\quad \calM_{\Lambda ,i}:= \bigoplus_{j,l} \Upsilon_{\Lambda,i,j}^l,\]
so it is locally isomorphic to the $\calO_F$-lattice
\[\Lambda^{\spl} :=\bigoplus_i \Lambda^{\spl, m_i}_i,\quad \Lambda^{\spl}_i:= \bigoplus_{j,l} \Lambda_{i,j}^l.\]

Consider the group
\[\calG^{\spl}_\Lambda := \Aut(\Lambda^{\spl}) \subset \prod_{i,j,l} \calG_{\Lambda,i,j}^l.\]
We simply define the splitting group $\calG^{\spl}$ as
\[\calG^{\spl} =\bigcap_{\Lambda \in \scrL} \calG_{\Lambda}^{\spl}.\]
For each $i,j,l$, let $\calG_{i,j}^l$ be the group scheme over $\calO_F$ defined by the automorphism of the multichain $(\Lambda_{i,j}^l)_{\Lambda \in \scrL}$.
Then it admits a decomposition as $\calG^{\spl}=\prod_{i,j,l}\G^l_{i,j}$.
Now we define \[\bbM^{\loc}(\calG^{\spl},\mu):=\prod_{i,j,l}\bbM^\loc(\G_{i,j}^l,\mu_{i,j}^l),\] where on the right hand side each $\bbM^\loc(\G_{i,j}^l,\mu_{i,j}^l)$ is the local model attached to the pair $(\G_{i,j}^l,\mu_{i,j}^l)$.
Note that if we consider the natural $\calG^{\spl}$-action on $\bbM^{\loc}(\calG^{\spl},\mu)$,  as in the proof of Proposition~\ref{prop:flatness} we get 
the following diagram of schemes over $\calO_F$ for splitting local models (modified version of the diagrams (5.10) and (9.13) of \cite{PappasRapoport2005}):
\[\xymatrix{
	& {\widetilde{\bbM}^{\spl}(\calG,\mu)} \ar[ld]_{\pi_1} \ar[rd]^{q} \\
	\bbM^{\spl}(\calG,\mu) && \bbM^{\loc}(\calG^{\spl}, \mu)
}\]
where $\pi_1$ is the $\prod_{i,j,l\geq 2}\G^l_{i,j}$-torsor, 
and $q$ is a $\calG^{\spl}$-equivariant morphism.
Note that $q = \pi_2$,
but here we only consider the group action on $\widetilde{\bbM}^{\spl}(\calG,\mu)$ given by $\pi_1$.
One can see that such action is compatible with the natural $\prod_{i,j,l\geq 2}\G^l_{i,j}$-action on $\bbM^{\loc}(\calG^{\spl}, \mu)$.
This motivates the following local model diagram for splitting integral models.

\begin{proposition}\label{prop loc mod Gspl}
	We have the following local model diagram for splitting models:
	\[\xymatrix{
	& {\widetilde{\mathscr{A}^{\spl}}} \ar[ld]_{\pi} \ar[rd]^{q} \\
	{\mathscr{A}^{\spl}} && \bbM^{\loc}(\calG^{\spl},\mu) \subset \prod_{i, j, l} \bbM_{\Lambda,i,j}^l, & {}
	}\]
	where for every $\calO_F$-scheme $S$,
	${\widetilde{\mathscr{A}^{\spl}}}(S)$ classifies isomorphism classes of
	\[(\underline{A}_{\scrL}, \underline{\scrF}_{\scrL}, (\tau_{\Lambda, i, j}^l:\Upsilon_{\Lambda,i,j}^l \simeq \Lambda_{i, j}^l\otimes_{\calO_F} \calO_S)_{\Lambda \in \scrL} ),\]
	with $(\underline{A}_{\scrL}, \underline{\scrF}_{\scrL}) \in \scrA^{\spl}(S)$, and
	 $\tau_{\scrL} := \{\tau_{\Lambda,i,j}^l\}$ can be viewed as a trivialization of the multi-chain $(\calM_\Lambda)_{\Lambda \in \scrL}$.
	The morphism $\pi$ is the natural forgetful morphism, 
	which is a $\calG^{\spl}$-torsor.
	The morphism $q$ is the natural $\calG^{\spl}$-equivariant smooth morphism given by
	\[(\underline{A}_{\scrL}, \underline{\scrF}_{\scrL}, \{\tau_{\Lambda, i, j}^l\}) \mapsto \{\tau_{\Lambda,i,j}^l(\scrF_{\Lambda,i,j}^l / \scrF_{\Lambda,i,j}^{l-1})\}_{\Lambda \in \scrL}.\]
\end{proposition}
\begin{proof}
		Combining the diagram in Proposition~\ref{Prop loc mod diag}(4) and the diagram in the proof of Proposition~\ref{prop:flatness}, one has the following diagram
	\[\xymatrix{
		& {{\mathscr{A}^{\spl, \#}}} \ar[ld]_{\pi'} \ar[rd]^{q'} 	&& {\widetilde{\bbM}^{\spl}(\calG,\mu)} \ar[ld]_{\pi_1} \ar[rd]^{q''}	\\
		{\mathscr{A}}^{\spl} && {\bbM^{\spl}(\G,\mu)} && \bbM^{\loc}(\calG^{\spl}, \mu).
	}\]
	Let $\widetilde{\widetilde{\scrA^{\spl}}}$ be the product of $q'$ and $\pi_1$, then we have a new diagram
	\[\xymatrix{
		& \widetilde{\widetilde{\scrA^{\spl}}} \ar[ld]_{\pi''} \ar[rd]^{q'''} \\
		\scrA^{\spl} && \bbM^{\loc}(\calG^{\spl}, \mu),
	}\]
	where $\pi''$ is a $\calG_{\calO_F} \times \prod_{i,j, l \geq 2} \calG_{i,j}^l$-torsor,
	and we can check that the same action of $\pi''$ makes $q'''$ a $\calG_{\calO_F} \times \prod_{i,j, l \geq 2} \calG_{i,j}^l$-equivariant smooth morphism. 
	Note that we always have $\Lambda_i^1 \subset \Lambda$, so the restriction of the $\calG$-action on $\Lambda$ induces a natural morphism
	\[ \calG_{\calO_F} = \Aut(\Lambda \otimes \calO_F) \to \Aut(\bigoplus_{i}\Lambda_i^1) = \calG^1. \]
	So we can push the $\calG$-torsor $\pi''$ along the natural morphism $\calG_{\calO_F}  \to \calG^1$ to get $\widetilde{\mathscr{A}^{\spl}}$, a $\G^\spl=\prod_{i,j,l}\G_{i,j}^l$-torsor. Moreover, the $\calG_{\calO_F} \times \prod_{i,j, l \geq 2} \calG_{i,j}^l$-action on $\bbM^{\loc}(\calG^{\spl}, \mu)$ factors through $\G^\spl$, therefore the morphism $q'''$  factors through $\widetilde{\mathscr{A}^{\spl}}$ and induces a $\G^\spl$-equivariant morphism $q: \widetilde{\mathscr{A}^{\spl}}\ra \bbM^{\loc}(\calG^{\spl}, \mu)$.
	This gives the diagram 
		\[\xymatrix{
		& {\widetilde{\mathscr{A}^{\spl}}} \ar[ld]_{\pi} \ar[rd]^{q} \\
		{\mathscr{A}^{\spl}} && \bbM^{\loc}(\calG^{\spl},\mu)  & {}
	}\]
	claimed in the proposition.
\end{proof}

\subsection{EKOR stratification}\label{subsec EKOR loc}

Recall that we write $\kappa$ for the residue field of $\calO_E$,
and $k = \overline{\kappa}$ for the algebraic closure of $\kappa$.
We will construct the EKOR stratification of \[\mathscr{A}_0 := \mathscr{A} \otimes k\] from its local model diagram,
following the idea of \cite{ShenYuZhang2021}.

Fix the triple $(G,\mu, K)$ with associated parahoric group scheme $\calG $.
We have the attached local model $\mathbb{M}^\loc=\mathbb{M}^{\loc}(\G,\mu)$ equipped with a left action of $\calG$.
Let $\calG_0 = \calG \otimes k$ and $M^\loc = \mathbb{M}^\loc \otimes k$. For $K=\G(\Z_p)$, let
$\Adm(\mu)_K$ be the $\mu$-admissible set as in \cite{ShenYuZhang2021} 1.2. By \cite{PappasRapoport2005, Levin2016, HainesRicharz2020}, we have
\begin{corollary}\label{KR}
	There is a set-theoretically disjoint union of locally closed subsets
	\[M^\loc = \coprod_{w \in \Adm(\mu)_K} M^w.\]
Moreover, we have
	\begin{enumerate}
		\item The closure $\ov{M^{w}} = \coprod_{v \leq w} M^v$;
		\item Each $M^w$ consists of a single $\calG_0$-orbit,
			and the stabilizer of each closed point is smooth.
	\end{enumerate}
\end{corollary}

The decomposition of $M^\loc$ induces the KR (Kottwitz-Rapoport) stratification \[\scrA_0=\coprod_{w\in \Adm(\mu)_K}\scrA_0^w,\] where
for each $w \in \Adm(\mu)_K$,
$\scrA_0^w$ is the fiber of the morphism of algebraic stacks over $k$ induced by the local model diagram (cf. Proposition \ref{Prop loc mod diag} (3))
\[\scrA_0 \to [\calG_0 \backslash M^\loc].\]
Each $\scrA_0^w$ is a locally closed smooth subvariety of $\scrA_0$, and we have
\[\overline{\scrA_0^w} = \coprod_{w' \leq w} \scrA_0^{w'}.\]

Consider the local model diagram
\[\xymatrix{
	& {\widetilde{\scrA_0}} \ar[ld]_{\pi} \ar[rd]^{q} \\
	{\scrA_0} && M^\loc
}\]
where $\pi$ is a $\calG_0$-torsor and $q$ is $\calG_0$-equivariant.
For each $w \in \Adm(\mu)_K$,
let $J_w$ be the set defined in 1.3.6 of \cite{ShenYuZhang2021}, and $\calG_0^{\rdt}$ the reductive quotient of $\calG_0$.
By the same method of \cite{ShenYuZhang2021}*{\S 3},
there is a $\calG_0^{\rdt}$-zip of type $J_w$ over $\scrA_0^w$, written as \[(\bbI^w, \bbI^w_+, \bbI^w_-, \iota).\]
The tuple $(\bbI^w, \bbI^w_+, \bbI^w_-, \iota)$ then induces a morphism of stacks
\[\zeta_w: \scrA_0^w \to \Zip{\calG_0^{\rdt}}{J_w}.\]
The proof of \cite{ShenYuZhang2021}*{Theorem 3.4.11} gives
\begin{corollary}
	The morphism $\zeta_w$ is smooth.
\end{corollary}

Let $W=W_{\calG_0^{\rdt}}$ be the Weyl group of $\calG_0^{\rdt}$. Then
the underlying topological space of $\Zip{\calG_0^{\rdt}}{J_w}$ is given by the partially ordered set ${}^{J_w} W$.
For each $x \in {}^{J_w} W$, define $\scrA_0^x := \zeta_w^{-1}(x)$. 
Then $\scrA_0^x$ is a locally closed subvariety of $\scrA_0^w$. Letting $w\in  \Adm(\mu)_K$ vary, we get
\[\scrA_0 = \coprod_{w \in \Adm(\mu)_K, x \in {}^{J_w} W} \scrA_0^x.\]
We will call such decomposition the EKOR stratification of $\scrA_0$.

The index set of EKOR strata is in fact given by the partially ordered set $({}^K \Adm(\mu), \leq_{K,\sigma})$ of \cite{ShenYuZhang2021} 1.2.
There is a natural surjection
\[\pi: {}^K\Adm(\mu) \to \Adm(\mu)_K,\]
such that for every $w \in \Adm(\mu)_K$,
the fiber $\pi^{-1}(w)$ is bijective to the partially ordered set ${}^{J_w} W$ by \cite{ShenYuZhang2021} 1.3.6.

\begin{example}[EKOR strata of Hilbert modular varieties]\label{example EKOR}

Consider $B = L$ a totally real field over $\Q$ and $V = L^2$, then $G$ is a subgroup of $ \Res_{L|\Q} \GL_2$. In this case, the canonical model $\scrA$ (with maximal parahoric level at $p$) was constructed in \cite{DP} and coincides with $\scrA^{\naive}$. 
Let $x = (A, \lambda, \iota, \alpha) \in \scrA_0(k)$ and $\calH$, $\omega$ the corresponding $k$-vector spaces as above. 
For each $i,j$, $\calH_{i,j} \simeq k[\varepsilon_i]^2$ and $\omega_{i,j} \subset \calH_{i,j}$ is a $k[\varepsilon_i]$-submodule with $k$-dimension $e_i$, so there is an integer $a_{i,j}$ such that
\[\omega_{i,j} \simeq (\varepsilon_i^{a_{i,j}}) \oplus (\varepsilon_i^{e_i - a_{i,j}}), \text{   for a unique integer } 0 \leq a_{i,j} \leq \lfloor\frac{e_i}{2}\rfloor.\]
For each tuple $a = (a_{i,j})$ such that $0 \leq a_{i,j} \leq \lfloor\frac{e_i}{2}\rfloor$, let
\[\scrA_0^a := \{x \in \scrA_0 \mid \omega_{x,i,j} \simeq (\varepsilon_i^{a_{i,j}}) \oplus (\varepsilon_i^{e_i - a_{i,j}}) \text{  for all } i, j\}.\]
Then
\[\scrA_0 = \coprod_a \scrA_0^a\] 
is the KR stratification of $\scrA_0$.
The stratum $\scrA_0^a$ is locally closed subvariety of $\scrA_0$ of dimension $\sum_{i,j} (e_i - 2 a_{i,j})$ with closure relation given by
\[a \leq a' \text{ if and only if } a_{i,j} \geq a'_{i,j} \text{ for all } i,j.\]

Now consider the tuple $a=(a_{i,j})$ such that $x \in \scrA_0^a(k)$, then $a_{i,j}$ is the maximal integer such that
\[\omega_{i,j} \subset \varepsilon_i^{a_{i,j}} \calH_{i,j}.\]
%Let $(\calH_x, C_x, D_x)$ be the $F$-zip associated to $x$, with corresponding standard $F$-zip $(\Lambda_0, C_0, D_0)$.
%Fix an isomorphism $\calH_{i,j} \simeq \Lambda_{i,j} = k[\ve_i]^2$ sends 
If we write $M_{i,j} = \calH_{i,j} / \varepsilon_i \calH_{i,j}$, the Verschiebung map $\calH_{i,j+1} \to \calH_{i,j}$ induces a morphism 
\[V_{i,j}: M_{i,j+1} \to  \varepsilon_i^{a_{i,j}} \calH_{i,j} /  \varepsilon_i^{a_{i,j}+1} \calH_{i,j} \simeq M_{i,j}.\]
Conversely, we first assume $a_{i,j} < e_i/2$, the Frobenius map $\calH_{i,j} \to \calH_{i,j+1}$ induces a morphism
\[F_{i,j}: M_{i,j} \simeq \varepsilon_i^{a_{i,j}} \calH_{i,j} /  \varepsilon_i^{a_{i,j}+1} \calH_{i,j} \to \varepsilon_i^{2a_{i,j}} \calH_{i,j+1} /  \varepsilon_i^{2a_{i,j}+1} \calH_{i,j+1} \simeq M_{i,j+1}.\]
For $a_{i,j} = e_i/2$, let $F_{i,j} = 0$, then we have the tuple 
\[(M := \bigoplus_{i,j} M_{i,j}, F:= \bigoplus_{i,j} F_{i,j}, V:= \bigoplus_{i,j} V_{i,j})\]
such that $M$ is a $k$-vector space with semi-linear morphisms $F,V$ satisfying the equation
\[\Image(F) = \Ker(V),\quad \Ker(F) = \Image(V).\]
The type of the $F$-zip $(M,\Ker(V), \Ker(F),\varphi_\bullet)$ is determined by the dimension of $\Ker(V)$, which is equal to the number of the indices $(i,j)$ such that $a_{i,j} \neq e_i/2$. 
More explicitly, consider the function
\[ \delta(u) = \begin{cases} 0, & u = 0; \\ 1, & u \neq 0. \end{cases}\]
For each tuple $a=(a_{i,j})$ corresponding to a KR stratum $\scrA_0^a$,
 set \[t_a := \sum_{i,j} \delta(e_i - 2 a_{i,j}).\] 
 There are $t_a+1$ EKOR strata contained in $\scrA_0^a$.
The EKOR type of a point $x \in \scrA_0^a(k)$ is given by an integer $0 \leq t_x \leq t_a$, 
and the EKOR stratum containing $x$ has dimension $\dim(\scrA_0^a) - t_x$.

\end{example}

\subsection{Global construction of EKOR stratification}\label{subsec EKOR gl}

We will show that the closure relation of EKOR strata is given by the partially ordered set $({}^K\Adm(\mu), \leq_{K,\sigma})$.
All results and detailed definitions come from \cite{ShenYuZhang2021}*{\S 4} with exactly the same proofs (note that thanks to the recent works \cite{AGLR, GL} the perfection of the geometric special fiber $M^{\loc}$ of the local model here can be embedded into the associated Witt vector affine flag variety, similar to the tamely ramified case used in \cite{ShenYuZhang2021}). 
So we omit all proofs in this subsection and refer to  loc. cit. for detailed arguments.

Keep the notations as in the last subsection.
Let $C(\calG,\mu)$ be the index set of central leaves in \cite{ShenYuZhang2021}*{\S 1}. 
There is a prestack $\Sht_{\mu,K}^{\loc}$ over $k = \overline{\bbF}_p$ classifying $\calG$-Shtukas of type $\mu$,
whose $k$-points are given by
\[\Sht_{\mu,K}^{\loc}(k) = C(\calG, \mu).\]	
There is a prestack $\Sht_{\mu,K}^{\loc(\infty,1)}$ over $k$, parameterizing the so-called $(\infty,1)$-restricted local Shtukas.
For sufficiently large integer $m$, 
there is an algebraic stack $\Sht_{\mu,K}^{\loc(m,1)}$ over $k$, parameterizing the so-called $(m,1)$-restricted local Shtukas,
such that (see \cite{ShenYuZhang2021}*{Lemma 4.2.4})
\[|\Sht_{\mu,K}^{\loc(\infty,1)}| \simeq |\Sht_{\mu,K}^{\loc(m,1)}| \simeq {}^K\Adm(\mu).\]
We also have natural maps (which are perfectly smooth by \cite{ShenYuZhang2021}*{Proposition 4.2.5})
\[\Sht_{\mu,K}^{\loc} \to \Sht_{\mu,K}^{\loc(\infty,1)} \to \Sht_{\mu,K}^{\loc(m,1)} \to [\calG_0 \backslash M^{\loc}].\]

Consider the perfection $\scrA_0^{\pf} = \varprojlim_\sigma \scrA_0$ of $\scrA_0$. Then the
same proof of \cite{ShenYuZhang2021}*{Proposition 4.4.1} shows that there exists a morphism of prestacks
\[\scrA_0^{\pf} \to \Sht_{\mu,K}^{\loc}.\]
Composing this morphism with the natural morphism $\Sht_{\mu,K}^{\loc} \to \Sht_{\mu,K}^{\loc(m,1)}$, 
we get a morphism of stacks
\[v_K: \scrA_0^{\pf} \to \Sht_{\mu,K}^{\loc(m,1)}.\]
Recall that the local model diagram gives the morphism of stacks
\[\lambda_K: \scrA_0^{\pf} \to [\calG_0 \backslash M^{\loc}],\]
which is perfectly smooth.
The same proof as \cite{ShenYuZhang2021}*{Theorem 4.4.3} (there is a small gap in the last step of the proof concerning the involved diagram, which is inherited from the corresponding place of the work of Xiao-Zhu; but a small modification without the commutativity of that diagram will make the argument still work, cf. \cite{SYZ}.) shows
\begin{theorem}
	The following diagram commutes:
	\[\xymatrix{
	\scrA_0^{\pf} \ar[r]^{v_K} \ar[rd]_{\lambda_K} & \Sht_{\mu,K}^{\loc(m,1)} \ar[d]^{\pi_K} \\
	& {[\calG_0 \backslash M^{\loc}]}
	}\]
	Moreover, $v_K$ is perfectly smooth.
\end{theorem}

Consider the morphism of stacks
\[v_K: \scrA_0^{\pf} \to \Sht_{\mu,K}^{\loc(m,1)}.\]
We know that $|\Sht_{\mu,K}^{\loc(m,1)}| \simeq {}^K\Adm(\mu)$,
the fibers of $v_K$ are then the EKOR strata of $\scrA_0^{\pf}$.
As we have the identification of underlying topological spaces
\[|\scrA_0^{\pf}| = |\scrA_{0}|.\]
The perfect smoothness of $v_K$ shows that
\begin{corollary}
	For any $x \in {}^K\Adm(\mu)$, 
	the Zariski closure of the EKOR stratum $\scrA_{0}^x$ is given by
	\[\overline{\scrA_{0}^x} = \coprod_{x' \leq_{K,\sigma} x} \scrA_{0}^{x'}.\]
\end{corollary}

\subsection{Non-emptiness of EKOR strata}\label{subsec nonempty}

We recall the proof of non-emptiness of EKOR strata following \cite{HeZhou2020}.
Our situation is slightly different from loc. cit. as the groups there are required to be tamely ramified at $p$. Nevertheless their method equally applies to our case.

Let $\tau_\mu$ be the minimal element in $\Adm(\mu)$ with respect to the Bruhat order.
Let $\breve{I}$ be a fixed Iwahori subgroup of $\breve{G}:=G(\breve{\bbQ}_p)$.
For simplicity,
we may assume that the lattice $\scrL$ is determined by a finite index set $J = \{0,\dots,m\}$.

Let $B(G)$ be the set of $\sigma$-conjugacy classes in $\breve{G}$.
Recall that the Kottwitz set (the index set of Newton stratification of $\scrA_0$) is
\[B(G,\mu) = \{[b] \in B(G) \mid \kappa([b]) = \mu^\#, \nu([b]) \leq \overline{\mu}\}.\]
There is a unique basic element $[b_0]\in B(G,\mu)$.
By \cite{KisinPeraShin2022}*{Theorem 1}, the basic locus of $\scrA_0$  is non-empty.
Let $x = (A_j, \lambda_j, \alpha)_{j\in J} \in \scrA(k)$ be a point in the basic locus.
Let $D_j$ be the Dieudonn\'e module of $A_j[p^\infty]$ and  $N$ the common rational Dieudonn\'e module.
Then $D_j$ form a lattice chain inside $N$.
The Frobenius gives $\delta \in \breve{G} $ such that $[\delta] = [b_0] \in B(G,\mu)$.

For any $w \in W_K \backslash \wt{W} / W_K$ and $b \in \breve{G}$,
The affine Deligne-Lusztig variety is defined as
\[X_{K,w}(b) = \{g \in \breve{G} / \breve{K} \mid g^{-1} b \sigma(g) \in \breve{K} w \breve{K}\}.\]
If $\breve{K} = \breve{I}$,
we simply write the corresponding affine Deligne-Lusztig variety as $X_w(b)$. For an element $w\in \wt{W}$, let $\dot{w}$ be a representative of it in $\breve{G}$.
The following lemma will be used in the proof of non-emptiness of EKOR strata.

\begin{lemma}\label{lemma non empty}
	\begin{enumerate}
		\item $\dot{\tau}_\mu$ is central.
		\item $X_{\tau_\mu}(\delta)$ is non-empty.
	\end{enumerate}
\end{lemma}
\begin{proof}
	Recall that we have the isomorphism
	\[\widetilde{W} \simeq W_a \rtimes \pi_1(G)_{\Gamma_0}, \]
	where $W_a$ is the affine Weyl group of $G$.
	Under this isomorphism, the minimal element $\tau_\mu$ corresponds to $(\id, \mu^\#) \in W_a \rtimes \pi_1(G)_{\Gamma_0}$.
	By the proof of \cite{PappasRapoport2008}*{Appendix, Lemma 14},
	such element lifts to the torus $\breve{T} \subset \breve{G}$.
	So $\tau_\mu$ is central in $\widetilde{W}$ and its lift is also central in $\breve{G}$.

	We have
	\[X_{\tau_\mu}(\delta) \neq \emptyset \iff \breve{I} \tau_\mu \breve{I} \cap [\delta] \neq \emptyset\]
	by definition.
	% It is known that there is a surjective map \cite{ShenYuZhang2021}*{Theorem 1.3.5}
	% \[\Adm(\mu)_{\sigma-\text{str}} \twoheadrightarrow B(G,\mu).\]
	% For basic $[\delta]$, there is only one straight element in the associated $\sigma$-conjugacy class in $\widetilde{W}$,
	% which is just the element $\tau_\mu$.
	As $\tau_\mu$ is a $\sigma$-straight element in $\widetilde{W}$,
	$\breve{I} \tau_\mu \breve{I}$ lies in a single $\sigma$-conjugacy class of $\breve{G}$ by \cite{HeRapoport2017}*{Theorem 5.1.(a)}.
	Moreover, such conjugacy class is given by $[\delta]$.
	This shows the non-emptiness of $X_{\tau_\mu}(\delta)$.
\end{proof}

\begin{proposition}\label{prop non-empty EKOR}
  We have $\scrA_{0}^x \neq \emptyset$ for all $x \in {}^K\Adm(\mu)$
\end{proposition}
\begin{proof}
  By Lemma \ref{lemma non empty} (2),
  there is an element $g \in X_{\tau_\mu}(\delta)$.
  Then $g^{-1} \delta \sigma(g) \in \breve{I} \dot{\tau}_\mu \breve{I}$.
 Lemma \ref{lemma non empty} (1) shows that $\delta \sigma(g) \in g \dot{\tau}_\mu \breve{I}$.
  This gives
  \[p g D_j \subset \delta \sigma(g) D_j = \dot{\tau}_\mu g D_j \subset g D_j,\]
  which essentially shows that
  \[\Frob (g D_j) = \delta \sigma(g) D_j \subset g D_j.\]
  Thus $g D_j$ corresponds to a $p$-divisible group which is isogenous to $A_j[p^\infty]$,
  so we get an abelian variety $g A_j$.
  The polarization $\lambda_j$ and $\calO_B$-action extends naturally to $g A_j$.
  The prime to $p$ level structure $\alpha$ also extends to $g A_j$;
  we thus obtain a triple $(g A_j, g \lambda_j, \alpha)$.
  This triple gives a $k$-point $g x \in \scrA_I(k)$.
  By construction,  $g x$ must live in the minimal EKOR stratum.
  Then by closure relation, we get the non-emptiness of EKOR strata.
\end{proof}

\textbf{Conflicts of interest}: none.
 
\textbf{Financial Support}:  the first author was partially supported by the National Key R$\&$D Program of China 2020YFA0712600, the CAS Project for Young Scientists in Basic Research, Grant No. YSBR-033, and the NSFC grant No. 12288201.

\end{document}